\documentclass[3p,times]{elsarticle}

\usepackage{ecrc}

\volume{00}

\firstpage{1}

\journalname{Journal of Complexity}

\runauth{C.Pang, X.Wang and Y.Wu}

\jid{J COMPLEXITY}

\jnltitlelogo{journal of complexity}

\usepackage{amssymb}
\usepackage{mathrsfs}
\usepackage{amsfonts}
\usepackage{graphicx}
\usepackage{colortbl,dcolumn}
\usepackage{amsmath}
\usepackage{psfrag}
\usepackage{booktabs}
\usepackage{appendix}
\usepackage{amsthm}
\usepackage{subfigure}
\numberwithin{equation}{section}

\usepackage[figuresright]{rotating}

\newcommand{\R}{\mathbb{R}}
\newcommand{\N}{\mathbb{N}}
\newcommand{\E}{\mathbb{E}}
\renewcommand{\P}{\mathbb{P}}
\newcommand{\tr}{\operatorname{trace}}

\newcommand{\dd}{\text{d}}

\newtheorem{theorem}{Theorem}[section]
\newtheorem{definition}[theorem]{Definition}
\newtheorem{lemma}[theorem]{Lemma}

\newtheorem{corollary}[theorem]{Corollary}
\newtheorem{remark}[theorem]{Remark}

\newtheorem{assumption}[theorem]{Assumption}

\begin{document}

\begin{frontmatter}



\dochead{}

\title{Linear implicit approximations of 
invariant measures of semi-linear SDEs with non-globally Lipschitz coefficients\tnoteref{t1}
}


\tnotetext[t1]{This work was supported by Natural Science Foundation of China (12071488, 12371417, 11971488)
                and Natural Science Foundation of Hunan Province (2020JJ2040). }

\author[label1]{Chenxu Pang}
\ead{c.x.pang@csu.edu.cn}
\author[label1]{Xiaojie Wang}
\ead{x.j.wang7@csu.edu.cn}
\author[label2]{Yue Wu\corref{cor1}}
\ead{yue.wu@strath.ac.uk}

\address[label1]{School of Mathematics and Statistics, HNP-LAMA, Central South University, Changsha, Hunan, P. R. China}
\address[label2]{Department of Mathematics and Statistics, University of Strathclyde, Glasgow G1 1XH, UK}

\cortext[cor1]{Corresponding author}
\begin{abstract}
        This article
         investigates the weak approximation towards the invariant measure of semi-linear stochastic differential equations (SDEs) under non-globally Lipschitz coefficients. For this purpose, 
         we propose a linear-theta-projected Euler (LTPE) scheme, which also admits an invariant measure, to handle the potential influence of the linear stiffness.
         Under certain assumptions, both the SDE and the corresponding LTPE method are shown to converge exponentially to the underlying invariant measures, respectively. Moreover, with time-independent regularity estimates for the corresponding Kolmogorov equation, the weak error between the numerical invariant measure and the original one can be guaranteed with convergence of order one. 
         In terms of computational complexity, 
         the proposed ergodicity preserving scheme with 
         the nonlinearity explicitly treated has 
         a significant advantage over the ergodicity preserving implicit Euler method in the literature.
         Numerical experiments are provided to verify our theoretical findings.

\end{abstract}

\begin{keyword}
stochastic differential equations \sep  invariant measure \sep
weak convergence \sep projected method \sep Kolmogorov equations
\MSC 60H35 \sep 37M25 \sep 65C30

\end{keyword}


\end{frontmatter}


\tableofcontents
\section{Introduction}	\label{section: introduction}
The primary objective of this paper is to study the invariant measures of semi-linear
stochastic differential equations (SDEs) with multiplicative noise 
and their weak approximations.  Given the probability space $(\Omega,\mathcal{F}, \mathbb{P})$, 
we consider the following $\mathbb{R}^d$-valued semi-linear SDEs of It\^o type:
\begin{equation} \label{eq:semi-linear-SODE}
\begin{split}
\left\{
    \begin{array}{ll}
    \dd X_t  = A X_t + f ( X_t ) \, \dd t + g ( X_t ) \,\dd W_t,
    \quad
    t \in (0, T],
    \\
   X_0 = x_0,
 \end{array}\right.
 \end{split}
\end{equation}
where $A \in \mathbb{R}^{d \times d}$ represents a negative definite  matrix, $f \colon \mathbb{R}^d \rightarrow \mathbb{R}^d$ is the drift coefficient function,
$g \colon \mathbb{R}^d \rightarrow \mathbb{R}^{d \times m}$ is the diffusion coefficient function, and $W_{\cdot} = \left(W_{1,\cdot}, \dots, W_{m,\cdot} 
\right)^{T} :[0, T] \times \Omega \rightarrow \mathbb{R}^{m}$ denotes the $\mathbb{R}^{m}$-valued standard Brownian motion with respect to $\left\{\mathcal{F}_{t}\right\}_{t \in [0, T]}$.
Moreover,  the initial data $x_{0}: \Omega \rightarrow \mathbb{R}^{d}$ is assumed to be $\mathcal{F}_{0}$-measurable.
This form covers a broad class of SDEs which are used to model real applications,  for instance, the stochastic Ginzburg–Landau equation (see \eqref{equation:ginzburg-landau-model}), the mean-reverting model (see \eqref{eq:mean-reverting-model} or \cite{liu2023backward,2013Convergence}) and space discretization of stochastic partial differential equations (SPDEs) (see \eqref{eq:SPDE-SODE-system} or \cite{wang2023mean,liu2021strong}).

In this paper, we pay particular attention to a class of SDEs that, under certain conditions, converge exponentially to a unique invariant measure $\pi$.
Evaluating the expectation of some function $\varphi$ with respect to that invariant measure $\pi$ is of great interest in mathematical biology, physics and Bayesian statistics:
\begin{equation} \label{introduction:invariant-pi}
    \pi(\varphi):=\int_{\mathbb{R}^d} \varphi(x) \pi(\dd x)=\lim _{t \rightarrow \infty} \mathbb{E}\left[\varphi\left(X_t\right)\right].
\end{equation}
%
Generally speaking, it is not easy to obtain either the analytical solutions of SDEs or the explicit expression of the invariant measure.  The study of the numerical approximations of $\pi$ therefore receives increased attention.  Previous research in this field typically focuses on stochastic differential equations (SDEs) characterized by coefficients that exhibit global Lipschitz continuity \cite{milstein2013stochastic}. Such a strong condition is however rarely satisfied by SDEs from applications. On the other hand, conventional numerical tools lose their powers when attempting to simulate SDEs under relaxed conditions. For example, as claimed in \cite{hutzenthaler2011strong,mattingly2002ergodicity}, for a large class of SDEs with super-linear growth coefficients, the widely-used Euler-Maruyama scheme leads to divergent numerical approximations in both finite and infinite time intervals.
A natural question thus arises as to how to design the numerical scheme of the SDE \eqref{eq:semi-linear-SODE} under a stiff condition caused by the linear operator in order to well approximate its invariant measure $\pi$ and perform the error analysis.

Recent years have seen a proper growth of the literature on this topic, and it is worth mentioning that a majority of existing works analyze numerical approximations of invariant measures from SDEs via strong approximation error bounds (see \cite{li2019explicit,liu2023backward,fang2018adaptive,mattingly2002ergodicity,majka2020nonasymptotic,neufeld2022non}). The direct study of weak approximation errors (see \cite{brehier2014approximation,brehier2020approximation,chen2017approximation,chen2022stochastic}), which hold particular relevance in fields like financial engineering and statistics, is still in its early stages. In \cite{chen2022stochastic}, the authors analyzed the backward Euler method of SDEs with piecewise continuous arguments (PCAs), where the drift is dissipative and the diffusion is globally Lipschitz, and recovered a time-independent convergence of order one. The author in \cite{brehier2020approximation} studied the tamed Euler scheme for ergodic SDEs with one-sided Lipschitz continuous drift
coefficient and additive noise, and gave a moment bound that still depends on terminal time. 
We also mention that the authors in \cite{abdulle2014high}  provided new suﬃcient conditions for a numerical method to approximate with high order of accuracy of the invariant measure of an ergodic SDE, independently of the weak order of accuracy of the method.

Each method exhibits drawbacks when approximating \eqref{introduction:invariant-pi} weakly. Implicit methods by their nature have better stability but at a price of escalated complexity; explicit methods such as the tamed methods (see \cite{hutzenthaler2012strong,wang2013tamed}) on the other hand may not preserve the long time property numerically since the taming factor has no positive lower bound. 
Even though the explicit projected method \cite{szpruch2018} does keep the asymptotic stability, it usually faces a severe stepsize restriction due to stability issues from solving stiff linear systems; to apply the truncated methods \cite{li2019explicit} to approximate the invariant distribution, one has to construct a strictly increasing function to control the growth of both drift and diffusion and to find its inverse version.
Besides, the weak error analysis of such schemes is,  to the best of our knowledge, still an open problem. We, therefore, aim to propose a family of linear-implicit methods that not only address the challenges posed by stiff systems but also preserve ergodicity and achieve weak convergence towards the invariant measure admitted by SDEs \eqref{eq:semi-linear-SODE}.

More formally, 
our scheme, called the linear-theta-implicit-projected Euler (LTPE) method, with a method parameter $\theta \in [0,1]$ on a uniform timestep size $h$ is given as follows,
\begin{equation} \label{introduction:LTPE-scheme}
\begin{aligned}
Y_{n+1} -  \theta A Y_{n+1} h 
&= \mathscr{P}(Y_{n}) + (1-\theta) A \mathscr{P}(Y_{n}) h 
+ f\big(\mathscr{P}(Y_{n})\big) h 
+ g\big(\mathscr{P}(Y_{n})\big) \Delta W_{n}, 
\quad Y_{0}=X_{0},
\end{aligned}
\end{equation}
where $\Delta W_{n}:=W_{t_{n+1}}-W_{t_{n}}$, $n \in\{0,1,2, \ldots, N-1\}$, $N \in \mathbb{N}$,
and $\mathscr{P}: \mathbb{R}^{d} \rightarrow \mathbb{R}^{d}$ is the projected operator denoted as
\begin{equation}\label{eq:projected-operator}
\mathscr{P}(x)
:=\min\left\{
1, h^{-\frac{1}{2\gamma}} \|x \|^{-1}
\right\} x, 
\quad \forall x \in \mathbb{R}^{d},
\end{equation}
with $\gamma$ being determined in Assumption \ref{assumption:growth-condition-of-frechet-derivatives-of-drift-and-diffusion} later.

We point out that the scheme above can be derived from the stochastic theta methods \cite{mao2013strong,wang2020mean} used to deal with different models.
Also, note that the parameter $\theta$ is pre-determined. 
Where there is a stiff system, we are able to treat the linear operator $A$ implicitly (i.e. $\theta=1$) without sacrificing numerical efficiency. 
And if one is working with the non-stiff system, using the explicit numerical scheme (i.e. $\theta=0$) would be more appropriate. In addition, we follow the projected technique, previously used in  \cite{beyn2016stochastic,beyn2017stochastic} for SDEs in finite time interval, to prevent the nonlinear drift and diffusion from producing extraordinary large values. Under certain conditions, 
for $ \forall \zeta \in L^{8\gamma+2}(\Omega,\mathbb{R}^{d})$, where $\gamma$ is given by Assumption \ref{assumption:growth-condition-of-frechet-derivatives-of-drift-and-diffusion}, the projected process $\mathscr{P}(x)$ converges strongly to the original random variable $\zeta$ of order 2 (see Lemma \ref{lemma:error-estimate-between-x-and-projected-x} or \cite{beyn2017stochastic}), i.e.
\begin{equation} \label{introduction:error-estimate-of-the-projection}
\|
\mathscr{P}(\zeta) - \zeta 
\|_{L^{2}(\Omega, \mathbb{R}^{d})} 
\leq Ch^{2}.
\end{equation}
Compared with the truncated method in \cite{li2019explicit}, the implementation of the LTPE method in \eqref{introduction:LTPE-scheme} is more straightforward, where the projected operator we have chosen depends only on the growth of the drift and diffusion. Besides, when facing with linear-stiff systems, our method with $\theta = 1$ may not suffer from too strict stepsize restriction.


To show the main result in Theorem \ref{thm:main}, the derivations of the whole paper are organised in the following way: under Assumption \ref{assumption:one-side-Lipschitz-condition-for-linear-operator}-\ref{assumption:growth-condition-of-frechet-derivatives-of-drift-and-diffusion}, which can be regarded as a kind of dissipative condition, we follow \cite{da2006introduction} to present the existence and uniqueness of the invariant measures of both SDEs \eqref{eq:semi-linear-SODE} and the LTPE scheme \eqref{introduction:LTPE-scheme}, respectively in Theorem \ref{theorem:invariant-measure-of-sode} and Theorem \ref{theorem:invariant-measure-of-LTPE-method}; the main result regarding weak error analysis, presented in Theorem \ref{theorem:time-independent-weak-error-analysis}, is derived based on the associated Kolmogorov equation \eqref{equation:kolmogorov-equation} of SDE \eqref{eq:semi-linear-SODE}.
However,
one may confront two main challenges. The first one is to get a couple of priori estimates that are independent of time and stepsize, including the uniform moment bounds of the LTPE method \eqref{introduction:LTPE-scheme} and the time-independent regularity estimates of the Kolmogorov equation.  Another one is the implicitness and discontinuity of the proposed LTPE method \eqref{introduction:LTPE-scheme}, which results in further difficulties in handling the weak error via the kolmogorov equation.
Different techniques are used to circumvent these obstacles. Discretization strategy based on the binomial theorem is adopted to obtain the uniform moment bounds of the LTPE scheme (see Lemma \ref{lemma:uniform-moments-bound-of-the-LTPE-method}), and we make use of the It\^o formula and the variational approach to obtain the time-independent regularity estimates of the Kolmogorov equation (see Lemma \ref{lemma:estimate-of-u-and-its-derivatives} and Corollary \ref{lemma:contractivity-of-u}). To deal with possible implicitness and discontinuity of the LTPE scheme \eqref{introduction:LTPE-scheme}, we introduce its continuous-version $\{\mathbb{Z}^{n}(t)\}_{t\in  [t_{n}, t_{n+1}]}$ with $n\in \{0,1,\dots,N-1 \}$, $N\in \mathbb{N}$ as 
\begin{equation} \label{intro:continuous-version-of-the-numerical-scheme}
\begin{split}
\left\{
    \begin{array}{ll}
    \mathbb{Z}^{n}(t) 
    = 
    \mathbb{Z}^{n}(t_{n}) 
    + F\big(\mathscr{P}(Y_{n})\big)(t-t_{n}) 
    + g\big(\mathscr{P}(Y_{n})\big)(W_{t}-W_{t_{n}}) ,
    \\
   \mathbb{Z}^{n}(t_{n})
   := \mathscr{P}(Y_{n}) - \theta A  \mathscr{P}(Y_{n}) h,
 \end{array}\right.
 \end{split}
\end{equation}
where $F(x):=Ax+f(x), \forall x \in \mathbb{R}^{d}$. 
It can be easily observed that $\mathbb{Z}^{n}(t_{n+1})  = Y_{n+1} - \theta A  Y_{n+1} h$.
%
In order to estimate the numerical approximation error of invariant measure, we separate the weak error $\big| \mathbb{E}\left[\varphi(Y_{N}) \right] - \mathbb{E}\left[\varphi(X_{T}) \right]\big|$, i.e., $\big| \mathbb{E}\left[u(T, x_{0}) \right] - \mathbb{E}\left[u(0, Y_{N}) \right] \big|$ based on the associated Kolmogorov equation (see \eqref{equation:kolmogorov-equation} or \cite[Chapter 1]{cerrai2001second}),  into three parts,
\begin{equation}
\begin{aligned}
\big| 
\mathbb{E}\left[u(T, x_{0}) \right] 
- \mathbb{E}\left[u(0, Y^{x_{0}}_{N}) \right]
\big| 
&\leq 
\underbrace{
\big| 
\mathbb{E}\left[u(0, Y_{N}) \right] 
- \mathbb{E}\left[u(0, Z_{N}) \right]
\big|
}_{:= \text{Error}_{1}} 
+
\underbrace{
\big|  
\mathbb{E}\left[u(T, Z_{0}) \right] 
- \mathbb{E}\left[u(T, x_{0}) \right] 
\big|
}_{:= \text{Error}_{2}} \\
&\quad + \underbrace{
\big| \mathbb{E}\left[u(0, Z_{N}) \right] 
- \mathbb{E}\left[u(T, Z_{0}) \right]\big|
}_{:= \text{Error}_{3}},
\end{aligned}
\end{equation}
where, for short, we denote $Z_{n}:=Y_{n}-\theta A Y_{n}h$.
Thanks to the fact that $Z_{n+1}=\mathbb{Z}^{n}(t_{n+1})$ and the time-independent regularity estimates of the Kolmogorov equation, one can treat $\text{Error}_{1}$ and $\text{Error}_{2}$ directly and get $\max\{\text{Error}_{1} , \text{Error}_{2}\} = \mathcal{O}(h)$. For $\text{Error}_{3}$, we take full advantage of 
\eqref{intro:continuous-version-of-the-numerical-scheme} and show further decomposition as
\begin{equation} \label{introduction:decomposition-of-error-3}
\begin{aligned}
\text{Error}_{3} 
&\leq 
\left|
\sum_{n=0}^{N-1}   
\mathbb{E}\left[
u\big(T-t_{n}, \mathbb{Z}^{n}(t_{n})\big) 
\right] 
- \mathbb{E}\big[u(T-t_{n}, Z_{n}) \big] 
\right|\\
&\quad + 
\left|
\sum_{n=0}^{N-1}  
\mathbb{E}\left[
u\big(T-t_{n+1}, \mathbb{Z}^{n}(t_{n+1})\big) 
\right] 
- \mathbb{E}\left[
u\big(T-t_{n}, \mathbb{Z}^{n}(t_{n})\big) 
\right] 
\right|.
\end{aligned}
\end{equation}
The first term on the right hand side of \eqref{introduction:decomposition-of-error-3} is $\mathcal{O}(h)$ due to the regularity estimates of $u(t, \cdot)$ and \eqref{introduction:error-estimate-of-the-projection}; the second one, based on the Kolmogorov equation and the It\^o formula, can also be proved to be $\mathcal{O}(h)$ (see more details in the proof of Theorem \ref{theorem:time-independent-weak-error-analysis}).
Hence, we obtain the the uniform weak error between the invariant measures, admitted by SDE \eqref{eq:semi-linear-SODE} and the LTPE method \eqref{introduction:LTPE-scheme},  of order one eventually.

We summarize our main contributions:
\begin{itemize}
    \item A family of linear implicit numerical methods, capable of dealing with stiff linear systems and inheriting invariant measures, is presented.
    \item Time-independent weak convergence between two invariant measures inherited by SDE \eqref{eq:semi-linear-SODE} and LTPE scheme \eqref{introduction:LTPE-scheme}, respectively, is established under non-globally Lipschitz coefficients.
\end{itemize}

Some numerical tests to illustrate our findings in Section 6. Finally, the Appendix contains the detailed proof of auxiliary lemmas.
\section{Settings and main result}
Throughout this paper, we use $\N$ to denote the set of all positive integers
and let $ d,m \in \N$, $ T \in (0, \infty) $ be given. Let $\| \cdot \|$ and $ \langle \cdot, \cdot \rangle $ denote the Euclidean
norm and the inner product of vectors in $\R^d$, respectively. We use $\max\{a,b\}$ and $\min\{a,b\}$ for the maximum and  minimum values of between $a$ and $b$ respectively, and sometimes we also use a simplified notation $a\land b$ for $\min\{a,b\}$. Adopting the same notation as the vector norm, we denote $\|M\| : =\sqrt{\tr(M^{T}M)}$ as the trace norm of a matrix $M \in \R^{d \times m}$, where $M^T$ represents the transpose of a matrix $M$.
Given a filtered probability space $ \left( \Omega, \mathcal{ F }, \{ \mathcal{ F }_t \}_{ t \in [0,T] }, \P 
\right) $, we use $\E$ to mean the expectation and $L^{r} (\Omega, \R^d ), r \geq 1 $, to denote 
the family of $\R^d$-valued random variables $\xi$ satisfying   $\E[ \|\xi \|^{r}]<\infty$.
The diffusion coefficient function $g \colon \mathbb{R}^d \rightarrow \mathbb{R}^{d \times m}$ is frequently written as $g = (g_{i,j})_{d \times m} =  (g_1, g_2,..., g_m)$ for $g_{i,j} \colon \mathbb{R}^d \rightarrow \mathbb{R}$
and $g_j \colon \mathbb{R}^d \rightarrow \mathbb{R}^{d }, i \in \{ 1, 2,..., d \}, j \in \{ 1, 2,..., m \} $.
Moreover, 
we introduce a new notation $X^{x}_{t}$ for $t\in [0,T]$ denoting the solution of SDE \eqref{eq:semi-linear-SODE} satisfying the initial condition $X^{x}_{0} = X_{0}=x$. Also, let $Y^{x}_{n}$, $n\in \{0,1, \dots, N \}$, $N\in \mathbb{N}$, be an approximation of the solution of SDE \eqref{eq:semi-linear-SODE} with the initial point $Y^{x}_{0}=x$.
In addition, denote by $C_{b}(\mathbb{R}^{d})$ 
the Banach space of all uniformly continuous and bounded mappings
$\phi: \mathbb{R}^{d} \rightarrow \mathbb{R}$ endowed with the norm $\|\phi \|_{0} = \sup_{x\in \mathbb{R}^{d}} |\phi(x) |$.
For the vector-valued function $\textbf{u}: \mathbb{R}^{d} \rightarrow \mathbb{R}^{{\ell}}$,
$\textbf{u} = (
u_{(1)}, \dots, u_{({d})}
)$, its first order partial derivative is considered as the Jacobian matrix as
\begin{equation}
D \mathbf{u}=
\left(\begin{array}{ccc}
\frac{\partial u_{(1)}}{\partial x_1} & \cdots & \frac{\partial u_{(1)}}{\partial x_d} \\
\vdots & \ddots & \vdots \\
\frac{\partial u_{(\ell)}}{\partial x_1} & \cdots & \frac{\partial u_{(\ell)}}{\partial x_d}
\end{array}\right)_{\ell \times d}.
\end{equation}
For any $v_{1} \in \mathbb{R}^{d}$, one knows $D (\mathbf{u} ) v_{1} \in \mathbb{R}^{\ell}$ and
one can define $D^{2} \mathbf{u}(v_{1},v_{2})$ as 
\begin{equation}
D^{2} \mathbf{u}(v_{1},v_{2})
:= D \big( D (\mathbf{u} )v_{1} \big) v_{2}, \quad \forall v_{1}, v_{2} \in \mathbb{R}^{d}
\end{equation}
In the same manner, one can define
\begin{equation}
D^{3} \mathbf{u}(v_{1},v_{2},v_{3})
:= D\Big( D \big( D (\mathbf{u} )v_{1} \big) v_{2} \Big)v_{3}, \quad \forall v_{1}, v_{2}, v_{3} \in \mathbb{R}^{d}
\end{equation}
and for any integer $k \geq 3$ the $k$-th order partial derivatives of the function  $\textbf{u}$ can be defined recursively.
Given the Banach spaces $\mathbb{X}$ and $\mathbb{Y}$, we denote by $L(\mathbb{X}, \mathbb{Y})$ the Banach space of bounded linear operators from  $\mathbb{X}$ into $\mathbb{Y}$.
Then the partial derivatives of the function  $\textbf{u}$ can be also regarded as the operators
\begin{equation}
D \textbf{u}(\cdot)(\cdot): \mathbb{R}^d \to L(\mathbb{R}^d, \mathbb{R}^\ell),    
\end{equation}
\begin{equation}
D^2 \textbf{u}(\cdot)(\cdot,\cdot): \mathbb{R}^d \to L(\mathbb{R}^d, L(\mathbb{R}^d,\mathbb{R}^d))\cong  L(\mathbb{R}^d\otimes\mathbb{R}^d, \mathbb{R}^\ell)
\end{equation}
and
\begin{equation}
D^3 \textbf{u}(\cdot)(\cdot,\cdot,\cdot): \mathbb{R}^d \to L(\mathbb{R}^d, L(\mathbb{R}^d,L(\mathbb{R}^d,\mathbb{R}^d)))\cong  L((\mathbb{R}^d)^{\otimes 3}, \mathbb{R}^\ell).
\end{equation}
We remark that the partial derivatives of the scalar valued function can be covered by the special case $\ell =1$.
For any $k\in \mathbb{N}$, let $C^{k}_{b}(\mathbb{R}^{d})$ be the subspace of $C_{b}(\mathbb{R}^{d})$ consisting of all functions with bounded partial derivatives $D^{i}\phi(x)$, $1\leq i \leq k$, and with the norm $\|\phi \|_{k}:=\|\phi \|_{0}+ \sum_{i=1}^{k} \sup_{x\in \mathbb{R}^{d}} \|D^{i}\phi(x) \|$. Further, let $\textbf{1}_{\mathbb{B}}$ be the indicative function of a set $\mathbb{B}$.
Denote $\frac{1}{0}:=\infty$.
To close this part, we let both $C$ and $C_{A}$ be the generic constant which are dependent of $T$ and the stepsize, but more specially, the notation $C_{A}$ further depends on the matrix $A$.

We present the following assumptions required to establish our main result.
\begin{assumption} \label{assumption:one-side-Lipschitz-condition-for-linear-operator}
Assume the matrix $A \in \mathbb{R}^{d\times d}$ is self-adjoint and negative definite. 
\end{assumption}
Assumption \ref{assumption:one-side-Lipschitz-condition-for-linear-operator} immediately implies that there exists a sequence of non-decreasing positive real numbers $\{ \lambda_{i} \}_{i=1}^{d}$ with
$0< \lambda_{1} \leq \lambda_{2} \leq \dots \leq \lambda_{d} < \infty$
and  an orthonormal basis $\{e_{i} \}_{i \in \{1,\dots,d\}}$ such that $Ae_{i}=\lambda_{i}e_{i}, i \in \{1,\dots,d\}$.
Moreover, one also obtains
\begin{equation} \label{equation:one-side-Lipschitz-conditoin-for-linear-operator}
\langle x-y , A(x-y) \rangle \leq -\lambda_{1}  \left\| x - y \right\| ^{2}, \quad \forall x, y \in \mathbb{R}^{d}.
\end{equation}
Setting $y=0$ leads to
\begin{equation} \label{equation:growth-conditoin-for-linear-operator}
\langle x , Ax \rangle \leq - \lambda_{1}   \left\| x \right\| ^{2}, \quad \forall x \in \mathbb{R}^{d}.
\end{equation}
\begin{assumption} \label{assumption:coercivity-condition-for-drift-and-diffusion}
(Coercivity condition) For some $p_{0} \in [1, \infty)$, there exists a constant $L_{1} \in \mathbb{R}$  such that,
\begin{equation}
\begin{aligned}
2 \langle x  , f ( x )  \rangle + (2p_{0} -1 ) \|g(x) \|^{2} \leq  L_{1} \left( 1 + \left\| x  \right\|^{2} \right) , \quad \forall x  \in \mathbb{R}^{d}.  \\
\end{aligned}
\end{equation}
\end{assumption}


\begin{assumption} \label{assumption:coupled-monotoncity-for-drift-and-diffusion}
(Coupled monotoncity condition)
For some $p_{1} \in (1, \infty)$, there exists a constant $L_{2}\in \mathbb{R}$  such that,
\begin{equation}
\begin{aligned}
2 \langle x - y , f ( x ) - f ( y ) \rangle +  (2p_{1} - 1)\left\| g ( x ) - g ( y ) \right\|^{2} \leq  L_{2}  \left\| x - y \right\| ^{2}, \quad \forall x, y \in \mathbb{R}^{d}.  
\end{aligned}
\end{equation}
\end{assumption}
Note that Assumption \ref{assumption:coupled-monotoncity-for-drift-and-diffusion} is equivalent to the following expression
\begin{equation}
\begin{aligned}
2 \langle x - y , f ( x ) - f ( y ) \rangle 
+  (2p_{1} - 1) \sum^{m}_{j=1}
\left\| g_{j} ( x ) - g_{j} ( y ) \right\|^{2} 
\leq  L_{2}  \left\| x - y \right\| ^{2}, 
\quad \forall x, y \in \mathbb{R}^{d}.  
\end{aligned}
\end{equation}
Thanks to Assumptions \ref{assumption:one-side-Lipschitz-condition-for-linear-operator}-\ref{assumption:coupled-monotoncity-for-drift-and-diffusion},
one obtains that
SDE \eqref{eq:semi-linear-SODE} possesses a unique solution with continuous sample paths. Moreover, we require that the coefficients $f$ and $g$ have continuous partial derivatives up to the third order. The corresponding assumption is presented as below.
\begin{assumption} \label{assumption:growth-condition-of-frechet-derivatives-of-drift-and-diffusion}
(Polynomial growth of drift and diffusion)
Assume that $f: \mathbb{R}^{d} \rightarrow \mathbb{R}^{d}$ and $g_{j}: \mathbb{R}^{d} \rightarrow \mathbb{R}^{d}$, $j\in \{1, \dots, m \}$, have all continuous derivatives up to order 3.
Then there exist some positive constant $\gamma \in [1,\infty)$ such that
\begin{equation}
\begin{aligned}
\left\|D^{3}f(x)( v_{1}, v_{2}, v_{3}) \right\| 
&\leq C  (1+\|x\|)^{\gamma-3} 
\|v_{1} \| \cdot \|v_{2} \| \cdot \|v_{3} \|,
\quad \forall x, v_{1}, v_{2}, v_{3}\in \mathbb{R}^{d}, \\  
\left\|D^{3}g_{j}(x) ( v_{1}, v_{2}, v_{3}) \right\|^{2}
&\leq C(1+\|x\|)^{\gamma-5} 
\|v_{1} \|^{2} \cdot \|v_{2} \|^{2} \cdot \|v_{3} \|^{2}, 
\quad \forall x, v_{1}, v_{2}, v_{3}\in \mathbb{R}^{d}.
\end{aligned}
\end{equation}
\end{assumption}
%
%
Assumption \ref{assumption:growth-condition-of-frechet-derivatives-of-drift-and-diffusion} is regarded  as a kind of polynomial growth conditions and
in proofs which follow we will need some implications of this assumption.
It follows immediately
that,
\begin{equation}
\begin{aligned}
\left\|
D^{2}f(x)  ( v_{1}, v_{2}) 
- D^{2}f(\tilde{x}) ( v_{1}, v_{2} ) 
\right\| 
&\leq 
C(1+\|x\|+\|\tilde{x}\|)^{\gamma-3} 
\|x-\tilde{x}\| 
\cdot \|v_{1} \| \cdot \|v_{2} \|, \
\forall x, \tilde{x}, v_{1}, v_{2}\in \mathbb{R}^{d},\\
\end{aligned}
\end{equation}
and
\begin{equation}
\left\|
D^{2}f(x) ( v_{1}, v_{2}) 
\right\| 
\leq C(1+\|x\|)^{\gamma-2}  
\|v_{1} \| \cdot \|v_{2} \|,
\quad \forall x, v_{1}, v_{2}\in \mathbb{R}^{d},
\end{equation}
which in turns gives
\begin{equation}
\begin{aligned}
\left\|
Df(x)  v_{1} - Df(\tilde{x}) v_{1}  
\right\| 
&\leq 
C(1+\|x\|+\|\tilde{x}\|)^{\gamma-2}  
\|x-\tilde{x}\| \cdot \|v_{1} \|, 
\quad \forall x, \tilde{x}, v_{1}\in \mathbb{R}^{d}, \\
\left\|Df(x)  v_{1}  \right\| 
&\leq C(1+\|x\|)^{\gamma-1} 
\|v_{1} \| ,
\quad \forall x, v_{1}\in \mathbb{R}^{d},
\end{aligned}
\end{equation}
and 
\begin{equation}\label{equation:growth-of-the-drift-f}
\begin{aligned}
\left\|f(x)  - f(\tilde{x}) \right\| 
&\leq C_{1}(1+\|x\|+\|\tilde{x}\|)^{\gamma-1} 
\|x-\tilde{x}\| , 
\quad \forall x, \tilde{x}\in \mathbb{R}^{d}, \\
\left\|f(x)   \right\| 
&\leq C_{2}(1+\|x\|)^{\gamma} ,
\quad \forall x\in \mathbb{R}^{d}.
\end{aligned}
\end{equation}
Following the same idea, Assumption \ref{assumption:growth-condition-of-frechet-derivatives-of-drift-and-diffusion} also ensures, for $j\in \{1, \dots,m \}$,
\begin{equation}
\left\|
D^{2}g_{j}(x) ( v_{1}, v_{2})
- D^{2}g_{j}(\tilde{x})( v_{1}, v_{2}) 
\right\|^{2} 
\leq C(1+\|x\|+\|\tilde{x}\|)^{\gamma-5} 
\|x-\tilde{x}\|^{2} 
\cdot \|v_{1} \|^{2} \cdot \|v_{2} \|^{2},\
\forall x, \tilde{x}, v_{1}, v_{2}\in \mathbb{R}^{d},\\
\end{equation}
and
\begin{equation}
\left\|
D^{2}g_{j}(x) ( v_{1}, v_{2}) 
\right\|^{2} 
\leq 
C(1+\|x\|)^{\gamma-3}  
\|v_{1} \|^{2} \cdot \|v_{2} \|^{2},
\quad \forall x, v_{1}, v_{2}\in \mathbb{R}^{d}.
\end{equation}
This in turns gives, for $j\in \{1,\dots,m \}$,
\begin{equation}
\begin{aligned}
\left\|
Dg_{j}(x)  v_{1} - Dg_{j}(\tilde{x}) v_{1}  
\right\|^{2} 
&\leq C(1+\|x\|+\|\tilde{x}\|)^{\gamma-3} 
\|x-\tilde{x}\|^{2} \cdot
\|v_{1} \|^{2}, 
\quad \forall x, \tilde{x}, v_{1}\in \mathbb{R}^{d}, \\
\left\|
Dg_{j}(x)  v_{1}  
\right\|^{2} 
&\leq C(1+\|x\|)^{\gamma-1}  
\|v_{1} \|^{2} ,
\quad \forall x, v_{1}\in \mathbb{R}^{d},
\end{aligned}
\end{equation}
and 
\begin{equation}
\begin{aligned}
\left\|
g_{j}(x)  - g_{j}(\tilde{x}) 
\right\|^{2} 
&\leq C(1+\|x\|+\|\tilde{x}\|)^{\gamma-1} 
\|x-\tilde{x}\|^{2}, 
\quad \forall x, \tilde{x}\in \mathbb{R}^{d}, \\
\left\|g_{j}(x)   \right\|^{2} 
&\leq 
C(1+\|x\|)^{\gamma+1} 
,\quad \forall x\in \mathbb{R}^{d}.
\end{aligned}
\end{equation}
We remark that Assumptions \ref{assumption:one-side-Lipschitz-condition-for-linear-operator}-\ref{assumption:growth-condition-of-frechet-derivatives-of-drift-and-diffusion} enable us to cover a broad class of SDEs with non-globally Lipschitz coefficients, which do not have closed-form solutions in general.

Now we are fully prepared to state the main result of this article as follows,

\begin{theorem} (Main result)\label{thm:main}
Let Assumptions \ref{assumption:one-side-Lipschitz-condition-for-linear-operator}-\ref{assumption:growth-condition-of-frechet-derivatives-of-drift-and-diffusion} hold with $p_{0} \geq \max\{4\gamma+1, 5\gamma-4\}$ and $2\lambda_{1} > \max\{L_{1}, L_{2}\}$ and consider SDE \eqref{eq:semi-linear-SODE}.
Given $p \in [1, p_{0}) \cap \mathbb{N}$ and method parameter $\theta \in [0,1]$,
let $h$ be the uniform timestep satisfying
\begin{equation}
h \in 
\left(
0, \min\left\{
\tfrac{1}{2(1-\theta)\lambda_{1}},
\tfrac{p_{0}-p}{(1-\theta)(2p_{0}-p-1)\lambda_{1}},
\tfrac{1}{(1-\theta)\lambda_{d}},
\tfrac{\kappa^{2}(2\lambda_{1}-L_{2})}
{(1- \theta)^{2}\lambda^{2}_{d}},
\tfrac{(1-\kappa)^{2\gamma}(2\lambda_{1}-L_{2})^{\gamma}}
{(\lambda_{f})^{2\gamma}}, 1
\right\} 
\right), 
\quad  \kappa\in (0,1),
\end{equation}
where $\lambda_{f}:=C_{1}(1+2h^{\frac{1}{2}})$, $C_{1}$ is a constant depending only on the drift $f$, determined in \eqref{equation:growth-of-the-drift-f}.
Then the SDE \eqref{eq:semi-linear-SODE} and the corresponding LTPE scheme \eqref{introduction:LTPE-scheme} method converge exponentially to a unique invariant measure, denoted by $\pi$ and $\widetilde{\pi}$, respectively. Moreover,
for some test functions $\varphi \in C^{3}_{b}(\mathbb{R}^{d})$,
\begin{equation}
\left|
\int_{\mathbb{R}^d} 
\varphi(x) \pi(\dd x)
-
\int_{\mathbb{R}^d} 
\varphi(x) \widetilde{\pi}( \dd x)
\right| \leq C_{A}h.
\end{equation}
\end{theorem}
This theorem can be divided into three parts as
\begin{itemize}
\item Existence and  uniqueness  of invariant measure of SDE \eqref{eq:semi-linear-SODE}.
\item Existence and  uniqueness of invariant measure of the LTPE scheme \eqref{introduction:LTPE-scheme}.
\item Time-independent weak error analysis between SDE \eqref{eq:semi-linear-SODE} and the LTPE scheme \eqref{introduction:LTPE-scheme}.
\end{itemize}
In the following, more details of each part will be shown.
%
%
\section{Invariant measure of semi-linear SDE} \label{subsection: Invariant measure of semi-linear SDE}
Indeed, we show the following result.
\begin{theorem} \label{theorem:invariant-measure-of-sode}
Let Assumptions \ref{assumption:one-side-Lipschitz-condition-for-linear-operator}-\ref{assumption:coupled-monotoncity-for-drift-and-diffusion} be fulfilled with $2\lambda_{1} > \max\{L_{1},L_{2}\}$, given $\varphi \in C_{b}^{1}(\mathbb{R}^{d})$,
then the semi-linear SDE $\{ X^{x_{0}}_{t}\}_{t\in [0,T]}$  in \eqref{eq:semi-linear-SODE}, with the initial condition  $X_{0}=x_{0}$, admits a unique invariant measure $\pi$ and there exists some positive constant 
$c_{1}\in (0, 2\lambda_{1}-L_{2}]$
such that
\begin{equation}
\left|
\mathbb{E} \left[ \varphi\left(X^{x_{0}}_{t}\right) \right]
-\int_{\mathbb{R}^d} 
\varphi(x) \pi(d x)
\right| 
\leq C e^{-\frac{c_{1}}{2} t}
\left(
1+\mathbb{E}\left[\|x_{0}\|^2\right]
\right), 
\quad \forall t \in [0,T].
\end{equation}
\end{theorem}
With the condition $2\lambda_{1} > \max\{L_{1}, L_{2}\}$, SDE \eqref{eq:semi-linear-SODE} can be regarded as a dissipative system. 
We follow the standard way, as shown in \cite{da2006introduction}, to prove the existence and uniqueness of the invariant measure inherited by such systems.
For completeness, we outline the central idea in the proof of Theorem \ref{theorem:invariant-measure-of-sode} while the detailed proof of the following lemmas can be found in Appendix.

It is desirable to consider SDE \eqref{eq:semi-linear-SODE} with a negative initial time, that is,
\begin{equation} \label{eq:SODE-with-negative-time}
\begin{split}
\left\{
    \begin{array}{ll}
    \dd X_t  = A X_t + f ( X_t ) \, \dd t + g ( X_t ) \,\dd \widetilde{W}_t,
    \quad
    t \geq -\iota,
    \\
   X_{-\iota} = x_0,
 \end{array}\right.
 \end{split}
\end{equation}
where $\iota\geq0$, $\widetilde{W}_t$ is specified in the following way.
Let $\overline{W}_{t}$ be another Brownian motion independent of $W_{t}$ defined on the probability space $(\Omega,\mathcal{F}, \mathbb{P})$, and define
\begin{equation}
\widetilde{W}_t
=
\left\{\begin{array}{ll}
W_{t}, & t \geq 0 \\
\overline{W}_{t}, & t<0
\end{array}\right.
\end{equation}
with the filtration $\widetilde{\mathcal{F}}_{t} := \sigma \{\widetilde{W}_{s}, s\leq t \}$, $t \in \mathbb{R}$.
In what follows, we write $X^{s,x}_{t}$ in lieu  of $X_{t}$ to highlight the initial value $X_{s}=x$.

Before moving on, we introduce a useful lemma, which is a slight generalization of Lemma 8.1 in \cite{ito1967stationary}, as below,

\begin{lemma} \label{lemma:ito}
If $r(t)$ and $m(t)$ are continuous on $[\tau, \infty)$, $\tau \in \mathbb{R}$, and if 
\begin{equation}
r(t) 
\leq r(s) - \widetilde{c} \int_{s}^{t} r(u) \dd u 
+ \int_{s}^{t} m(u) \dd u , 
\quad \tau \leq s \leq t < \infty
\end{equation}
where $\widetilde{c}$ is a positive constant, then
\begin{equation}
r(t) 
\leq 
r(\tau) 
+  \int_{\tau}^{t} e^{-\widetilde{c}(t-u)} m(u)  \dd u .
\end{equation}
\end{lemma}
The proof of Lemma \ref{lemma:ito} has been shown in \cite{guo2023order}.
It is time to present the uniform moment bounds of the SDE \eqref{eq:SODE-with-negative-time}.
\begin{lemma} \label{lemma:uniform-moments-bound-of-SDEs}
(Uniform moment bounds of semi-linear SDEs.) Let the semi-linear SDEs $\{ X^{-\iota, x_{0}}_{t} \}_{t\geq -\iota}$ in \eqref{eq:SODE-with-negative-time} satisfy Assumptions \ref{assumption:one-side-Lipschitz-condition-for-linear-operator}, \ref{assumption:coercivity-condition-for-drift-and-diffusion} with $2\lambda_{1} > L_{1} $. Then, for any $p \in [1 , p_{0}]$ and $t \in [0, \infty)$,
\begin{equation}
\mathbb{E} 
\left[ 
\left\| X^{-\iota, x_{0}}_{t}  \right\|^{2p} 
\right] 
\leq C < \infty.
\end{equation}
\end{lemma}
The proof of Lemma \ref{lemma:uniform-moments-bound-of-SDEs} can be found in Appendix \ref{proof-of-lemma:uniform-moments-bound-of-SDEs}.
Note that Lemma \ref{lemma:uniform-moments-bound-of-SDEs} can also cover the case $p\in [0,1)$ due to the H\"older inequality.
Following Lemma \ref{lemma:uniform-moments-bound-of-SDEs}, we obtain the contractive property of SDE \eqref{eq:semi-linear-SODE} as follows,
\begin{lemma} \label{lemma:contractivity-of-sde}
(Contractivity of semi-linear SDEs.) Consider the pair of solutions of the semi-linear SDE \eqref{eq:SODE-with-negative-time}, $X^{-\iota, x^{(1)}_{0}}_{t}$ and $X^{-\iota, x^{(2)}_{0}}_{t}$, driven by the same Brownian motion but with different initial state $x^{(1)}_{0}$, $x^{(2)}_{0}$.
Let Assumptions \ref{assumption:one-side-Lipschitz-condition-for-linear-operator}, \ref{assumption:coupled-monotoncity-for-drift-and-diffusion} 
hold with $2\lambda_{1}>L_{2}$, then, there exists a constant $c_{1} \in (0,2\lambda_{1} - L_{2}]$ 
such that, for any $p \in [1, p_{1}]$, $t \geq -\iota$,
\begin{equation}
\mathbb{E} 
\left[  
\Big\| X^{-\iota, x^{(1)}_{0}}_{t} 
- X^{-\iota, x^{(2)}_{0}}_{t} 
\Big\|^{2p}
\right] 
\leq e^{- c_{1}p (t+\iota)} 
\mathbb{E} 
\left[  
\left\| x_0^{(1)} - x_0^{(2)} \right\|^{2p}
\right].
\end{equation}
\end{lemma}
The proof of Lemma \ref{lemma:contractivity-of-sde} can be found in Appendix \ref{proof-of-lemma:contractivity-of-sde}.
The next Lemma is a direct consequence of Lemma \ref{lemma:uniform-moments-bound-of-SDEs} and Lemma \ref{lemma:contractivity-of-sde}. 
\begin{lemma} \label{lemma:cauchy-sequence-of-sode}
Consider the semi-linear SDE in \eqref{eq:SODE-with-negative-time} satisfying Assumptions \ref{assumption:one-side-Lipschitz-condition-for-linear-operator}-\ref{assumption:coupled-monotoncity-for-drift-and-diffusion}
hold with $2\lambda_{1}> \max\{L_{1}, L_{2}\}$.
Let $X^{-s_{1}, x_{0}}_{t}$ and $X^{-s_{2}, x_{0}}_{t}$ with $s_{1}, s_{2} >0 $ satisfying $-s_{1} < -s_{2} \leq t < \infty$, be the solutions of SDE \eqref{eq:SODE-with-negative-time} at time $t$ starting from the same point $x_{0}$ but at different moments.
Then, for any $p \in [1, p_{0}]$, there exists some constant
$c_{2} \in (0, 2\lambda_{1} - L_{2}]$
such that
\begin{equation}
\mathbb{E} 
\left[  
\| X^{-s_{1}, x_{0}}_{t} - X^{-s_{2}, x_{0}}_{t} \|^{2p}
\right] 
\leq 
Ce^{- c_{2}p (t+s_{2})} 
\mathbb{E} 
\left[  
\left(1 + \| x_{0} \|^{2} \right) ^{p}
\right].
\end{equation}
\end{lemma}
The proof of Lemma \ref{lemma:cauchy-sequence-of-sode} is postponed to Appendix \ref{proof-of-lemma:cauchy-sequence-of-sode}.
Equipped with the previously derived lemmas, it is not hard to show Theorem \ref{theorem:invariant-measure-of-sode}.
To be precise,
recalling Lemma \ref{lemma:cauchy-sequence-of-sode}, by sending $s_{1}$ to infinity, one directly observes that $\{X^{-s,x_{0}}_{0}\}_{s>0}$ is a Cauchy sequence in $L^{2}(\Omega, \mathbb{R}^{d})$ and  there exists $\vartheta^{x_{0}}$ in $L^{2}(\Omega, \mathbb{R}^{d})$ such that
\begin{equation}
\vartheta^{x_{0}}
:=\lim_{s_{1} \rightarrow \infty}
X^{-s_{1},x_{0}}_{0}.
\end{equation}
Using Lemma \ref{lemma:cauchy-sequence-of-sode} again yields
\begin{equation}
\mathbb{E} 
\left[  
\| X^{-s_{2}, x_{0}}_{0} - \vartheta^{x_{0}}  \|^{2}
\right] 
= 
\lim_{s_{1} \rightarrow \infty} 
\mathbb{E} 
\left[  
\| X^{-s_{2}, x_{0}}_{0} - X^{-s_{1}, x_{0}}_{0} \|^{2}
\right] 
\leq 
e^{- c_{2} s_{2}} 
\mathbb{E} 
\left[  1 + \| x_{0} \|^{2} \right].
\end{equation}
By Lemma \ref{lemma:contractivity-of-sde}, we know $\vartheta^{x_{0}}$ is independent of $x_{0}$, i.e. 
\begin{equation}
\begin{aligned}
\mathbb{E}
\left[
\left\|
\vartheta^{x_{0}}- \vartheta^{x_{1}} 
\right\|^{2}
\right] 
= \lim_{s_{1}\rightarrow \infty}
\mathbb{E}
\left[
\left\|
X^{-s_{1},x_{0}}_{0}- X^{-s_{1},x_{1}}_{0}
\right\|^{2} 
\right]
\leq 
\lim_{s_{1}\rightarrow \infty} 
e^{- c_{1} s_{1}} 
\mathbb{E} 
\left[  \|x_{0} - x_{1}\|^{2}\right] = 0,
\end{aligned}
\end{equation}
and thus denoted by $\vartheta$.
Let $\pi$ be the law of the random variable $\vartheta$, then $\pi$ is the unique invariant measure for SDE \eqref{eq:semi-linear-SODE}. 
Moreover, 
since $X^{x_{0}}_{t}$ and $X^{-t,x_{0}}_{0}$ have the same distribution, for any function $\varphi \in C^{1}_{b}(\mathbb{R}^{d})$, we can get 
\begin{equation}
\begin{aligned}
\left|
\mathbb{E} 
\left[
\varphi\left(X^{x_{0}}_{t}\right)
\right]
-\int_{\mathbb{R}^d} 
\varphi(x) \pi(\dd x)
\right| 
&= 
\big|
\mathbb{E} 
\left[
\varphi\left(X^{x_{0}}_{t}\right)
- \varphi\left(\vartheta\right)
\right]
\big| \\
&\leq 
\|\varphi \|_{1} 
\mathbb{E} 
\left[
\left\|
X^{-t,x_{0}}_{0} - \vartheta 
\right\| 
\right]\\
&\leq 
C e^{-\frac{c_{2}}{2} t}
\left(
1
+\mathbb{E} \left[\|x_{0}\|^2 \right]
\right), 
\quad \forall t \in [0,T].
\end{aligned}
\end{equation}
\section{Invariant measure of the LTPE scheme} \label{subsection:Invariant-measure-of-the-numerical-scheme}
The main result of this Section is provided as below.
\begin{theorem} \label{theorem:invariant-measure-of-LTPE-method}
Let Assumptions \ref{assumption:one-side-Lipschitz-condition-for-linear-operator}-\ref{assumption:growth-condition-of-frechet-derivatives-of-drift-and-diffusion} hold with $2\lambda_{1}> \max\{L_{1}, L_{2}\}$. For a method parameter $\theta \in [0,1]$, 
 consider the LTPE method in \eqref{introduction:LTPE-scheme} subject to
a  uniform timestep $h$ satisfying
\begin{equation}
h \in 
\left(
0, \min\left\{
\tfrac{\kappa^{2}(2\lambda_{1}-L_{2})}{(1- \theta)^{2}\lambda^{2}_{d}}, 
\tfrac{(1-\kappa)^{2\gamma}(2\lambda_{1}-L_{2})^{\gamma}}{(\lambda_{f})^{2\gamma}} ,
1
\right\} 
\right), \quad \kappa\in (0,1).
\end{equation}
Then the numerical simulation from LTPE \eqref{introduction:LTPE-scheme} method, denoted by $\{ Y^{x_{0}}_{n}\}_{0\leq n \leq N}$ with the initial point $x_{0}$, admits a unique invariant measure $\widetilde{\pi}$. Moreover, there exists some positive constant
$\widetilde{C}_{1}$ such that, for some function $\varphi \in C^{1}_{b}(\mathbb{R}^{d})$, $t_{n}=nh$, $ n \in \{0, 1, \cdots N \}, \ N \in \mathbb{N}$,
\begin{equation}
    \left|\mathbb{E} \left[\varphi\left(Y^{x_{0}}_{n}\right) \right]-\int_{\mathbb{R}^d} \varphi(x) \widetilde{\pi}(d x)\right| \leq C_{A} e^{-\frac{\widetilde{C}_{1}}{2} t_{n}}\left(1+\mathbb{E}\left[\|x_{0}\|^2\right]\right).
\end{equation}
\end{theorem}
The theorem above can be proved in exactly the same way that Theorem \ref{theorem:invariant-measure-of-sode} is proved, where the ergodicity of the LTPE \eqref{introduction:LTPE-scheme} boils down to verifying the uniform moment bounds (see Lemma \ref{lemma:uniform-moments-bound-of-the-LTPE-method}) and the contractive property (see Lemma \ref{lemma:contractivity-of-the-theta-linear-projected-Euler-method}). Before proceeding further, we first establish some preliminary estimates necessary for the proof of Theorem \ref{theorem:invariant-measure-of-LTPE-method}.
\begin{lemma} \label{lemma:necessary-estimates}
Recall the definition of $\mathscr{P}(x)$ in \eqref{eq:projected-operator}. Let Assumptions \ref{assumption:coercivity-condition-for-drift-and-diffusion}, \ref{assumption:growth-condition-of-frechet-derivatives-of-drift-and-diffusion} be fulfilled, then for any $ x \in \mathbb{R}^{d}$ the following estimates
\begin{equation} \label{equation:first-three-estimates-in-lemma}
\begin{aligned}
\|\mathscr{P}(x) \| 
&\leq \min \left\{ 
\|x \| , h^{-\frac{1}{2\gamma}}
\right\}, \quad
\| f\big(\mathscr{P}(x)\big) \| 
\leq  C_{f}h^{-\frac{1}{2}}, \\
\|g(\mathscr{P}(x)) \|^{2} 
&\leq \tfrac{L_{1}}{2p_{0}-1} (1 + \| \mathscr{P}(x)\|^{2}) + 2C_{f}h^{-\frac{1}{2}} \|\mathscr{P}(x)\|
\end{aligned}
\end{equation}
hold true, 
where $C_{f} := C_{2} (1+h^{\frac{1}{2}})$. Especially, for any integer $p \geq 1$, we have, for $ x \in \mathbb{R}^{d}$,
\begin{equation} \label{equation:second-estimate-in-lemma}
\begin{aligned}
\|g(\mathscr{P}(x)) \|^{2p} 
&\leq 
\left(
\tfrac{L_{1}}{2p_{0}-1} 
\right)^{p} 
\big(
1 + \| \mathscr{P}(x)\|^{2}
\big)^{p} 
+  C   h^{-\frac{p}{2}} 
\big(
1 + \| \mathscr{P}(x)\|^{2}
\big)^{p-1}.
\end{aligned}
\end{equation}
Moreover, for any $x, y \in \mathbb{R}^{d}$, the following estimates hold true
\begin{equation} \label{eq:estimates-in-lemma-projected-Lipschitz}
\begin{aligned}
\| \mathscr{P}(x) - \mathscr{P}(y)  \| 
&\leq \|x-y \| , \\
\big\| 
f\big(\mathscr{P}(x)\big) - f\big(\mathscr{P}(y)\big) 
\big\|
&\leq \lambda_{f}  h^{-\frac{\gamma- 1}{2\gamma}}  \|x-y \|,
\end{aligned}
\end{equation}
where $\lambda_{f} := C_{1}(1 + 2h^{\frac{1}{2}}) $ depending only on $f$.
\end{lemma}
The proof of Lemma \ref{lemma:necessary-estimates} can be found in Appendix \ref{proof-of-lemma:necessary-estimates}.
%
The next lemma provides the uniform moment estimates for the LTPE scheme \eqref{introduction:LTPE-scheme}.
\begin{lemma} \label{lemma:uniform-moments-bound-of-the-LTPE-method}
(Uniform moment bounds of the LTPE method) Let Assumptions \ref{assumption:one-side-Lipschitz-condition-for-linear-operator}, \ref{assumption:coercivity-condition-for-drift-and-diffusion} and \ref{assumption:growth-condition-of-frechet-derivatives-of-drift-and-diffusion} hold  with $2\lambda_{1} >  L_{1}$. For a method parameter $\theta \in [0,1]$, 
 consider the numerical simulation $Y_n$ from LTPE method in \eqref{introduction:LTPE-scheme}. Then, for any uniform stepsize $h\in (0,1)$ and $n\in \{ 0,1,2,\dots,N-1 \}$, $N\in \mathbb{N}$, 
\begin{equation} \label{equation:second-order-moment-estimate-of-LTPE-method}
\mathbb{E} 
\left[ \left\| Y_{n}  \right\|^{2} \right] 
\leq C_{A} < \infty.
\end{equation}
Moreover, for $p \in (1,p_{0}) \cap \mathbb{N}$, if the timestep $h$ further
satisfies
\begin{equation}
h \in 
\left(
0, 
\min\left\{
\tfrac{1}{2(1-\theta)\lambda_{1}},
\tfrac{p_{0}-p}{(1-\theta)(2p_{0}-p-1)\lambda_{1}},
\tfrac{1}{(1-\theta)\lambda_{d}},1
\right\}
\right),
\end{equation}
then,  for any $n\in \{ 0,1,2,\dots,N-1 \}$,   $N\in \mathbb{N}$,
\begin{equation} \label{equation:2p-thmoment-estimate-of-LTPE-method}
\mathbb{E} \left[
\left\| Y_{n}  \right\|^{2p} 
\right] 
\leq C_{A} < \infty.
\end{equation}
\end{lemma}
\begin{proof}
[Proof of Lemma \ref{lemma:uniform-moments-bound-of-the-LTPE-method}]
We first take square of \eqref{introduction:LTPE-scheme} on both sides and analyze the left and right hand sides individually. With Assumption \ref{assumption:one-side-Lipschitz-condition-for-linear-operator} 
being used, the left hand side goes to 
\begin{equation} \label{equation:lhs}
\begin{aligned}
\| Y_{n+1} -  \theta A Y_{n+1} h \| ^{2} 
& = 
\|Y_{n+1} \|^{2} 
- 
2 \theta h \langle Y_{n+1} ,A Y_{n+1} \rangle  
+ \theta ^{2} h^{2} \|A Y_{n+1} \|^{2} \\
 &\geq (1 + 2 \theta \lambda_{1} h ) \|Y_{n+1} \|^{2}.
\end{aligned}
\end{equation}
On the other hand,
the right hand side goes to
\begin{equation} \label{equation:rhs}
\begin{aligned}
&\left\| 
\mathscr{P}(Y_{n}) 
+ (1-\theta) A \mathscr{P}(Y_{n}) h 
+ f\big(\mathscr{P}(Y_{n})\big) h 
+ g\big(\mathscr{P}(Y_{n})\big) \Delta W_{n} 
\right\|^{2} \\
& = 
\| \mathscr{P}(Y_{n}) \|^{2} 
+ (1-\theta)^{2}h^{2} 
\left\| A \mathscr{P}(Y_{n}) \right\|^{2} 
+ h^{2} 
\left\| f\big(\mathscr{P}(Y_{n})\big)\right\|^{2} 
+ 
\left\| 
g\big(\mathscr{P}(Y_{n})\big) \Delta W_{n} 
\right\|^{2}\\
&\quad + 2 (1-\theta) h 
\left\langle 
\mathscr{P}(Y_{n}) , A\mathscr{P}(Y_{n}) 
\right\rangle 
+ 2  h 
\left\langle 
\mathscr{P}(Y_{n}) ,
f\big(\mathscr{P}(Y_{n})\big) 
\right\rangle 
+ 2 \left\langle 
B\mathscr{P}(Y_{n}) , g\big(\mathscr{P}(Y_{n})\big) \Delta W_{n} 
\right\rangle \\
&\quad + 2 (1-\theta)h^{2}  
\left\langle 
A \mathscr{P}(Y_{n}) , f\big(\mathscr{P}(Y_{n})\big) 
\right\rangle 
+ 2  h 
\big\langle 
f\big(\mathscr{P}(Y_{n})\big) , g\big(\mathscr{P}(Y_{n})\big) \Delta W_{n} 
\big\rangle, \\
\end{aligned}
\end{equation}
where $B := I + (1-\theta)Ah$.
In the following, let us start by the estimation of \eqref{equation:second-order-moment-estimate-of-LTPE-method}.

\noindent\textbf{Case I: estimate of $\mathbb{E}\left[\|Y_{n+1} \|^{2p}\right]$ when $p=1$.}

Using the Young inequality yields
\begin{equation} \label{equation:young-inequality-of-some-cross-terms-first-part}
\begin{aligned}
2 (1-\theta)h^{2}  
\left\langle 
A \mathscr{P}(Y_{n}) , f\big(\mathscr{P}(Y_{n})\big) 
\right\rangle 
&\leq (1-\theta)^{2}h^{2} 
\left\| A \mathscr{P}(Y_{n}) \right\|^{2} 
+ h^{2} \left\| f\big(\mathscr{P}(Y_{n})\big)\right\|^{2}.
\end{aligned}
\end{equation}
Taking  expectations of \eqref{equation:lhs} and \eqref{equation:rhs} respectively with Lemma \ref{lemma:necessary-estimates} and the fact that $\mathbb{E}\left[\Delta W_{n} | \mathcal{F}_{t_{n}} \right]=0$ shows
\begin{equation}
\begin{aligned}
&(1 + 2 \theta \lambda_{1} h )
\mathbb{E} 
\left[
\|Y_{n+1} \|^{2}  
\right]  \\
&\leq 
\left[
1- 2 (1-\theta)\lambda_{1}h 
\right]  
\mathbb{E}
\left[ \|\mathscr{P}(Y_{n}) \|^{2} \right]  
+ h 
\mathbb{E}
\left[
\|g\big(\mathscr{P}(Y_{n})\big)  \|^{2} 
\right] 
+ 2  h 
\mathbb{E}
\left[ 
\left\langle \mathscr{P}(Y_{n}) , f\big(\mathscr{P}(Y_{n})\big) 
\right\rangle 
\right]\\
&\quad + 2  C_{f}^{2} h +2 (1-\theta)^{2} \lambda_{d}^{2} h^{2- \frac{1}{\gamma} } .
 \end{aligned}
 \end{equation}
This in conjunction with Assumption \ref{assumption:coercivity-condition-for-drift-and-diffusion} with $2\lambda_{1}>L_{1}$ leads to, for some positive constant $C_{A}=C(L_{1},\lambda_{d}, C_{f}, \theta)$ and $\overline{C}:= (2\lambda_{1}-L_{1})/(1+2\theta \lambda_{1}h)$, 
\begin{equation}
\begin{aligned}
\mathbb{E} 
\left[\|Y_{n+1} \|^{2}  \right] 
&\leq  
\tfrac{1- 2 (1-\theta)\lambda_{1}h + L_{1}h}
{(1 + 2 \theta \lambda_{1} h )}   \mathbb{E}\big[\|\mathscr{P}(Y_{n}) \|^{2} \big] 
+   C_{A} h\\
& =\left( 1-\overline{C} h\right)  \mathbb{E}\big[\|\mathscr{P}(Y_{n}) \|^{2}\big]  +   C_{A} h \\
& \leq \left( 1-\overline{C} h\right)^{n+1}  
\mathbb{E}\left[ \|x_{0} \|^{2} \right] 
+ \tfrac{C_{A}}{\overline{C}} \\
& \leq e^{-\overline{C}t_{n+1}} 
\mathbb{E}\left[\|x_{0} \|^{2} \right] 
+ \tfrac{C_{A}}{\overline{C}},
\end{aligned}
\end{equation} 
where $1-x \leq e^{-x}$ for any $x>0$.

\noindent\textbf{Case II: estimate of $\mathbb{E}\left[\|Y_{n+1} \|^{2p}\right]$ when $p\in (1,p_{0}) \cap \mathbb{N}$.}

Proceeding to the estimate of higher order moment of the LTPE method \eqref{introduction:LTPE-scheme}, some restrictions need to be imposed on the timestep $h$. Recalling $B = I + (1-\theta)Ah$,
with $h\in (0,1/[(1-\theta)\lambda_{d}])$, 
obviously, the matrix $B$ is positive definite and 
$\max_{i=1,\dots,d}\lambda_{B,i} = 1 - (1-\theta)\lambda_{1}h$.
By the Young inequality, we get, for some positive constant $\epsilon_{1}\in (0, (2p_{0}-2p)/(2p-1)]$,
\begin{equation} \label{equation:young-inequality-of-some-cross-terms}
\begin{aligned}
2  h 
\left\langle 
f\big(\mathscr{P}(Y_{n})\big) , 
g\big(\mathscr{P}(Y_{n})\big) \Delta W_{n} 
\right\rangle 
& \leq \tfrac{h^{2}}{\epsilon_{1}} 
\left\| f\big(\mathscr{P}(Y_{n})\big)\right\|^{2} 
+ \epsilon_{1} 
\left\|g\big(\mathscr{P}(Y_{n})\big) \Delta W_{n} \right\|^{2}.
\end{aligned}
\end{equation}
Plugging estimates \eqref{equation:young-inequality-of-some-cross-terms-first-part}, \eqref{equation:young-inequality-of-some-cross-terms}, together with Assumption \ref{assumption:one-side-Lipschitz-condition-for-linear-operator} and Lemma \ref{lemma:necessary-estimates}, into \eqref{equation:rhs} to show that
\begin{equation}
\begin{aligned}
&(1 + 2 \theta \lambda_{1} h ) (1+\|Y_{n+1} \|^{2}) \\
&\leq \left[1- 2 (1-\theta)\lambda_{1}h \right] \left(1 + \|\mathscr{P}(Y_{n}) \|^{2} \right)  
+ (1+\epsilon_{1})
\left\|
g\big(\mathscr{P}(Y_{n})\big) \Delta W_{n} 
\right\|^{2}  
+ 2  h 
\left\langle \mathscr{P}(Y_{n}) , f\big(\mathscr{P}(Y_{n})\big) 
\right\rangle \\
&\quad + 
2 
\left\langle 
B\mathscr{P}(Y_{n}) 
, 
g\big(\mathscr{P}(Y_{n})\big) \Delta W_{n} 
\right\rangle  
+ \left[
(2+\tfrac{1}{\epsilon_{1}} )  C_{f}^{2}  + (1-\theta)^{2} \lambda_{d}^{2}  + 2 \lambda_{1} 
\right]h\\
& =: \widetilde{L} 
\left(1 + \|\mathscr{P}(Y_{n}) \|^{2} \right) 
(1 + \Xi_{n+1}) 
+  C_{\epsilon_{1}}  h,
\end{aligned}
\end{equation}
where $\widetilde{L} := 1- 2 (1-\theta)\lambda_{1}h \geq 0$, $C_{\epsilon_{1}} = C(\lambda_{1}, \lambda_{d}, \theta, C_{f})  = (2+\tfrac{1}{\epsilon_{1}} )  C_{f}^{2}  + (1-\theta)^{2} \lambda_{d}^{2}  + 2 \lambda_{1}$, and
\begin{equation} \label{equation: xi}
\begin{aligned}
\Xi_{n+1} &=  \underbrace{\tfrac{(1+\epsilon_{1})\left\|g\left(\mathscr{P}(Y_{n})\right) \Delta W_{n} \right\|^{2}  }{\widetilde{L}  (1 + \|\mathscr{P}(Y_{n}) \|^{2})} }_{=:I_{1}} 
+ \underbrace{\tfrac{2  h \left\langle \mathscr{P}(Y_{n}) , f\left(\mathscr{P}(Y_{n})\right) \right\rangle }{\widetilde{L}  (1 + \|\mathscr{P}(Y_{n}) \|^{2})} }_{=:I_{2}}
+ \underbrace{\tfrac{ 2 \left\langle B\mathscr{P}(Y_{n}) , g\left(\mathscr{P}(Y_{n})\right) \Delta W_{n} \right\rangle }{\widetilde{L}  (1 + \|\mathscr{P}(Y_{n}) \|^{2})} }_{=:I_{3}} .\\
\end{aligned}
\end{equation}
Following the binomial expansion theorem and taking the conditional mathematical expectation with respect to $\mathcal{F}_{t_{n}}$ on both sides to show that,
\begin{equation} \label{equatoin:basic-expansion-of-p-th-moment}
\begin{aligned}
&(1 + 2 \theta \lambda_{1} h )^{p} 
\mathbb{E}
\left[
(1+\|Y_{n+1} \|^{2})^{p} \big| \mathcal{F}_{t_{n}} 
\right] \\
&\leq 
  \underbrace{ (1 + \|\mathscr{P}(Y_{n}) \|^{2})^{p}  \widetilde{L}^{p}\mathbb{E} \left[(1 + \Xi_{n+1}) ^{p} \big| \mathcal{F}_{t_{n}} \right]}_{=:\mathbb{I}_{1}} + \underbrace{ C_{\epsilon_{1}}h (1 + \|\mathscr{P}(Y_{n}) \|^{2})^{p-1} \sum_{i=0}^{p-1} \widetilde{L}^{i} \mathbb{E} \left[(1 + \Xi_{n+1}) ^{i} \big| \mathcal{F}_{t_{n}} \right] }_{=:\mathbb{I}_{2}}
\end{aligned}
\end{equation}
with $C_{\epsilon_{1}}=C(\lambda_{1}, \lambda_{d}, \theta, C_{f},p)$.
Hence, the analysis can be divided into the following two parts.

\noindent \textbf{For the estimate of $\mathbb{I}_{1}$:}

According to the binomial expansion theorem again, one has
\begin{equation} 
\begin{aligned}
\mathbb{E}\left[ (1 + \Xi_{n+1})^{p} \big| \mathcal{F}_{t_{n}} \right] &=
 \sum_{i=0}^{p}\mathscr{C}_{p}^{i} \mathbb{E}\left[\Xi_{n+1}^{i} \big| \mathcal{F}_{t_{n}} \right]  \\ 
 &= 1 + p \mathbb{E}\left[\Xi_{n+1} \big| \mathcal{F}_{t_{n}} \right] + \tfrac{p(p-1)}{2} \mathbb{E}\left[\Xi_{n+1}^{2} \big| \mathcal{F}_{t_{n}} \right] + \tfrac{p(p-1)(p-2)}{6}\mathbb{E}\left[\Xi_{n+1}^{3} \big| \mathcal{F}_{t_{n}} \right] + ... \ ,
\end{aligned}
\end{equation}
where $\mathscr{C}_{p}^{i}:= p!/(i! (p-i)!)$.
Let us decompose the estimate of $\mathbb{I}_{1}$ further into four steps.
\begin{description}
    \item[Step I: the estimate of]$\mathbb{E}\left[\Xi_{n+1} \big| \mathcal{F}_{t_{n}} \right]$.

Based on the property of Brownian motion and the fact that $\Delta W_{n}$ is independent of $\mathcal{F}_{t_{n}}$, we deduce
\begin{equation} \label{equation:property-of-Brownian-motion}
\begin{aligned}
\mathbb{E} 
\left[
\Delta W_{j,n} \big| \mathcal{F}_{t_{n}}  
\right]  =0, \quad 
\mathbb{E} 
\left[
|\Delta W_{j,n}|^{2} \big| \mathcal{F}_{t_{n}}  
\right]  =h, 
\quad j \in \{1, \ldots, m\},
\end{aligned}
\end{equation}
leading to 
\begin{equation} 
\begin{aligned}
\mathbb{E}\left[\Xi_{n+1} \big| \mathcal{F}_{t_{n}} \right] = \tfrac{(1+\epsilon_{1})h \left\|g\left(\mathscr{P}(Y_{n})\right) \right\|^{2} + 2  h \left\langle \mathscr{P}(Y_{n}) , f\left(\mathscr{P}(Y_{n})\right) \right\rangle  }{\widetilde{L} (1 + \|\mathscr{P}(Y_{n}) \|^{2})}.
\end{aligned}
\end{equation}
    \item[Step II: the estimate of]$\mathbb{E}\left[\Xi_{n+1}^{2} \big| \mathcal{F}_{t_{n}} \right]$.

Recalling some power properties of Brownian motions, we derive that, for any $\ell \in \mathbb{N}$,
\begin{equation} \label{equation:power-property-of-Brownian-motion}
\begin{aligned}
\mathbb{E} 
\left[
\left( \Delta W_{j,n}\right)^{2\ell-1}  
\big| \mathcal{F}_{t_{n}} 
\right] =0,\quad
\mathbb{E} 
\left[
\left( \Delta W_{j,n}\right)^{2\ell}  
\big| \mathcal{F}_{t_{n}} 
\right] =  (2 \ell -1)!! \ h^{\ell},  
\quad \forall \ n \in \mathbb{N},\ j \in \{1, \ldots, m\},
\end{aligned}
\end{equation}
where $(2 \ell -1)!! := \Pi_{i=1}^{\ell}(2\ell-1)$.
Before moving on, we here introduce a series of useful estimates. For any $\ell \in [2, \infty) \cap \mathbb{N}$
, by Lemma \ref{lemma:necessary-estimates} and \eqref{equation:power-property-of-Brownian-motion}, one can achieve with some constant $C = C(L_{1}, C_{f}, p)$,
\begin{equation} \label{equation:estimate-of-I1}
\begin{aligned}
\mathbb{E}
\left[
(I_{1})^{\ell} \big| \mathcal{F}_{t_{n}}  
\right] 
&= 
\tfrac{
(1 + \epsilon_{1})^{\ell} 
(2\ell-1)!! \ h^{\ell}
\left\|
g\left(\mathscr{P}(Y_{n})\right)
\right\|^{2\ell}
}
{
\widetilde{L}^{\ell} (1 + \|\mathscr{P}(Y_{n}) \|^{2})^{\ell}
} \\
&\leq 
\tfrac{
(1 + \epsilon_{1})^{\ell} (2\ell-1)!! \ 
h^{\ell} 
\left[
L_{1}^{\ell} 
(1 + \| \mathscr{P}(Y_{n})\|^{2})^{\ell} 
+  C     h^{-\frac{\ell}{2}} 
(1 + \| \mathscr{P}(Y_{n})\|^{2})^{\ell-1} 
\right]
}
{
(2p_{0} - 1)^{\ell} 
\widetilde{L}^{\ell} 
(1 + \|\mathscr{P}(Y_{n}) \|^{2})^{\ell}
} \\
& \leq 
\tfrac{
(1 + \epsilon_{1})^{\ell} 
L_{1}^{\ell} (2\ell-1)!!  \ h^{\ell}
}
{
(2p_{0} - 1)^{\ell} \widetilde{L}^{\ell}
} 
+ 
\tfrac{
C(1 + \epsilon)^{\ell} 
(2\ell-1)!!  h^{\frac{\ell}{2}}
}
{
\widetilde{L}^{\ell} 
(1 + \|\mathscr{P}(Y_{n}) \|^{2})
}.
\end{aligned}
\end{equation}
Similarly, with the Cauchy Schwarz inequality, one gets
\begin{equation} 
\begin{aligned}
\mathbb{E}
\left[ 
(I_{2})^{\ell} \big| \mathcal{F}_{t_{n}}  
\right] 
&= \tfrac{
2^{\ell} h^{\ell} 
\left( 
\left\langle 
\mathscr{P}(Y_{n}) , f\left(\mathscr{P}(Y_{n})\right)
\right\rangle 
\right)^{\ell} }
{
\widetilde{L}^{\ell} 
(1 + \|\mathscr{P}(Y_{n}) \|^{2})^{\ell}
} 
\leq 
\tfrac{
2^{\ell} h^{\ell}  
\|\mathscr{P}(Y_{n})\|^{\ell}
\|f\left(\mathscr{P}(Y_{n})\right) \|^{\ell}
}
{
\widetilde{L}^{\ell} 
(1 + \|\mathscr{P}(Y_{n}) \|^{2})^{\ell}
}.
\end{aligned}
\end{equation}
For any $\ell \geq 2$ and $x \geq 0$, we know that $x^{\frac{\ell}{2}} \leq (1 + x^{2})^{\ell-1}$. Therefore, with $C=C(C_{f}, \ell)$, we obtain
\begin{equation} \label{equation:estimate-of-I2}
\begin{aligned}
\mathbb{E}
\left[
(I_{2})^{\ell} \big| \mathcal{F}_{t_{n}}  
\right] 
\leq
\tfrac{
C h^{\frac{\ell}{2}}
}
{
\widetilde{L}^{\ell} 
(1 + \|\mathscr{P}(Y_{n}) \|^{2})
}.
\end{aligned}
\end{equation}
One needs to be careful about the estimate of term $I_{3}$. Equipping with \eqref{equation:property-of-Brownian-motion} yields
\begin{equation} \label{equation:estimate-of-I3-square}
\begin{aligned}
\mathbb{E}\left[ (I_{3})^{2} \big| \mathcal{F}_{t_{n}}  \right] &=
\tfrac{4h \left\| \left(B\mathscr{P}(Y_{n})\right)^{T} g\left(\mathscr{P}(Y_{n})\right) \right\|^{2} }{\widetilde{L}^{2} (1 + \|\mathscr{P}(Y_{n}) \|^{2})^{2} } 
\leq \tfrac{4h\left[1-(1-\theta)\lambda_{1}h \right]^{2}}{1-2(1-\theta)\lambda_{1}h} \tfrac{\left\|  g\left(\mathscr{P}(Y_{n})\right) \right\|^{2}}{\widetilde{L} (1 + \|\mathscr{P}(Y_{n}) \|^{2})} .
\end{aligned}
\end{equation}
It is time to move on to the estimate of $\mathbb{E}\left[\Xi_{n+1}^{2} \big| \mathcal{F}_{t_{n}} \right]$. We begin with the following expansion
\begin{equation} \label{equation: expansion xi 2}
\begin{aligned}
\mathbb{E}\left[\Xi_{n+1}^{2} \big| \mathcal{F}_{t_{n}} \right] &= \mathbb{E}\left[ (I_{1}+I_{2}+I_{3})^{2} \big| \mathcal{F}_{t_{n}}  \right]  \\
 & = \mathbb{E}\left[ (I_{1})^{2} + (I_{2})^{2} + (I_{3})^{2} + 2I_{1}I_{2} + 2 I_{1}I_{3} + 2I_{2}I_{3}\big| \mathcal{F}_{t_{n}}  \right].
\end{aligned}
\end{equation}
As claimed before, one will observe
\begin{equation} 
\begin{aligned}
\mathbb{E}\left[ I_{1}I_{3} \big| \mathcal{F}_{t_{n}}  \right] = \mathbb{E}\left[ I_{2}I_{3} \big| \mathcal{F}_{t_{n}}  \right] = 0,
\end{aligned}
\end{equation}
and, for $C=C(L_{1}, C_{f})$, 
\begin{equation} 
\begin{aligned}
\mathbb{E}\left[ I_{1}I_{2} \big| \mathcal{F}_{t_{n}}  \right] \leq 
\tfrac{ C  h} {\widetilde{L}^{2} (1 + \|\mathscr{P}(Y_{n}) \|^{2})}
\end{aligned}
\end{equation}
from \eqref{equation:estimate-of-I1}, \eqref{equation:estimate-of-I2} and the H\"older inequality. Plugging these with \eqref{equation:estimate-of-I1}, \eqref{equation:estimate-of-I2} and \eqref{equation:estimate-of-I3-square} into \eqref{equation: expansion xi 2} to show that
\begin{equation} \label{equation: estimate xi 2}
\begin{aligned}
\mathbb{E}\left[\Xi_{n+1}^{2} \big| \mathcal{F}_{t_{n}} \right] &\leq
 \tfrac{3 (1 + \epsilon_{1})^{2} L_{1}^{2}  h^{2}}{(2p_{0}-1)^{2}\widetilde{L}^{2}} +
 \tfrac{4h\left[1-(1-\theta)\lambda_{1}h \right]^{2}}{1-2(1-\theta)\lambda_{1}h} \tfrac{\left\|  g(\mathscr{P}(Y_{n})) \right\|^{2}}{\widetilde{L} (1 + \|\mathscr{P}(Y_{n}) \|^{2})} 
 +\tfrac{C h} {\widetilde{L}^{2} (1 + \|\mathscr{P}(Y_{n}) \|^{2})}, \\
\end{aligned}
\end{equation}
where $C = C(L_{1}, C_{f})$.
As we know, for any positive constant $\ell \in [2,p] \cap \mathbb{N}$, $p<p_{0}$ and 
$\epsilon_{1}\in (0, (2p_{0}-2p)/(2p-1)]$,
\begin{equation}\label{equation:inequality-for-ell-and-p}
(2\ell-1)!! (1 + \epsilon_{1})^{\ell} 
< (2p-1)^{\ell}(1 + \epsilon_{1})^{\ell} 
< (2p_{0}-1)^{\ell},
\end{equation}
so that we obtain
\begin{equation} \label{equation: second estimate xi 2}
\begin{aligned}
\mathbb{E}\left[\Xi_{n+1}^{2} \big| \mathcal{F}_{t_{n}} \right] &\leq
 \tfrac{ L_{1}^{2}  h^{2}}{\widetilde{L}^{2}} +
 \tfrac{4h\left[1-(1-\theta)\lambda_{1}h \right]^{2}}{1-2(1-\theta)\lambda_{1}h} \tfrac{\left\|  g(\mathscr{P}(Y_{n})) \right\|^{2}}{\widetilde{L} (1 + \|\mathscr{P}(Y_{n}) \|^{2})} 
 +\tfrac{C h} {\widetilde{L}^{2} (1 + \|\mathscr{P}(Y_{n}) \|^{2})}. \\
\end{aligned}
\end{equation}
    \item[Step III: the estimate of]$\mathbb{E}\left[\Xi_{n+1}^{3} \big| \mathcal{F}_{t_{n}} \right]$.

By the similar procedure, we can acquire that
\begin{equation} 
\begin{aligned}
\mathbb{E}\left[\Xi_{n+1}^{3} \big| \mathcal{F}_{t_{n}} \right] &= \mathbb{E}\left[ (I_{1}+I_{2}+I_{3})^{3} \big| \mathcal{F}_{t_{n}}  \right] \\
\end{aligned}
\end{equation}
where \eqref{equation:property-of-Brownian-motion} and \eqref{equation:power-property-of-Brownian-motion} are   used to imply that
\begin{equation} 
\begin{aligned}
\mathbb{E}
\left[
(I_{3})^{3} \big| \mathcal{F}_{t_{n}}  
\right] 
&= \mathbb{E}
\left[
(I_{1})^{2}I_{3} \big| \mathcal{F}_{t_{n}}  
\right] 
= 
\mathbb{E}
\left[ 
(I_{2})^{2}I_{3} \big| \mathcal{F}_{t_{n}}  
\right] 
=
\mathbb{E}
\left[
I_{1}I_{2}I_{3} \big| \mathcal{F}_{t_{n}}  
\right] =
0. \\
\end{aligned}
\end{equation}
Obeying \eqref{equation:estimate-of-I1}-\eqref{equation:estimate-of-I3-square},\eqref{equation:inequality-for-ell-and-p}  yields with $C=C(L_{1}, C_{f})$,
\begin{equation} 
\begin{aligned}
\mathbb{E}
\left[
\Xi_{n+1}^{3} \big| \mathcal{F}_{t_{n}} 
\right] 
&=
\mathbb{E}
\left[
(I_{1})^{3} \big| \mathcal{F}_{t_{n}}  
\right] 
+
\mathbb{E}
\left[
(I_{2})^{3} \big| \mathcal{F}_{t_{n}}  
\right] 
+ 3\mathbb{E}
\left[
(I_{1})^{2} \cdot I_{2} \big| \mathcal{F}_{t_{n}}  
\right] 
+ 3\mathbb{E}
\left[
I_{1} \cdot (I_{2})^{2} \big| \mathcal{F}_{t_{n}}  
\right] \\
& \quad + 3\mathbb{E}
\left[
I_{1} \cdot (I_{3})^{2} \big| \mathcal{F}_{t_{n}}  
\right] 
+ 3\mathbb{E}
\left[
I_{2} \cdot (I_{3})^{2} \big| \mathcal{F}_{t_{n}}  
\right] 
\\
&\leq  
\tfrac{
15 (1 + \epsilon_{1})^{3} L_{1}^{3}    h^{3}
}
{
(2p_{0}-1)^{3}\widetilde{L}^{3}
} 
+\tfrac{C h} 
{
\widetilde{L}^{3} (1 + \|\mathscr{P}(Y_{n}) \|^{2})
}\\ 
& \leq 
\tfrac{
L_{1}^{3}    h^{3}
}
{
\widetilde{L}^{3}
} 
+\tfrac{C h} 
{
\widetilde{L}^{3} 
(1 + \|\mathscr{P}(Y_{n}) \|^{2})
}.
\end{aligned}
\end{equation}
    \item[Step IV: the estimate of]$\mathbb{E}\left[\Xi_{n+1}^{\ell} \big| \mathcal{F}_{t_{n}} \right]$, $ \ell \in [4,p] \cap \mathbb{N}$.

It follows from Lemma \ref{lemma:necessary-estimates} that, for $ \ell \in [4,p] \cap \mathbb{N}$,
\begin{equation} \label{equation:estimate-of-I3-p}
\begin{aligned}
\mathbb{E}\left[
(I_{3})^{\ell} \big| \mathcal{F}_{t_{n}}  
\right] 
& \leq
\tfrac{
2^{\ell} h^{\frac{\ell}{2}} 
\left\|
\mathscr{P}(Y_{n})^{T} g\left(\mathscr{P}(Y_{n})\right) 
\right\|^{\ell} 
}{
\widetilde{L}^{\ell} (1 + \|\mathscr{P}(Y_{n}) \|^{2})^{\ell} 
}  \\
& \leq 
\tfrac{
2^{\ell} h^{\frac{\ell}{2}} 
\| \mathscr{P}(Y_{n})  \|^{\ell} 
\left[
L_{1}^{\frac{\ell}{2}} 
(1 + \| \mathscr{P}(Y_{n})\|^{2})^{\frac{\ell}{2}} 
+ C  h^{-\frac{\ell}{4}} 
(1 + \| \mathscr{P}(Y_{n})\|^{2})^{\frac{\ell}{2}-1} 
\right] 
}{
\widetilde{L}^{\ell} 
(
1 + \|\mathscr{P}(Y_{n}) \|^{2}
)^{\ell} 
} \\
&  = \tfrac{
2^{\ell} h^{\frac{\ell}{2}} 
\left[
L_{1}^{\frac{\ell}{2}} 
(1 + \| \mathscr{P}(Y_{n})\|^{2})^{\frac{\ell}{2}} 
\| \mathscr{P}(Y_{n})  \|^{\ell-2} 
\| \mathscr{P}(Y_{n})  \|^{2} 
+ C h^{-\frac{\ell}{4}} 
(
1 + \| \mathscr{P}(Y_{n})\|^{2}
)^{\frac{\ell}{2}-1} 
\| \mathscr{P}(Y_{n}) \|^{\ell}
\right] 
}{
\widetilde{L}^{\ell} (1 + \|\mathscr{P}(Y_{n}) \|^{2})^{\ell} 
}.
\end{aligned}
\end{equation}
Bearing the fact from Lemma \ref{lemma:necessary-estimates} that $\| \mathscr{P}(Y_{n})  \|^{2} \leq h^{-\frac{1}{\gamma}}$ in mind, we deduce, for some constant $C=C(L_{1}, C_{f}, \ell)$,
\begin{equation} 
\begin{aligned}
\mathbb{E}
\left[
(I_{3})^{\ell} \big| \mathcal{F}_{t_{n}}  
\right] 
\leq 
\tfrac{
Ch^{\frac{\ell}{2} 
- \frac{1}{\gamma}} + C h^{\frac{\ell}{4}}
}{
\widetilde{L}^{\ell} (1 + \|\mathscr{P}(Y_{n}) \|^{2})
} 
\leq 
\tfrac{ C h^{\frac{\ell}{4}}}{
\widetilde{L}^{\ell} (1 + \|\mathscr{P}(Y_{n}) \|^{2})
},
\end{aligned}
\end{equation}
Using the Young inequality yields, for some positive constants $\epsilon_{\ell} \in (0, (2\ell-1)^{\ell}/(2\ell-1)!!]$, $\ell \in [4,p]\cap \mathbb{N}$,
\begin{equation} 
\begin{aligned}
\mathbb{E}
\left[
\Xi_{n+1}^{\ell} \big| \mathcal{F}_{t_{n}} 
\right] 
= \mathbb{E}
\left[ 
(I_{1}+I_{2}+I_{3})^{\ell} \big| \mathcal{F}_{t_{n}}  
\right] 
\leq (1+\epsilon_{\ell}) 
\mathbb{E}
\left[
|I_{1}|^{\ell} \big| \mathcal{F}_{t_{n}} 
\right] 
+ (1+\tfrac{1}{\epsilon_{\ell}})
\mathbb{E}
\left[
|I_{2}+I_{3}|^{\ell} \big| \mathcal{F}_{t_{n}} 
\right].
\end{aligned}
\end{equation}
In light of the estimates \eqref{equation:estimate-of-I1}-\eqref{equation:estimate-of-I2} and \eqref{equation:estimate-of-I3-p} with the elementary inequality, we obtain that, for some constant $C_{\epsilon_{\ell}}=C(L_{1}, C_{f}, \ell)$,
\begin{equation} \label{equation:first-estimate-for-xi-when-l-is-bigger-than-3}
\begin{aligned}
\mathbb{E}\left[
\Xi_{n+1}^{\ell} \big| \mathcal{F}_{t_{n}} 
\right] 
\leq  (1+\epsilon_{\ell})
\tfrac{
(2\ell-1)!! \ (1 + \epsilon_{1})^{\ell} L_{1}^{\ell}  
h^{\ell}
}{
(2p_{0} - 1)^{\ell}\widetilde{L}^{\ell}
} 
+\tfrac{C_{\epsilon_{\ell}} h} {
\widetilde{L}^{\ell} (1 + \|\mathscr{P}(Y_{n}) \|^{2})
}.
\end{aligned}
\end{equation}
We would like to mention that the following inequality holds for any $\ell \in [4,p]\cap \mathbb{N}$ and $\epsilon_{1}\in (0, (2p_{0}-2p)/(2p-1)]$,
\begin{equation} 
\begin{aligned}
(1+\epsilon_{\ell})(2\ell-1)!! 
(1+\epsilon_{1})^{\ell} 
\leq (2\ell-1)^{\ell} (1+\epsilon_{1})^{\ell} 
\leq (2p_{0}-1)^{\ell} . 
\end{aligned}
\end{equation}


Therefore, the estimate \eqref{equation:first-estimate-for-xi-when-l-is-bigger-than-3} can be rewritten as 
\begin{equation} \label{equation:second-estimate-for-xi-when-l-is-bigger-than-3}
\begin{aligned}
\mathbb{E}
\left[
\Xi_{n+1}^{\ell} \big| \mathcal{F}_{t_{n}} 
\right] 
\leq  
\tfrac{ L_{1}^{\ell}   h^{\ell}}{\widetilde{L}^{\ell}} 
+\tfrac{C_{\epsilon_{\ell}} h} 
{\widetilde{L}^{\ell} (1 + \|\mathscr{P}(Y_{n}) \|^{2})}.
\end{aligned}
\end{equation}
\end{description}
Combining \textbf{Step I}$\sim $\textbf{Step IV} to show that, for some constant $C_{\epsilon_{\ell}}=C(L_{1}, C_{f}, p)$,
\begin{equation} 
\begin{aligned}
&\mathbb{E}\left[ (1 + \Xi_{n+1})^{p} \big| \mathcal{F}_{t_{n}} \right]\\ 
&\leq 1 
+ 
ph \tfrac{
2  \left\langle 
\mathscr{P}(Y_{n}) ,
f\left(\mathscr{P}(Y_{n})\right) 
\right\rangle 
+ \left[
(2p-2) \tfrac{
\left[1-(1-\theta)\lambda_{1}h \right]^{2}
}{
1-2(1-\theta)\lambda_{1}h
} +1+ \epsilon_{1} 
\right] 
\left\|  g\left(\mathscr{P}(Y_{n})\right) \right\|^{2} 
}{
\widetilde{L} (1 + \|\mathscr{P}(Y_{n}) \|^{2})
} 
+ \sum_{\ell=2}^{p} 
\mathscr{C}_{p}^{\ell} \ 
\tfrac{L_{1}^{\ell} h^{\ell}}{\widetilde{L}^{\ell}} 
+ \tfrac{C_{\epsilon_{\ell}} h}{
\widetilde{L}^{p} (1 + \|\mathscr{P}(Y_{n}) \|^{2})
}.  \\
\end{aligned}
\end{equation}
Moreover, we can choose a appropriate $h$ such that 
\begin{equation}
h\in \left(
0, \tfrac{p_{0}-p}{(1-\theta)\lambda_{1}(2p_{0}-p-1)}
\right)
\end{equation}
to make sure
\begin{equation}
2p_{0}-1 >  (2p-2) 
\tfrac{1-(1-\theta)\lambda_{1}h }
{1-2(1-\theta)\lambda_{1}h}
+1 
\geq (2p-2) 
\tfrac{
\left[1-(1-\theta)\lambda_{1}h \right]^{2}
}{1-2(1-\theta)\lambda_{1}h} +1,
\end{equation}
which leads to the following estimate by Assumption \ref{assumption:coercivity-condition-for-drift-and-diffusion},
\begin{equation} 
\begin{aligned}
\widetilde{L}^{p}
\mathbb{E}\left[
(1 + \Xi_{n+1})^{p} \big| \mathcal{F}_{t_{n}} 
\right] 
\leq 1+ \sum_{\ell=1}^{p} 
\mathscr{C}_{p}^{\ell}  
\widetilde{L}^{p-\ell}   L_{1}^{\ell}  h^{\ell} 
+ \tfrac{ C_{\epsilon_{\ell}}h}{ 1 + \|\mathscr{P}(Y_{n}) \|^{2}} 
= \left(
\widetilde{L} + L_{1}h 
\right)^{p} 
+\tfrac{
C_{\epsilon_{\ell}}h
}{
1 + \|\mathscr{P}(Y_{n}) \|^{2}
}.
\end{aligned}
\end{equation}
Hence, we deduce that
\begin{equation} \label{eq:estimate-of-the-first-part-I1}
\begin{aligned}
\mathbb{I}_{1} 
\leq 
\left(
\widetilde{L} + L_{1}h 
\right)^{p} 
(
1 + \|\mathscr{P}(Y_{n}) \|^{2}
)^{p} 
+ C_{\epsilon_{\ell}}h (1 + \|\mathscr{P}(Y_{n}) \|^{2})^{p-1}.
\end{aligned}
\end{equation}
\noindent \textbf{For the estimate of $\mathbb{I}_{2}$:}

For the estimate of $\mathbb{I}_{2}$, the key point is to get the estimate of $\widetilde{L}^{i} \mathbb{E} \left[(1 + \Xi_{n+1}) ^{i} \big| \mathcal{F}_{t_{n}} \right]$, $i \in (1,p)\in \mathbb{N}$, which is uniform bounded with the same analysis as the estimate of $\mathbb{I}_{1}$, i.e., there exists some  positive constant $C=C(L_{1}, C_{f}, p)$ such that,
\begin{equation}
\sum_{i=0}^{p-1}\widetilde{L}^{i} 
\mathbb{E} 
\left[
(1 + \Xi_{n+1}) ^{i} \big| \mathcal{F}_{t_{n}} 
\right] 
\leq C,
\end{equation}
leading to
\begin{equation} \label{eq:estimate-of-the-first-part-I2}
\begin{aligned}
\mathbb{I}_{2} 
\leq C_{\epsilon_{1}}h 
(1 + \|\mathscr{P}(Y_{n}) \|^{2})^{p-1}.
\end{aligned}
\end{equation}

\noindent \textbf{Combining the estimates of $\mathbb{I}_{1}$ and $\mathbb{I}_{2}$:}

Taking the estimates of $\mathbb{I}_{1}$ and $\mathbb{I}_{2}$ into \eqref{equatoin:basic-expansion-of-p-th-moment}, for some constant $C_{\epsilon_{1}, \epsilon_{\ell}}=C(\lambda_{1}, \lambda_{d}, \theta, L_{1}, C_{f},p)$, we recall $\widetilde{L}=1-2(1-\theta)\lambda_{1}h$ to show
\begin{equation} \label{equation:final-conditional-expectations-version}
\begin{aligned}
(1 + 2 \theta \lambda_{1} h )^{p} 
&\mathbb{E}
\left[ 
(1+\|Y_{n+1} \|^{2})^{p} \big| \mathcal{F}_{t_{n}} 
\right] \\
&\leq \left[
1-2(1-\theta)\lambda_{1}h + L_{1}h 
\right]^{p} 
(
1 + \|\mathscr{P}(Y_{n}) \|^{2}
)^{p} 
+  C_{\epsilon_{1}, \epsilon_{\ell}} h 
(
1 + \|\mathscr{P}(Y_{n}) \|^{2}
)^{p-1}.   \\
\end{aligned}
\end{equation}
For $2\lambda_{1} > L_{1}$, we take expectations on both sides of \eqref{equation:final-conditional-expectations-version} with Lemma \ref{lemma:necessary-estimates} and the Young inequality to show that, for some $\epsilon_{2} > 0$,
\begin{equation}
\begin{aligned}
\mathbb{E}
\left[
(1+\|Y_{n+1} \|^{2})^{p}  
\right]
&\leq 
\left(
\tfrac{ 1-2(1-\theta)\lambda_{1}h +  L_{1}h }
{1 + 2\theta \lambda_{1} h}
\right) ^{p} 
\mathbb{E}
\left[
(1 + \|\mathscr{P}(Y_{n}) \|^{2})^{p} 
\right]  
+  C_{\epsilon_{1}, \epsilon_{\ell}} h  
\mathbb{E}
\left[
(1 + \|\mathscr{P}(Y_{n}) \|^{2})^{p-1} 
\right]\\
& = \big(
1- \tfrac{2\lambda_{1} - L_{1}}{1+2\theta \lambda_{1}h} h 
\big) 
\mathbb{E}\left[
(1 + \|\mathscr{P}(Y_{n}) \|^{2})^{p} 
\right] 
+ C_{\epsilon_{1}, \epsilon_{\ell}} h  
\mathbb{E}\left[
(1 + \|\mathscr{P}(Y_{n}) \|^{2})^{p-1} 
\right] \\
& \leq  
\left(
1-\tfrac{2\lambda_{1} - L_{1}}{1+2\theta \lambda_{1}h}h 
+  \tfrac{p-1}{p}\epsilon_{2} C_{\epsilon_{1}, \epsilon_{\ell}} h 
\right) 
\mathbb{E}\left[
(1 + \|Y_{n} \|^{2})^{p} 
\right] 
+ \tfrac{\epsilon_{2}^{1-p}}{p}  
C_{\epsilon_{1}, \epsilon_{\ell}} h  .
\end{aligned}
\end{equation}
Then we can choose a suitable $\epsilon_{2}$ to ensure that 
\begin{equation}
\widetilde{C} 
:= \tfrac{2\lambda_{1} - L_{1}}{1+2\theta \lambda_{1}h}
- \tfrac{p-1}{p}\epsilon_{2} 
C_{\epsilon_{1}, \epsilon_{\ell}} >0.
\end{equation}
Therefore, for some constant $C_{\epsilon_{1}, \epsilon_{2}, \epsilon_{\ell}}=C(\lambda_{1}, \lambda_{d}, \theta, L_{1}, C_{f},p)$, we get
\begin{equation}
\begin{aligned}
\mathbb{E}
\left[
(1+\|Y_{n+1} \|^{2})^{p}  
\right] 
&\leq 
\left( 1-\widetilde{C}h \right) 
\mathbb{E}
\left[ 
(1 + \|Y_{n} \|^{2})^{p} 
\right] 
+ C_{\epsilon_{1}, \epsilon_{2}, \epsilon_{\ell}} h \\
&= \left(
1-\widetilde{C}h 
\right)^{n+1} 
\mathbb{E}
\left[
(1 + \|x_{0} \|^{2})^{p} 
\right] 
+ \sum_{i=0}^{n}
\left(
1-\widetilde{C}h 
\right)^{i} 
C_{\epsilon_{1}, \epsilon_{2}, \epsilon_{\ell}} h \\
& \leq e^{-\widetilde{C} t_{n+1}}
\mathbb{E}
\left[ 
(1 + \|x_{0} \|^{2})^{p} 
\right] 
+ \tfrac{C_{\epsilon_{1}, \epsilon_{2}, \epsilon_{\ell}}}{\widetilde{C}},
\end{aligned}
\end{equation}
where we have used the fact that for any $x>0$, $1-x \leq e^{-x}$.
The proof is completed.
\end{proof}
We remark that to verify the existence and uniqueness of the invariant measure of the LTPE method \eqref{introduction:LTPE-scheme}, the uniform estimate of the second order moment (i.e. \eqref{equation:second-order-moment-estimate-of-LTPE-method} in Lemma \ref{lemma:uniform-moments-bound-of-the-LTPE-method}) is enough. The estimate of the  $2p-$th  order moment (i.e. \eqref{equation:2p-thmoment-estimate-of-LTPE-method} in Lemma \ref{lemma:uniform-moments-bound-of-the-LTPE-method}) of the LTPE method \eqref{introduction:LTPE-scheme} is essential to the error analysis that follows.

The contractivity of the LTPE method \eqref{introduction:LTPE-scheme} follows directly from Lemma \ref{lemma:uniform-moments-bound-of-the-LTPE-method} and Lemma \ref{lemma:necessary-estimates}.
\begin{lemma} \label{lemma:contractivity-of-the-theta-linear-projected-Euler-method}
(Contractivity of the theta-linear-projected Euler method.)
Consider the following 
pair of solutions of LTPE method \eqref{introduction:LTPE-scheme} 
 with a parameter $\theta \in [0 , 1]$ driven by the same Brownian motion:
\begin{equation}
\begin{aligned}
Y_{n+1}^{(1)} - \theta A Y_{n+1}^{(1)} h 
&= \mathscr{P}(Y^{(1)}_{n}) 
+ (1- \theta) A \mathscr{P}(Y^{(1)}_{n}) h 
+ f\big(\mathscr{P}(Y^{(1)}_{n})\big) h 
+  g\big(\mathscr{P}(Y^{(1)}_{n})\big) \Delta W_{n}, \
Y_{0}^{(1)} = x_{0}^{(1)};  \\
Y_{n+1}^{(2)} - \theta A Y^{(2)}_{n+1} h
&= \mathscr{P}(Y^{(2)}_{n}) 
+ (1-\theta)A \mathscr{P}(Y^{(2)}_{n}) h 
+ f\big(\mathscr{P}(Y^{(2)}_{n})\big) h 
+  g\big(\mathscr{P}(Y^{(2)}_{n})\big) \Delta W_{n}, 
\ Y_{0}^{(2)} = x_{0}^{(2)},
\end{aligned}
\end{equation}
where 
$h$ is the uniform timestep with
\begin{equation}
h \in 
\left(
0, 
\min\left\{
\tfrac{\kappa^{2}(2\lambda_{1}-L_{2})}{(1- \theta)^{2}\lambda^{2}_{d}}, 
\tfrac{(1-\kappa)^{2\gamma}(2\lambda_{1}-L_{2})^{\gamma}}{(\lambda_{f})^{2\gamma}},
1 
\right\} \right), \quad \kappa\in (0,1).
\end{equation}
The constant $\lambda_{f}$ depends only on the drift $f$, denoted in Lemma \ref{lemma:necessary-estimates}. 
In addition,
let Assumptions \ref{assumption:one-side-Lipschitz-condition-for-linear-operator},  \ref{assumption:coupled-monotoncity-for-drift-and-diffusion} and \ref{assumption:growth-condition-of-frechet-derivatives-of-drift-and-diffusion} hold for $2\lambda_{1}>L_{2}$,
then there exists a positive constant $\widetilde{C}_{1}$ such that, for any $n \in\{0,1,2, \ldots, N\}$, $N\in \mathbb{N}$ and $t_{n}=nh$,
\begin{equation}
\mathbb{E} \left[  \| Y_n^{(1)} - Y_n^{(2)} \|^{2}\right] \leq e^{- \widetilde{C}_{1} t_{n}} \mathbb{E} \big[  \big\| x_0^{(1)} - x_0^{(2)} \big\|^{2} \big].
\end{equation}

\end{lemma}
The proof of Lemma \ref{lemma:contractivity-of-the-theta-linear-projected-Euler-method} is deferred to Appendix \ref{proof-of-lemma:contractivity-of-the-theta-linear-projected-Euler-method}.
\begin{proof} [Proof of Theorem \ref{theorem:invariant-measure-of-LTPE-method}]
With Lemma \ref{lemma:uniform-moments-bound-of-the-LTPE-method} in mind, the existence of the invariant measure $\widetilde{\pi}$ admitted by the LTPE scheme \eqref{introduction:LTPE-scheme} is obtained by Krylov-Bogoliubov theorem \cite{da2006introduction}.
Further, the proof of the uniqueness of such invariant measure $\widetilde{\pi}$ follows almost the same idea quoted from Theorem 7.9 in \cite{li2019explicit}, which is a consequence of Lemma \ref{lemma:contractivity-of-the-theta-linear-projected-Euler-method}, so that we omit it here.
Then, using Lemma \ref{lemma:contractivity-of-the-theta-linear-projected-Euler-method} and the Chapman-Kolmogorov equation yields, for $\varphi \in C^{1}_{b}(\mathbb{R}^{d})$, 
\begin{equation}
\begin{aligned}
\left|
\mathbb{E} 
\left[
\varphi\left(Y^{x_{0}}_{n}\right)
\right]
-\int_{\mathbb{R}^d} 
\varphi(x) \widetilde{\pi}(\dd x)
\right| 
&= 
\left|
\int_{\mathbb{R}^d} 
\mathbb{E} 
\left[
\varphi\left(Y^{x_{0}}_{n}\right)
\right]
\widetilde{\pi}(\dd x)
-
\int_{\mathbb{R}^d} 
\mathbb{E} 
\left[
\varphi\left(Y^{x}_{n}\right)
\right]
\widetilde{\pi}(\dd x)
\right| \\
&\leq 
\|\varphi \|_{1} 
\int_{\mathbb{R}^d}
\mathbb{E} 
\left[
\left\|
Y^{x_{0}}_{n} - Y^{x}_{n} 
\right\| 
\right]
\widetilde{\pi}(\dd x)
\\
&\leq 
C e^{-\frac{ \widetilde{C}_{1}}{2} t}
\left(
1
+\mathbb{E} \left[\|x_{0}\|^2 \right]
\right), 
\quad \forall t \in [0,T].
\end{aligned}
\end{equation}

\end{proof}
\section{Time-independent weak error analysis}\label{subsection: Time-independent weak error analysis}
Our aim is to estimate the error between the invariant measure $\pi$ and $\widetilde{\pi}$. i.e.
\begin{equation}
    \left|\int_{\mathbb{R}^d} \varphi(x) \pi(\dd x)-\int_{\mathbb{R}^d} \varphi(x) \widetilde{\pi}( \dd x)\right|.
\end{equation}
As we have claimed before, both $\{X_{t_{n}} \}_{n\in \mathbb{N}}$, defined by \eqref{eq:semi-linear-SODE}, and $\{Y_{n} \}_{n\in \mathbb{N}}$, defined by \eqref{introduction:LTPE-scheme}, are ergodic, namely
\begin{equation}
\lim_{N \rightarrow \infty} 
\tfrac{1}{N} 
\sum_{n=0}^{N-1} 
\mathbb{E}
\left[
\varphi \left(X_{t_{n}} \right) 
\right] = 
\int_{\mathbb{R}^d} \varphi(x) \pi(\dd x), 
\quad 
\lim_{N \rightarrow \infty}
\tfrac{1}{N} \sum_{n=0}^{N-1} 
\mathbb{E}
\left[
\varphi \left(Y_{k} \right) 
\right] 
= \int_{\mathbb{R}^d} \varphi(x) \widetilde{\pi}(\dd x), \quad a.s.
\end{equation}
Hence, the error estimate boils down to the time-independent weak convergence analysis of the LTPE method \eqref{introduction:LTPE-scheme} as follows,
\begin{equation} \label{equation:estimate-decomposition-of-two-measures}
\begin{aligned}
\left|
\int_{\mathbb{R}^d} \varphi(x) \pi(\dd x)
-\int_{\mathbb{R}^d} \varphi(x) \widetilde{\pi}( \dd x)
\right|  
&= 
\left|
\lim_{N \rightarrow \infty} 
\sum_{n=0}^{N-1} 
\Big(
\mathbb{E}\left[\varphi \left(X_{t_{n}} \right) \right] 
- \mathbb{E}\left[\varphi \left(Y_{n} \right) \right] 
\Big) 
\right|\\
& \leq 
\lim_{N \rightarrow \infty} 
\tfrac{1}{N} 
\sum_{n=0}^{N-1} 
\Big|
\mathbb{E}
\left[
\varphi \left(X_{t_{n}} \right) 
\right] 
-  
\mathbb{E}
\left[
\varphi \left(Y_{n} \right) 
\right] 
\Big|.
\end{aligned}
\end{equation}
In order to carry out the error analysis, we need some priori estimates and lemmas.
The key ingredient is to introduce $u: [0, T] \times \mathbb{R}^{d} \rightarrow \mathbb{R}$ defined by
\begin{equation} \label{equation:def-of-u}
u(t, x) 
:=  \mathbb{E}
\left[
\varphi\!\left(X^{x}_{t} \right) 
\right].
\end{equation}
where $\varphi \in C^{3}_{b}(\mathbb{R}^{d})$.
In what follows, we will show that
$u(\cdot,\cdot)$ is the unique solution of the associated Kolmogorov equations
as
\begin{equation} \label{equation:kolmogorov-equation}
\partial_{t} u(t,x) 
= Du(t,x) F(x) 
+ \tfrac{1}{2} 
\sum_{j=1}^{m} D^{2}u(t,x) 
\big(g_{j}(x) , g_{j}(x) \big),
\end{equation}
with initial condition $u(0,\cdot) = \varphi(\cdot)$,  where we denote that $F(x):= Ax + f(x)$. To examine the regularity of $u$, we need the following properties.



For the matrix $A \in \mathbb{R}^{d\times d}$, it is apparent that
\begin{equation}
D(Ax)   v_{1} = Av_{1}, 
\quad \forall x, v_{1} \in \mathbb{R}^{d},
\end{equation}
and for $i \in [2,\infty) \cap \mathbb{N}$,
\begin{equation}
D^{i}(Ax)   (
v_{1}, \cdots, v_{i}
) 
= 0, 
\quad \forall x, v_{1}, \cdots, v_{i} \in \mathbb{R}^{d}.
\end{equation}
Moreover, for convenience, we denote a mapping $\mathcal{P}_{\cdot}(\cdot, \cdot) : \mathbb{R}\times  \mathbb{R}^{d} \times  \mathbb{R}^{d} \rightarrow [1, \infty)$ as
\begin{equation}\label{eqn:Pgammaxx}
\mathcal{P}_{\bar{\gamma}}(x, \tilde{x} ) 
:= \max\left\{
1, (1+\|x\|+\|\tilde{x}\|)^{\bar{\gamma}} 
\right\}, 
\quad  
\forall \bar{\gamma} \in \mathbb{R}, \ 
\forall x, \tilde{x} \in \mathbb{R}^{d}.
\end{equation}
In particular, let $\mathcal{P}_{\cdot}(\cdot):  \mathbb{R} \times  \mathbb{R}^{d} \rightarrow [1, \infty) $ be defined as
\begin{equation}\label{eqn:Pgammax}
\mathcal{P}_{\bar{\gamma}}(x) 
:= \mathcal{P}_{\bar{\gamma}}(x, 0) 
= \max\left\{
1, (1+\|x\|)^{\bar{\gamma}} 
\right\}, 
\quad  \forall \bar{\gamma} \in \mathbb{R}, \ 
\forall x \in \mathbb{R}^{d}.
\end{equation}
Obviously, these mappings are non-decreasing with respect to $\bar{\gamma}$.
Hence, it follows from Assumption \ref{assumption:growth-condition-of-frechet-derivatives-of-drift-and-diffusion} and its consequences that, for $j\in \{1, \dots, m \}$,
\begin{equation}\label{eqn:D2FPestimate}
\begin{aligned}
\left\|
D^{2}F(x) ( v_{1}, v_{2})- D^{2}F(\tilde{x})( v_{1},  v_{2}) 
\right\| 
&\leq C
\mathcal{P}_{\gamma-3}(x, \tilde{x}) 
\cdot \|x-\tilde{x}\| \cdot \|v_{1} \| \cdot \|v_{2} \|,\quad \forall x, \tilde{x}, v_{1}, v_{2}\in \mathbb{R}^{d},\\
\left\|D^{2}F(x) ( v_{1}, v_{2}) \right\| &\leq C\mathcal{P}_{\gamma-2}(x)  \cdot \|v_{1} \| \cdot \|v_{2} \|,\quad \forall x, v_{1}, v_{2}\in \mathbb{R}^{d}, \\
\end{aligned}
\end{equation}
which directly implies
\begin{equation}\label{eq:growth-of-the-diff-drift-F}
\begin{aligned}
\left\|DF(x) v_{1}- DF(\tilde{x}) v_{1} \right\| &\leq C\mathcal{P}_{\gamma-2}(x, \tilde{x}) \cdot \|x-\tilde{x}\| \cdot \|v_{1} \|,\quad \forall x, \tilde{x}, v_{1}\in \mathbb{R}^{d}, \\
\left\|DF(x) v_{1}\right\| &\leq C_{A}\mathcal{P}_{\gamma-1}(x)  \cdot \|v_{1} \|, \quad \forall x,  v_{1}\in \mathbb{R}^{d},\\
\end{aligned}
\end{equation}
and
\begin{equation}\label{equation:growth-of-the-drift-F}
\begin{aligned}
\left\|F(x) - F(\tilde{x})  \right\| &\leq C_{A}(1+\|x\|+\|\tilde{x}\|)^{\gamma-1} \|x-\tilde{x}\| , \quad \forall x, \tilde{x} \in \mathbb{R}^{d},\\
\left\|F(x)  \right\| &\leq C_{A}(1+\|x\|)^{\gamma}, \quad \forall x \in \mathbb{R}^{d}  .\\
\end{aligned}
\end{equation}
Correspondingly, the following estimates hold true, for $j\in \{1,\dots,m \}$,
\begin{equation}\label{eqn:D2GPestimate}
\begin{aligned}
\left\|D^{2}g_{j}(x) (v_{1}, v_{2})- D^{2}g_{j}(\tilde{x}) (v_{1}, v_{2}) \right\|^{2} &\leq C\mathcal{P}_{\gamma-5}(x, \tilde{x})  \cdot \|x-\tilde{x}\|^{2} \cdot \|v_{1} \|^{2} \cdot \|v_{2} \|^{2},\quad \forall x, \tilde{x}, v_{1}, v_{2}\in \mathbb{R}^{d},\\
\left\|D^{2}g_{j}(x) (v_{1}, v_{2}) \right\|^{2} &\leq C\mathcal{P}_{\gamma-3}(x)  \cdot \|v_{1} \|^{2} \cdot \|v_{2} \|^{2}, \quad \forall x, v_{1}, v_{2}\in \mathbb{R}^{d},\\
\end{aligned}
\end{equation}
which also shows
\begin{equation}
\begin{aligned}
\left\|Dg_{j}(x) v_{1}- Dg_{j}(\tilde{x}) v_{1} \right\|^{2} &\leq C\mathcal{P}_{\gamma-3}(x, \tilde{x})  \cdot \|x-\tilde{x}\|^{2} \cdot \|v_{1} \|^{2}, \quad \forall x, \tilde{x}, v_{1}\in \mathbb{R}^{d}. \\
\end{aligned}
\end{equation}
Besides, 
Assumptions \ref{assumption:one-side-Lipschitz-condition-for-linear-operator}, \ref{assumption:coupled-monotoncity-for-drift-and-diffusion} lead to, for some $p_{1} \geq 1$ with $2\lambda_{1} > L_{2}$,
\begin{equation} \label{eq:enhanced-coupled-condition}
2\langle DF(x)y, y\rangle 
+ (2p_{1}-1) \sum^{m}_{j=1}
\|Dg_{j}(x)y \|^{2} 
\leq - \alpha \| y\|^{2}, 
\quad \forall x, y \in \mathbb{R}^{d},
\end{equation}
where $\alpha:= 2\lambda_{1}-L_{2}>0$.
For random functions,  let us
introduce the mean-square differentiability,
quoted from \cite{wang2021weak1},
as follows.
\begin{definition}(Mean-square differentiable)
Let $\Psi: \Omega \times \mathbb{R}^{d} \rightarrow \mathbb{R}$ and $\psi_{i}: \Omega \times \mathbb{R}^{d} \rightarrow \mathbb{R}$ be random functions satisfying
\begin{equation}
\lim _{\tau \rightarrow 0} 
\mathbb{E}
\left[
\left|
\tfrac{1}{\tau}
\left[
\Psi\left(x+\tau e_i\right)-\Psi(x)
\right]
-\psi_i(x)
\right|^2
\right]=0, 
\quad \forall i \in\{1,2, \cdots, d\},
\end{equation}
where $e_{i}$ is the  unit vector in $\mathbb{R}^{d}$ with the $i-$th element being  $1$.  Then $\Psi$ is called to be mean-square differentiable,  with $\psi = (\psi_{1}, \dots, \psi_{d})$  being the derivative (in the mean-square differentiable sense) of $\Psi$ at $x$.
Also denoting $\mathcal{D}_{(i)} \Psi = \psi_{i}$  and  $\mathcal{D} \Psi(x) =\psi $.
\end{definition}
The above definition can be generalized to vector-valued functions in a component-wise manner.
Now we are in the position to derive the uniform estimate of the 
derivatives of $\{X^{x_{0}}_{t}\}_{t\in[0,T]}$ of  \eqref{eq:semi-linear-SODE} in the mean-square differentiable sense.  Here for each $t$ we take the function $X^{\cdot}_t: \mathbb{R}^d\to \mathbb{R}^d$, and write its derivative as $\mathcal{D}X^{x}_{t}\in L(\mathbb{R}^d,\mathbb{R}^d)$. Higher order derivatives $\mathcal{D}^2X^{x}_{t}$ and $\mathcal{D}^3X^{x}_{t}$ can be defined similarly.


\begin{lemma} \label{lemma:differentiability-of-solutions}
Consider the SDE \eqref{eq:semi-linear-SODE} subject to Assumptions \ref{assumption:one-side-Lipschitz-condition-for-linear-operator}-\ref{assumption:growth-condition-of-frechet-derivatives-of-drift-and-diffusion} with $2\lambda_{1} > L_{2}$. Then the solution $\{X_{t}\}_{t\in [0,T]}$ of \eqref{eq:semi-linear-SODE} is three times mean-square differentiable.
Moreover, recall $p_1$ given in Assumption \ref{assumption:coupled-monotoncity-for-drift-and-diffusion}, for any $q_{1}\in [1,p_{1}]$, $q_{2} \in [1,q_{1})$, 
and some random variables
$v_{1} \in L^{2\max\{q_{1}, \rho_{2}q_{2} \}}(\Omega, \mathbb{R}^{d})$,
$v_{2} \in L^{2\max\{\rho_{3}q_{2}, \rho_{3}\rho_{5} \}}(\Omega, \mathbb{R}^{d})$,
$v_{3} \in L^{2\rho_{3} }(\Omega, \mathbb{R}^{d})$,
where $\rho_{1}, \rho_{2}, \rho_{3}, \rho_{4}, \rho_{5}, \rho_{6} >1$ satisfying $1/\rho_{1} + 1/\rho_{2} + 1/\rho_{3}=1$,
$1/\rho_{4} + 1/\rho_{5} + 1/\rho_{6}=1$
and 
$p_{0}$ in Assumption \ref{assumption:coercivity-condition-for-drift-and-diffusion} fulfilling
\begin{equation} \label{eq:range-of-p0-in-lemma:differentiability-of-solutions}
p_{0} \in \big[
\max\{\rho_{1}q_{2},  \rho_{2} \} \times (\gamma-2)
,
\infty
\big) \cap 
[1,\infty),
\end{equation}
such that
\begin{equation}
\left\| 
\mathcal{D}X^{x}_{t}  v_{1}
\right\|_{L^{2q_{1}}(\Omega, \mathbb{R}^{d})} 
\leq Ce^{-\alpha_{1}t}  
\left\|  
v_{1}
\right\|_{L^{2q_{1}}(\Omega, \mathbb{R}^{d})}, \notag
\end{equation}
\begin{equation}
\left\|
\mathcal{D}^{2}X^{x}_{t}   (v_{1}, v_{2})
\right\|_{L^{2q_{2}}(\Omega, \mathbb{R}^{d})}
\leq C e^{-\alpha_{2}t}
\sup_{r \in [0,T]}
\left\| 
\mathcal{P}_{\gamma-2}(X^{x}_r) 
\right\|_{L^{2\rho_{1} q_{2}  }(\Omega, \mathbb{R})} 
\left\| v_{1} \right\|_{L^{2\rho_{2} q_{2}} (\Omega, \mathbb{R}^{d})}
\left\| v_{2} \right\|_{L^{2\rho_{3} q_{2}} (\Omega, \mathbb{R}^{d})} , \notag
\end{equation}
\begin{equation}
\begin{aligned}
\left\|
\mathcal{D}^{3}X^{x}_{t}  (v_{1}, v_{2}, v_{3})
\right\|_{L^{2}(\Omega, \mathbb{R}^{d})} 
&\leq Ce^{-\alpha_{3}t} \sup_{r \in [0,T]}
\left\| 
\mathcal{P}_{\gamma-2}(X^{x}_r) 
\right\|_{L^{2\max\{ \rho_{1}, \rho_{3}\rho_{4} \} }(\Omega, \mathbb{R})} \times \\
&\hspace{10em}
\|v_{1} \|_{L^{2\rho_{2}} (\Omega, \mathbb{R}^{d})} 
\|v_{2} \|_{L^{2\rho_{3}\rho_{5}} (\Omega, \mathbb{R}^{d})} 
\|v_{3} \|_{L^{2\rho_{3}\rho_{6}} (\Omega, \mathbb{R}^{d})} ,
\end{aligned}
\end{equation}
where $\alpha_{1}=\alpha/2$, 
$\alpha_{2}=(q_{2}\alpha - \tilde{\epsilon}_{1})/q_{2}$, 
$\tilde{\epsilon}_{1}\in (0,q_{2}\alpha)$ and
$\alpha_{3}=\alpha-\tilde{\epsilon}_{4}$, $\tilde{\epsilon}_{4}\in (0,\alpha)$.

\end{lemma}
The proof of Lemma \ref{lemma:differentiability-of-solutions} will be presented in Appendix \ref{proof-of-lemma:differentiability-of-solutions}. As a consequence of Lemma \ref{lemma:differentiability-of-solutions}, the uniform estimate of the derivatives of $u(t.\cdot)$ is obtained by the following lemma.
\begin{lemma} \label{lemma:estimate-of-u-and-its-derivatives}
For any $x \in \mathbb{R}^{d}$ and 
some random variables
$v_{1} \in L^{2\rho_{2}}(\Omega, \mathbb{R}^{d})$,
$v_{2} \in L^{2\rho_{3}\rho_{5}}(\Omega, \mathbb{R}^{d})$,
$v_{3} \in L^{2\rho_{3}\rho_{6} }(\Omega, \mathbb{R}^{d})$,
where $\rho_{1}, \rho_{2}, \rho_{3}, \rho_{4}, \rho_{5}, \rho_{6} >1$ satisfying $1/\rho_{1} + 1/\rho_{2} + 1/\rho_{3}=1$.
Let Assumptions \ref{assumption:one-side-Lipschitz-condition-for-linear-operator}-\ref{assumption:growth-condition-of-frechet-derivatives-of-drift-and-diffusion}  be fulfilled with $2\lambda_{1} > \max\{L_{1}, L_{2}\}$ and
\begin{equation} \label{eq:range-of-p0-lemma:estimate-of-u-and-its-derivatives}
p_{0} \in 
\big[
\max\{\rho_{1}, \rho_{3}\rho_{4} \} \times (\gamma-2)
,
\infty \big)
\cap
[1,\infty),
\end{equation}
such that
\begin{equation}
\begin{aligned}
\left\|
Du(t,x)   v_{1} 
\right\|_{L^{1} (\Omega, \mathbb{R})}
&\leq Ce^{-\alpha_{1}t }
\|v_{1}\|_{L^{2} (\Omega, \mathbb{R}^{d})},  \\
\end{aligned}
\end{equation}
\begin{equation}
\begin{aligned}
\left\|
D^{2}u(t,x) (v_{1}, v_{2}) 
\right\|_{L^{1} (\Omega, \mathbb{R})} 
&\leq Ce^{-\tilde{\alpha}_{2}t} 
\sup_{r \in [0,T]}
\left\| 
\mathcal{P}_{\gamma-2}(X^{x}_{r}) 
\right\|_{L^{2\rho_{1} }(\Omega, \mathbb{R})}
\left\| 
v_{1} 
\right\|_{L^{2\rho_{2} } (\Omega, \mathbb{R}^{d})}
\left\| 
v_{2} 
\right\|_{L^{2\rho_{3} } (\Omega, \mathbb{R}^{d})},\\ \notag
\end{aligned}
\end{equation}
and
\begin{equation}
\begin{aligned}
\left\| 
D^{3} u(t,x)   (v_{1}, v_{2}, v_{3})
\right\|_{L^{1} (\Omega, \mathbb{R})} 
&\leq C e^{-\tilde{\alpha}_{3}t} 
\sup_{r \in [0,T]}
\left\| 
\mathcal{P}_{\gamma-2}(X^{x}_{r}) 
\right\|_{L^{2\max\{ \rho_{1}, \rho_{3}\rho_{4} \} }(\Omega, \mathbb{R})}  \times \\
& \hspace{10em} 
\|v_{1} \|_{L^{2\rho_{2}} (\Omega, \mathbb{R}^{d})} 
\|v_{2} \|_{L^{2\rho_{3}\rho_{5}} (\Omega, \mathbb{R}^{d})} 
\|v_{3} \|_{L^{2\rho_{3}\rho_{6}} (\Omega, \mathbb{R}^{d})},\\
\end{aligned}
\end{equation}
where $\alpha_{1}$, $\tilde{\alpha}_{2}$ and $\tilde{\alpha}_{3}$ are positive constants, with the latter two depending on $\alpha_{1}$, $\alpha_{2}$ and $\alpha_{3}$ defined as Lemma \ref{lemma:differentiability-of-solutions}, i.e.
\begin{equation}
\tilde{\alpha}_{2}:=\min\{2\alpha_{1}, \alpha_{2} \}, 
\quad 
\tilde{\alpha}_{3}:= \min\{3\alpha_{1}, \alpha_{1}+\alpha_{2}, \alpha_{3} \}.
\end{equation}
\end{lemma}
\begin{remark}
Bearing Lemma \ref{lemma:estimate-of-u-and-its-derivatives} in mind,
we obtain that given the test function $\varphi \in C^{3}_{b}(\mathbb{R}^{d})$ and $t>0$, the function $u(t,\cdot) \in C^{3}_{b}(\mathbb{R}^{d})$.
Then $u(t,x)$ is the unique solution of \eqref{equation:kolmogorov-equation} (see Theorem 1.6.2 in \cite{cerrai2001second}).
\end{remark}

The proof of Lemma \ref{lemma:estimate-of-u-and-its-derivatives} can be seen in Appendix \ref{proof-of-lemma:estimate-of-u-and-its-derivatives}. Moreover,
Lemma \ref{lemma:estimate-of-u-and-its-derivatives} apparently yields the contractivity of $u(t,\cdot)$, which can also be derived by Lemma \ref{lemma:contractivity-of-sde}. Thus, one can have the following result.
\begin{corollary} \label{lemma:contractivity-of-u}
Let Assumptions \ref{assumption:one-side-Lipschitz-condition-for-linear-operator}-\ref{assumption:growth-condition-of-frechet-derivatives-of-drift-and-diffusion} hold with $2\lambda_{1} > \max\{L_{1}, L_{2}\}$, and recall that $\alpha_1=\alpha/2$, then 
\begin{equation}
\begin{aligned}
\left\|
u(t, \zeta_{1}) - u(t, \zeta_{2}) 
\right\|_{L^{1} (\Omega, \mathbb{R})}
&\leq Ce^{-\alpha_{1} t} 
\|\zeta_{1} - \zeta_{2} \|_{L^{2} (\Omega, \mathbb{R}^{d})},
\end{aligned}
\end{equation}
\end{corollary}
Before proceeding further, there is no guarantee that the LTPE method \eqref{introduction:LTPE-scheme} is continuous in the whole time interval since the numerical solutions are prevented from leaving a ball, whose  radius depends on the timestep size, in each iteration.
To address this issue and fully exploit the Kolmogorov equations, we recall the continuous version of the LTPE scheme \eqref{introduction:LTPE-scheme} defined in \eqref{intro:continuous-version-of-the-numerical-scheme}, ie,  $\{\mathbb{Z}^{n}(t)\}_{t\in  [t_{n}, t_{n+1}]}$, $n\in \{0,1,\dots,N-1 \}$, $N\in \mathbb{N}$.
It is time to show the next lemma concerning some regular estimates of this process.
\begin{lemma} \label{lemma:holder-continuity-of-the-process-z}
Let Assumptions \ref{assumption:one-side-Lipschitz-condition-for-linear-operator},  \ref{assumption:coercivity-condition-for-drift-and-diffusion},  \ref{assumption:growth-condition-of-frechet-derivatives-of-drift-and-diffusion}  hold with $2\lambda_{1}>L_{1}$. 
For $p \in [1,p_{0})$ and $t\in  [t_{n}, t_{n+1}]$, $n\in \{0,1,\dots,N-1 \}$, $N\in \mathbb{N}$, 

\begin{equation} \label{equation:moments-bound-of-the-auxiliary-process-Z}
\mathbb{E}
\left[
\big\| \mathbb{Z}^{n}(t) \big\|^{2p} 
\right] \leq C_{A}.
\end{equation}
Moreover,
for $p \in [1,p_{0}/\gamma]$ and $s, t \in [t_{n}, t_{n+1}]$, then
\begin{equation} \label{equation:holder-continuty-of-the-auxiliary-process-Z}
\begin{aligned}
\mathbb{E}
\left[
\big\|
\mathbb{Z}^{n}(t) - \mathbb{Z}^{n}(s)  
\big\|^{2p}
\right] 
\leq C_{A} |t-s |^{p}.
\end{aligned}
\end{equation}
\end{lemma}
%
The proof of Lemma \ref{lemma:holder-continuity-of-the-process-z} is presented in Appendix \ref{proof-of-lemma:holder-continuity-of-the-process-z}.
At this time, we would like to present the error estimate between the random variable $\zeta \in \mathbb{R}^{d}$ and the projected one $\mathscr{P}(\zeta)\in \mathbb{R}^{d}$, which is defined by \eqref{introduction:LTPE-scheme}.
\begin{lemma} \label{lemma:error-estimate-between-x-and-projected-x}
Recall $\gamma$ given in Assumption \ref{assumption:growth-condition-of-frechet-derivatives-of-drift-and-diffusion}, let $\zeta \in L^{8\gamma+2}(\Omega, \mathbb{R}^{d})$ and let $\mathscr{P}(\zeta)$ be defined as \eqref{introduction:LTPE-scheme}, then 
\begin{equation}
\mathbb{E}
\left[
\|\zeta-\mathscr{P}(\zeta) \|^{2}
\right] 
\leq Ch^{2} 
\mathbb{E}
\left[ 
\|\zeta \|^{8\gamma+2}
\right].
\end{equation}
\end{lemma}
The proof of Lemma \ref{lemma:error-estimate-between-x-and-projected-x} can be found in Appendix \ref{subsection:proof-of-lemma:error-estimate-between-x-and-projected-x}.
Up to this point, we have developed sufficient machinery to obtain the uniform weak error estimate of the SDE \eqref{eq:semi-linear-SODE} and the LTPE scheme \eqref{introduction:LTPE-scheme} as below.
\begin{theorem} \label{theorem:time-independent-weak-error-analysis}
Let Assumptions \ref{assumption:one-side-Lipschitz-condition-for-linear-operator}-\ref{assumption:growth-condition-of-frechet-derivatives-of-drift-and-diffusion} hold with $p_{0} \geq \max\{4\gamma+1, 5\gamma-4\}$ and $2\lambda_{1} > \max\{L_{1}, L_{2}\}$. Also, let $h$ be the uniform timestep satisfying 
\begin{equation}
h \in 
\left(
0, 
\min\left\{\tfrac{1}{2(1-\theta)\lambda_{1}}, \tfrac{p_{0}-p}{(1-\theta)(2p_{0}-p-1)\lambda_{1}}, \tfrac{1}{(1-\theta)\lambda_{d}},
1 
\right\}
\right),
\end{equation}
where $p \in (1,p_{0}) \cap \mathbb{N}$.
Moreover, denote by $\{X^{x_{0}}_{t} \}_{t\in [0,T]}$ and $\{Y^{x_{0}}_{n} \}_{0 \leq n \leq N}$, $N \in \mathbb{N}$ the solutions to SDE \eqref{eq:semi-linear-SODE} and the LTPE numerical scheme \eqref{introduction:LTPE-scheme} with the initial state $x_{0}$, respectively.
Then, for some test functions $\varphi \in C^{3}_{b}(\mathbb{R}^{d})$,
\begin{equation}
\big| 
\mathbb{E}\left[\varphi(Y^{x_{0}}_{N}) \right] 
- \mathbb{E}\left[\varphi(X^{x_{0}}_{T}) \right]
\big| 
\leq C_{A}h.
\end{equation}
\end{theorem}
\begin{proof} [Proof of Theorem \ref{theorem:time-independent-weak-error-analysis}]
We begin with the following denotation, for $n\in \{0,1,\dots,N-1 \}$, $N\in \mathbb{N}$, 
\begin{equation}
Z_{n}:= Y^{x_{0}}_{n} - \theta A  Y^{x_{0}}_{n} h. 
\end{equation}
Due to the fact that
\begin{equation}
\mathbb{E}\left[\varphi(X^{x_{0}}_{T}) \right]
=u(T,x_0), \quad
\mathbb{E}\left[\varphi(Y^{x_{0}}_{N}) \right] 
=\mathbb{E}\left[\varphi\Big(X^{Y^{x_{0}}_{N}}_{0}\Big) \right]
=u(0,Y^{x_{0}}_{N}),
\end{equation}
the weak error analysis can be divided into several parts as
\begin{equation} \label{eq:decomposition-of-the-weak-error}
\begin{aligned}
\big| 
\mathbb{E}\left[\varphi(Y^{x_{0}}_{N}) \right] 
- \mathbb{E}\left[\varphi(X^{x_{0}}_{T}) \right]
\big| 
&=  
\big| 
\mathbb{E}\left[u(T, x_{0}) \right] 
- \mathbb{E}\left[u(0, Y^{x_{0}}_{N}) \right] 
\big| \\
& \leq  
\big| 
\mathbb{E}\left[u(0, Y^{x_{0}}_{N}) \right] 
- \mathbb{E}\left[u(0, Z_{N}) \right]
\big|  
+ \big|  
\mathbb{E}\left[u(T, Z_{0}) \right] 
- \mathbb{E}\left[u(T, x_{0}) \right] 
\big| \\
& \quad + \big| 
\mathbb{E}\left[u(0, Z_{N}) \right] 
- \mathbb{E}\left[u(T, Z_{0}) \right]
\big|\\
& =: J_{1} + J_{2} + J_{3}.
\end{aligned}
\end{equation}
For the estimate of $J_{1}$, one observes by the construction of $Z_{N}$ and Lemma \ref{lemma:uniform-moments-bound-of-the-LTPE-method} that
\begin{equation} \label{eq:final-estimate-of-J1}
\begin{aligned}
J_{1} \leq h
\mathbb{E} \left[ \|AY_{N}\| \right] 
\leq C_{A}h.
\end{aligned}
\end{equation}
For the estimate of $J_{2}$, due to the fact $Z_{0}=x_{0}-\theta A x_{0}$, Lemma \ref{lemma:uniform-moments-bound-of-the-LTPE-method} and Corollary \ref{lemma:contractivity-of-u},
we derive, for some positive constant $\alpha_{1}$ defined in Corollary \ref{lemma:contractivity-of-u},
\begin{equation} \label{eq:estimate-of-J3}
\begin{aligned}
J_{2} &\leq C e^{-\alpha_{1} T}
\|Z_{0}-x_{0} \|_{L^{2} (\Omega, \mathbb{R}^{d})}  
\leq C_{A}e^{-\alpha_{1} T}h.
\end{aligned}
\end{equation}
About $J_{3}$, by \eqref{intro:continuous-version-of-the-numerical-scheme}, it is easy to see $Z_{n+1}=\mathbb{Z}^{n}(t_{n+1})$. Then using a telescoping sum argument shows that
\begin{equation}
\begin{aligned}
J_{3} 
&= \left| 
\sum_{n=0}^{N-1}  
\mathbb{E}\big[u(T-t_{n+1}, Z_{n+1}) \big] 
- \mathbb{E}\big[u(T-t_{n}, Z_{n}) \big] 
\right| \\
& \leq 
\left| 
\sum_{n=0}^{N-1}   
\mathbb{E}
\left[
u\big(T-t_{n}, \mathbb{Z}^{n}(t_{n})\big) 
\right] 
- 
\mathbb{E}\big[u(T-t_{n}, Z_{n}) \big] 
\right| \\
& \quad + \left| 
\sum_{n=0}^{N-1}  
\mathbb{E}
\left[
u\big( T-t_{n+1}, \mathbb{Z}^{n}(t_{n+1}) \big) 
\right] - \mathbb{E}\left[u\big(T-t_{n}, \mathbb{Z}^{n}(t_{n})\big) \right] \right| \\
& =: J_{3,1} + J_{3,2}.
\end{aligned}
\end{equation}
Together with the same analysis as \eqref{eq:estimate-of-J3}, applying Lemma \ref{lemma:error-estimate-between-x-and-projected-x}, Corollary \ref{lemma:contractivity-of-u} and the construction of $ \mathbb{Z}^{n}(t_{n})$  yields
\begin{equation}
\begin{aligned}
J_{3,1} &\leq C\sum_{n=0}^{N-1} 
e^{-\alpha_{1} (T-t_{n} )} 
\|
Z_{n}-\mathbb{Z}^{n}(t_{n}) 
\|_{L^{2} (\Omega, \mathbb{R}^{d})} 
\leq C_{A}
\sum_{n=0}^{N-1} 
he^{-\alpha_{1} (T-t_{n})} 
\left(
1
+\sup_{0\leq r \leq N} 
\|Y_{r} \|^{4\gamma}_{L^{8\gamma+2} (\Omega, \mathbb{R}^{d})} 
\right)  h \\
& \leq 
C_{A} 
\left(
1
+
\sup_{0\leq r \leq N} 
\|Y_{r} \|^{4\gamma+1}_{L^{8\gamma+2} (\Omega, \mathbb{R}^{d})} 
\right)h,
\end{aligned}
\end{equation}
where $\sum_{n=0}^{N-1} he^{-\alpha_{1} (T-t_{n})}
$ is uniformly bounded. 

For the remaining term $J_{3,2}$,
recalling the associated Kolmogorov equation \eqref{equation:kolmogorov-equation}, the It\^o formula and \eqref{intro:continuous-version-of-the-numerical-scheme}, we obtain that, for every $n \in \{ 0, 1, \dots, N-1\}$, $N \in \mathbb{N}$,
\begin{equation}
\begin{aligned}
&\mathbb{E}
\left[
u\big(T-t_{n+1}, \mathbb{Z}^{n}(t_{n+1})\big) 
\right] 
- 
\mathbb{E}
\left[
u\big(T-t_{n}, \mathbb{Z}^{n}(t_{n})\big) 
\right]\\
& = 
\mathbb{E} 
\left[
\int_{t_{n}}^{t_{n+1}} 
Du\big(T-s, \mathbb{Z}^{n}(s)\big)  
\Big(
F\big(\mathscr{P}(Y_{n})\big) - F\big(\mathbb{Z}^{n}(s)\big)
\Big) \dd s
\right] \\
& \quad + \tfrac{1}{2}\sum_{j=1}^{m} 
\mathbb{E}
\Bigg[
\int_{t_{n}}^{t_{n+1}}
D^{2}u\big(T-s, \mathbb{Z}^{n}(s)\big) 
\Big(
g_{j}\big(\mathscr{P}(Y_{n})\big),
g_{j}\big(\mathscr{P}(Y_{n})\big) 
\Big) 
\\
&\qquad- 
D^{2}u\big(T-s, \mathbb{Z}^{n}(s) \big) 
\Big(
g_{j}\big(\mathbb{Z}^{n}(s)\big),
g_{j}\big(\mathbb{Z}^{n}(s)\big) 
\Big) \dd s
\Bigg] \\
& =: \mathbb{J}_{1} + \mathbb{J}_{2}.
\end{aligned}
\end{equation}
A further decomposition is introduced for $\mathbb{J}_{1}$
\begin{equation}
\begin{aligned}
\mathbb{J}_{1}
&= 
\mathbb{E} 
\left[
\int_{t_{n}}^{t_{n+1}} 
Du\big(T-s, \mathbb{Z}^{n}(s)\big)  
\Big(
F\big(\mathscr{P}(Y_{n})\big) 
- F\big(\mathbb{Z}^{n}(t_{n})\big) 
\Big) \dd s
\right] \\
& \quad  + 
\mathbb{E} 
\left[
\int_{t_{n}}^{t_{n+1}} 
Du\big(T-s, \mathbb{Z}^{n}(t_{n}) \big) 
\Big( 
F\big(\mathbb{Z}^{n}(t_{n})\big)  
- F\big(\mathbb{Z}^{n}(s)\big) 
\Big) \dd s
\right] \\
& \quad  + 
\mathbb{E} 
\left[
\int_{t_{n}}^{t_{n+1}} 
\Big( 
Du\big(T-s, \mathbb{Z}^{n}(s)\big) 
- Du\big(T-s, \mathbb{Z}^{n}(t_{n}) \big)
\Big) 
\Big(
F\big(\mathbb{Z}^{n}(t_{n})\big)  
- F\big(\mathbb{Z}^{n}(s)\big) 
\Big) \dd s
\right] \\
& =: \mathbb{J}_{1,1} + \mathbb{J}_{1,2} + \mathbb{J}_{1,3}.
\end{aligned}
\end{equation}
Now we are in a position to estimate $\mathbb{J}_{1,1}$.
By Lemma \ref{lemma:necessary-estimates}, Lemma \ref{lemma:estimate-of-u-and-its-derivatives}, \eqref{equation:growth-of-the-drift-F} 
and the H\"older inequality, we get 
\begin{equation}
\begin{aligned}
\mathbb{J}_{1,1} 
&\leq C\int_{t_{n}}^{t_{n+1}}  e^{-\alpha_{1}(T-s)} 
\left\|
F\big(\mathscr{P}(Y_{n})\big) 
- F\big(\mathbb{Z}^{n}(t_{n})\big) 
\right\|_{L^{2}(\Omega, \mathbb{R}^{d})} \dd s \\
& \leq C_{A} \int_{t_{n}}^{t_{n+1}}
e^{-\alpha_{1}(T-s)} \dd s   
\left(
1+ \sup_{0\leq r \leq N}
\|Y_{r} \|^{\gamma}_{L^{2\gamma} (\Omega, \mathbb{R}^{d})} 
\right)h .
\end{aligned}
\end{equation}
For the estimate of $\mathbb{J}_{1,2}$, the Taylor expansion 
and a conditional expectation argument gives
\begin{equation}
\begin{aligned}
-\mathbb{J}_{1,2} 
&= \mathbb{E} 
\left[
\int_{t_{n}}^{t_{n+1}} 
\left\langle 
Du\big(T-s, \mathbb{Z}^{n}(t_{n})\big),  
DF\big(\mathbb{Z}^{n}(t_{n})\big)
F\big(\mathscr{P}(Y_{n})\big) (s-t_n)
+ \mathcal{R}_{F} 
\big(\mathbb{Z}^{n}(s), \mathbb{Z}^{n}(t_{n})\big)
\right\rangle \dd s
\right], \\
\end{aligned}
\end{equation}
where
\begin{equation}
\begin{aligned}
\mathcal{R}_{F} 
\big(\mathbb{Z}^{n}(s), \mathbb{Z}^{n}(t_{n})\big) 
:= \int_{0}^{1}  
\bigg(
DF\Big(\mathbb{Z}^{n}(t_{n}) 
+ r \big(
\mathbb{Z}^{n}(s)  - \mathbb{Z}^{n}(t_{n}) 
\big)
\Big) 
-DF\big(\mathbb{Z}^{n}(t_{n})\big)
\bigg)
\Big(\mathbb{Z}^{n}(s)  - \mathbb{Z}^{n}(t_{n}) \Big) 
\dd r .
\end{aligned}
\end{equation}
Keeping 
\eqref{eq:growth-of-the-diff-drift-F} and 
\eqref{equation:growth-of-the-drift-F} 
in mind, we obtain that 
\begin{equation}
\begin{aligned}
(s-t_n)
\left\|
DF\big(\mathbb{Z}^{n}(t_{n})\big) 
F\big(\mathscr{P}(Y_{n})\big) 
\right\|_{L^{2} (\Omega, \mathbb{R}^{d})}  
&\leq C_{A} 
\left\|
\big(1+\|\mathbb{Z}^{n}(t_{n})\|\big)^{\gamma-1} \big(1+\|\mathscr{P}(Y_{n})\|\big)^{\gamma} 
\right\|_{L^{2} (\Omega, \mathbb{R})} (s-t_n). \\
\end{aligned}
\end{equation}
If $\gamma = 1$, by Lemma \ref{lemma:uniform-moments-bound-of-the-LTPE-method}, one directly arrives at 
\begin{equation}
\begin{aligned}
(s-t_n)
\left\|
DF\big(\mathbb{Z}^{n}(t_{n})\big) 
F\big(\mathscr{P}(Y_{n})\big) 
\right\|_{L^{2} (\Omega, \mathbb{R}^{d})}  
&\leq C_{A} 
\left(
1+\sup_{0\leq r \leq N} 
\|Y_{r} \|_{L^{2} (\Omega, \mathbb{R}^{d})}
\right)  h.
\end{aligned}
\end{equation}
If $\gamma >1$, using  the H\"older inequality yields
\begin{small}
\begin{equation}
\begin{aligned}
\left\|
\big(1+\|\mathbb{Z}^{n}(t_{n})\|\big)^{\gamma-1} \big(1+\|\mathscr{P}(Y_{n})\|\big)^{\gamma} 
\right\|_{L^{2} (\Omega, \mathbb{R})} 
\leq
\left\|
\big(1+\|\mathbb{Z}^{n}(t_{n})\|\big)^{\gamma-1} 
\right\|_{L^{2k_{1}} (\Omega, \mathbb{R})}
\Big\|
\big(1+\|\mathscr{P}(Y_{n})\|\big)^{\gamma} 
\Big\|_{L^{2k_{2}} (\Omega, \mathbb{R})},
\end{aligned}
\end{equation}
\end{small}
where we take $k_{1}=(2\gamma-1)/(\gamma-1)$ and $k_{2}=(2\gamma-1)/\gamma$ with Lemma \ref{lemma:necessary-estimates}, Lemma \ref{lemma:uniform-moments-bound-of-SDEs} and Lemma \ref{lemma:holder-continuity-of-the-process-z} to get
\begin{equation}
\begin{aligned}
&(s-t_n)
\left\|
DF\big(\mathbb{Z}^{n}(t_{n})\big) 
F\big(\mathscr{P}(Y_{n})\big) 
\right\|_{L^{2} (\Omega, \mathbb{R}^{d})}  \\
&\leq C_{A}
\left(
1+ 
\|\mathbb{Z}^{n}(t_{n}) \|^{\gamma-1}_{L^{4\gamma-2} (\Omega, \mathbb{R}^{d})}
\right) 
\left(
1+
\|Y_{n} \|^{\gamma}_{L^{4\gamma-2} (\Omega, \mathbb{R}^{d})}
\right)
h \\
&\leq C_{A}
\left(
1+\sup_{0\leq r \leq N} 
\|Y_{r} \|^{2\gamma-1}_{L^{4\gamma-2} (\Omega, \mathbb{R}^{d})}
\right)  h.
\end{aligned}
\end{equation}
Similarly, we can also attain
\begin{small}
\begin{equation}
\begin{aligned}
\left\|
\mathcal{R}_{F} 
\big(
\mathbb{Z}^{n}(s), \mathbb{Z}^{n}(t_{n})
\big) \right\|_{L^{2} (\Omega, \mathbb{R}^{d})} 
&\leq C 
\left\|
\mathcal{P}_{\gamma-2} 
\Big(
\mathbb{Z}^{n}(t_{n}) + r \big( \mathbb{Z}^{n}(s)
-\mathbb{Z}^{n}(t_{n}) \big), \mathbb{Z}^{n}(t_{n}) 
\Big)  
\left\| 
\mathbb{Z}^{n}(s) -\mathbb{Z}^{n}(t_{n})
\right\|^{2} 
\right\|_{L^{2} (\Omega, \mathbb{R})} \\
& \leq C_{A}
\left(
1 + \sup_{0\leq r \leq N} 
\|Y_{r} \|^{\max \{2\gamma, 3\gamma-2\}}_{L^{\max\{ 4\gamma, 6\gamma-4\}} (\Omega, \mathbb{R}^{d})} 
\right) h.
\end{aligned}
\end{equation}
\end{small}
Then it follows from Lemma \ref{lemma:estimate-of-u-and-its-derivatives} that
\begin{equation}
\begin{aligned}
\mathbb{J}_{1,2} 
&\leq C_{A} 
\int_{t_{n}}^{t_{n+1}}
e^{-\alpha_{1}(T-s)} \dd s   
\left(
1+\sup_{0\leq r \leq N}  
\|Y_{r} \|^{\max \{2\gamma, 3\gamma-2\}}_{L^{\max\{ 4\gamma, 6\gamma-4\}} (\Omega, \mathbb{R}^{d})} 
\right)  h.  \\
\end{aligned}
\end{equation}
For the estimate of $\mathbb{J}_{1,3}$, the Taylor expansion to $u(t,\cdot)$ shows,  there exists some
$\mathbb{R}^{d}$-valued random variable $\widetilde{\upsilon} \in L^{2\rho_{1}}(\Omega,\mathbb{R}^{d})$
lying between $\mathbb{Z}^{n}(s)$ and $\mathbb{Z}^{n}(t_{n})$,
\begin{equation}
\begin{aligned}
\mathbb{J}_{1,3} 
&= \mathbb{E} 
\left[
\int_{t_{n}}^{t_{n+1}} 
D^{2}u(T-s, \widetilde{\upsilon})  
\Big(
\mathbb{Z}^{n}(s) - \mathbb{Z}^{n}(t_{n}),
F\big(\mathbb{Z}^{n}(t_{n})\big) 
- F\big(\mathbb{Z}^{n}(s)\big) 
\Big) \dd s
\right].
\end{aligned}
\end{equation}
Applying 
Lemma \ref{lemma:uniform-moments-bound-of-SDEs}, 
Lemma \ref{lemma:estimate-of-u-and-its-derivatives} and the H\"older inequality yields, 
\begin{equation}
\begin{aligned}
\mathbb{J}_{1,3} 
& \leq C 
\int_{t_{n}}^{t_{n+1}}
e^{-\tilde{\alpha}_{2}(T-s)} 
\sup_{r \in [0,T]}  
\left\| 
\mathcal{P}_{\gamma-2}(X^{\widetilde{\upsilon}}_{r}) 
\right\|_{L^{2\rho_{1} }(\Omega, \mathbb{R})}
\left\| 
\mathbb{Z}^{n}(s) - \mathbb{Z}^{n}(t_{n}) 
\right\|_{L^{2\rho_{2} } (\Omega, \mathbb{R}^{d})} \times\\
&\hspace{18em}   
\left\| 
F\big(\mathbb{Z}^{n}(t_{n})\big) 
- F\big(\mathbb{Z}^{n}(s)\big) 
\right\|_{L^{2\rho_{3}} (\Omega, \mathbb{R}^{d})} \dd s. \\
\end{aligned}
\end{equation}
We need to discuss the estimation of $\mathbb{J}_{1,3}$ through the range of $\gamma$. 
For the case that $\gamma >2$, taking 
$\rho_{1}=(4\gamma-3)/(\gamma-2)$, $\rho_{2}=(4\gamma-3)/\gamma$ and $\rho_{3}=(4\gamma-3)/(2\gamma-1)$ with Assumption \ref{assumption:growth-condition-of-frechet-derivatives-of-drift-and-diffusion} and Lemma \ref{lemma:holder-continuity-of-the-process-z}  gives
\begin{equation}
\begin{aligned}
\mathbb{J}_{1,3} 
\leq C_{A} \int_{t_{n}}^{t_{n+1}} 
e^{-\tilde{\alpha}_{2}(T-s)} \dd s   
\left(
1 +\sup_{0\leq r \leq N} 
\left\|
Y_{r} 
\right\|^{4\gamma-3}_{L^{8\gamma-6 }(\Omega, \mathbb{R}^{d})}  
\right) h.
\end{aligned}
\end{equation}
For the case that $1\leq \gamma \leq 2$, choosing $\rho_{1}=\infty$, $\rho_{2}=(3\gamma-1)/\gamma$ and $\rho_{3}= (3\gamma-1)/(2\gamma-1)$ shows
\begin{equation}
\begin{aligned}
\mathbb{J}_{1,3} 
\leq C_{A} \int_{t_{n}}^{t_{n+1}}  
e^{-\tilde{\alpha}_{2}(T-s)}\dd s   
\left(
1 +\sup_{0\leq r \leq N} 
\left\|Y_{r} \right\|^{3\gamma-1}_{L^{6\gamma-2 }(\Omega, \mathbb{R}^{d})}  
\right) h.
\end{aligned}
\end{equation}
Consequently, combining the estimations of $\mathbb{J}_{1,1}$-$\mathbb{J}_{1,3}$ leads to
\begin{equation} \label{eq:final-estimate-of-J1}
\begin{aligned}
\mathbb{J}_{1} 
\leq  C_{A} \int_{t_{n}}^{t_{n+1}}  
e^{-\min(\alpha_{1}, \tilde{\alpha}_{2} )(T-s)}  \dd s  
\left(
1 + \sup_{0\leq r \leq N}
\left\|
Y_{r} 
\right\|^{\max\{3\gamma-1, 4\gamma-3\}}_{L^{\max\{6\gamma-2, 8\gamma-6 \}}(\Omega, \mathbb{R}^{d})}  
\right) h.
\end{aligned}
\end{equation}
For the estimate of $\mathbb{J}_{2}$, we here use the following equality, for any matrix $U \in \mathbb{R}^{d\times d}$ and any $a, b \in \mathbb{R}^{d}$,
\begin{equation}
\begin{aligned}
a^{T}Ua - b^{T}Ub 
= -(a-b)^{T}U(a-b)-(a-b)^{T}Ub-a^{T}U(a-b).
\end{aligned}
\end{equation}
As a result, one can show a further decomposition of $\mathbb{J}_{2}$ as follows,
\begin{equation}
\begin{aligned}
\mathbb{J}_{2} 
& = -\tfrac{1}{2}\sum_{j=1}^{m} 
\mathbb{E}\left[
\int_{t_{n}}^{t_{n+1}}
D^{2}u\big(T-s, \mathbb{Z}^{n}(s)\big)
\Big(
g_{j}\big(\mathbb{Z}^{n}(s)\big) 
- g_{j}\big(\mathscr{P}(Y_{n})\big),
g_{j}\big(\mathbb{Z}^{n}(s)\big)
-  g_{j}\big(\mathscr{P}(Y_{n})\big) 
\Big) \dd s
\right] \\
& \quad -\tfrac{1}{2}\sum_{j=1}^{m} 
\mathbb{E}\left[
\int_{t_{n}}^{t_{n+1}}
D^{2}u\big(T-s, \mathbb{Z}^{n}(s)\big) 
\Big(
g_{j}\big(\mathbb{Z}^{n}(s)\big) 
- g_{j}\big(\mathscr{P}(Y_{n})\big), 
g_{j}\big(\mathscr{P}(Y_{n})\big) 
\Big) \dd s
\right]\\
& \quad -\tfrac{1}{2}\sum_{j=1}^{m} 
\mathbb{E}
\left[ 
\int_{t_{n}}^{t_{n+1}}
D^{2}u\big(T-s, \mathbb{Z}^{n}(s)\big) 
\Big(
g_{j}\big(\mathscr{P}(Y_{n})\big),
g_{j}\big(\mathbb{Z}^{n}(s)\big)
-  g_{j}\big(\mathscr{P}(Y_{n})\big) 
\Big) \dd s
\right] \\
& =: \mathbb{J}_{2,1} + \mathbb{J}_{2,2} + \mathbb{J}_{2,3}.
\end{aligned}
\end{equation}
For $\mathbb{J}_{2,1}$, combining Lemma \ref{lemma:differentiability-of-solutions}, Lemma \ref{lemma:estimate-of-u-and-its-derivatives}  with $\rho_{2}=\rho_{3}$ implies
\begin{equation}
\begin{aligned}
\mathbb{J}_{2,1} 
&\leq C \sum_{j=1}^{m}  
\int_{t_{n}}^{t_{n+1}}  
e^{-\tilde{\alpha}_{2}(T-s)} 
\sup_{r \in [0,T]}  
\left\|
\mathcal{P}_{\gamma-2}
\left(
X^{\mathbb{Z}^{n}(s)}_{r} 
\right)
\right\|_{L^{2\rho_{1} }(\Omega, \mathbb{R})}      
\left\| 
g_{j}\big(\mathbb{Z}^{n}(s)\big) 
- g_{j}\big(\mathscr{P}(Y_{n})\big) 
\right\|^{2}_{L^{2\rho_{2}} (\Omega, \mathbb{R}^{d})} \dd s , \\
\end{aligned}
\end{equation}
where using Lemma \ref{lemma:holder-continuity-of-the-process-z},  Assumption \ref{assumption:growth-condition-of-frechet-derivatives-of-drift-and-diffusion} and the H\"older inequality 
gives that, for some $\rho_{2} \geq 1$,  $j\in \{1,\dots,m \}$,
\begin{equation}
\begin{aligned}
&\left\|
g_{j}\big(\mathbb{Z}^{n}(s)\big) 
- g_{j}\big(\mathscr{P}(Y_{n})\big) 
\right\|_{L^{2\rho_{2}} (\Omega, \mathbb{R}^{d})} \\ 
&\leq  
\left\| 
g_{j}(\mathbb{Z}^{n}(s)) 
- g_{j}\big(\mathbb{Z}^{n}(t_{n})\big) 
\right\|_{L^{2\rho_{2}} (\Omega, \mathbb{R}^{d})} 
+  \left\| 
g_{j}\big(\mathbb{Z}^{n}(t_{n})\big) 
- g_{j}\big(\mathscr{P}(Y_{n})\big) 
\right\|_{L^{2\rho_{2}} (\Omega, \mathbb{R}^{d})} \\
& \leq C_{A} h^{\frac{1}{2}} 
\left(
1 + \sup_{0 \leq r \leq N} 
\left\|Y_{r} \right\|^{\frac{3\gamma-1}{2}}_{L^{\rho_{2}(3\gamma-1)} (\Omega, \mathbb{R}^{d}) } 
\right) 
+ C_{A} h  
\left(
1 + \sup_{0 \leq r \leq N} 
\left\|Y_{r} \right\|^{\frac{\gamma+1}{2}}_{L^{\rho_{2}(\gamma+1)} (\Omega, \mathbb{R}^{d}) } 
\right) \\
& \leq C_{A} h^{\frac{1}{2}} 
\left(
1 + \sup_{0 \leq r \leq N}
\left\|Y_{r} \right\|^{\frac{3\gamma-1}{2}}_{L^{\rho_{2}(3\gamma-1)} (\Omega, \mathbb{R}^{d}) } 
\right).
\end{aligned}
\end{equation}
For $\gamma >2$, choosing $\rho_{1} = (4\gamma-3)/(\gamma-2)$, $\rho_{2}=\rho_{3} = (8\gamma-6)/(3\gamma-1)$ yields
\begin{equation}
\begin{aligned}
\mathbb{J}_{2,1} 
\leq C_{A}  \int_{t_{n}}^{t_{n+1}}  
e^{- \tilde{\alpha}_{2}(T-s)}    \dd s  
\left(
1 + \sup_{0 \leq r \leq N} 
\left\|Y_{r} \right\|^{4\gamma-3}_{L^{8\gamma-6 }(\Omega, \mathbb{R}^{d})}  
\right) h.
\end{aligned}
\end{equation}
For $1\leq \gamma \leq 2 $, taking $\rho_{1}= \infty$, $\rho_{2}= \rho_{3}=2$ leads to
\begin{equation}
\begin{aligned}
\mathbb{J}_{2,1} 
\leq C_{A}  \int_{t_{n}}^{t_{n+1}}  
e^{-\tilde{\alpha}_{2}(T-s)}    \dd s  
\left(
1 + \sup_{0 \leq r \leq N}  
\left\|Y_{r} \right\|^{3\gamma-1}_{L^{6\gamma-2 }(\Omega, \mathbb{R}^{d})}  
\right) h.
\end{aligned}
\end{equation}
The estimates of $\mathbb{J}_{2,2}$ and $\mathbb{J}_{2,3}$ are in the same way. As a consequence, we take $\mathbb{J}_{2,2}$ as an example.
Then an application of the Taylor expansion with a conditional expectation argument yields that
\begin{footnotesize}
\begin{equation}
\begin{aligned}
-\mathbb{J}_{2,2} 
&= 
\tfrac{1}{2}\sum_{j=1}^{m} 
\mathbb{E}
\left[ 
\int_{t_{n}}^{t_{n+1}}
D^{2}u\big(T-s, \mathbb{Z}^{n}(s)\big) 
\Big(
g_{j}\big(\mathbb{Z}^{n}(s)\big) 
- g_{j}\big(\mathbb{Z}^{n}(t_{n})\big), 
g_{j}\big(\mathscr{P}(Y_{n})\big) 
\Big) \dd s
\right] \\
& \quad + \tfrac{1}{2}\sum_{j=1}^{m} 
\mathbb{E}
\left[
\int_{t_{n}}^{t_{n+1}}
D^{2}u\big(T-s, \mathbb{Z}^{n}(s)\big) 
\Big(
g_{j}\big(\mathbb{Z}^{n}(t_{n})\big) 
- g_{j}\big(\mathscr{P}(Y_{n})\big), 
g_{j}\big(\mathscr{P}(Y_{n})\big) 
\Big) \dd s
\right]\\
& = \tfrac{1}{2} (s-t_{n})\sum_{j=1}^{m} 
\mathbb{E}
\left[
\int_{t_{n}}^{t_{n+1}}
D^{2}u\big(T-s, \mathbb{Z}^{n}(t_{n})\big) 
\Big( 
Dg_{j}\big(\mathbb{Z}^{n}(t_{n})\big) 
F\big(\mathscr{P}(Y_{n})\big) \dd r \ , 
g_{j}\big(\mathscr{P}(Y_{n})\big) 
\Big)
\right] \\
& \quad + \tfrac{1}{2}\sum_{j=1}^{m} 
\mathbb{E}
\left[
\int_{t_{n}}^{t_{n+1}}
D^{2}u\big(T-s, \mathbb{Z}^{n}(t_{n})\big)  
\Big(
\mathcal{R}_{g_{j}} 
\big(\mathbb{Z}^{n}(s), \mathbb{Z}^{n}(t_{n}) \big), 
g_{j}\big(\mathscr{P}(Y_{n})\big) 
\Big) \dd s
\right]  \\
& \quad +\tfrac{1}{2}\sum_{j=1}^{m} 
\mathbb{E}
\left[
\int_{t_{n}}^{t_{n+1}}
\Big(
D^{2}u\big(T-s, \mathbb{Z}^{n}(s)\big) 
- D^{2}u\big(T-s, \mathbb{Z}^{n}(t_{n})\big)
\Big) 
\Big(
g_{j}(\mathbb{Z}^{n}(s)) 
- g_{j}\big(\mathbb{Z}^{n}(t_{n})\big), 
g_{j}\big(\mathscr{P}(Y_{n})\big) 
\Big) \dd s
\right] \\
& \quad + \tfrac{1}{2}\sum_{j=1}^{m} 
\mathbb{E}
\left[ 
\int_{t_{n}}^{t_{n+1}}
D^{2}u\big(T-s, \mathbb{Z}^{n}(s)\big) 
\Big(
g_{j}\big(\mathbb{Z}^{n}(t_{n})\big) 
- g_{j}\big(\mathscr{P}(Y_{n})\big), 
g_{j}\big(\mathscr{P}(Y_{n})\big) 
\Big) \dd s
\right]\\
& =: \mathbb{J}_{2,2,1} + \mathbb{J}_{2,2,2} + \mathbb{J}_{2,2,3}  + \mathbb{J}_{2,2,4},
\end{aligned}
\end{equation}
\end{footnotesize}
where we denote that, for $j\in \{1,\dots,m \}$,
\begin{equation}
\begin{aligned}
\mathcal{R}_{g_{j}} 
\Big(\mathbb{Z}^{n}(s), \mathbb{Z}^{n}(t_{n})\Big)
:=
\int_{0}^{1} 
\left[
Dg_{j} 
\left(
\mathbb{Z}^{n}(t_{n}) + r\big(\mathbb{Z}^{n}(s) 
- \mathbb{Z}^{n}(t_{n})\big) 
\right) 
- Dg_{j} \big(\mathbb{Z}^{n}(t_{n}) \big) 
\right] 
\big(\mathbb{Z}^{n}(s) - \mathbb{Z}^{n}(t_{n})\big)\dd r.
\end{aligned}
\end{equation}
Using Lemma \ref{lemma:uniform-moments-bound-of-SDEs} and Lemma \ref{lemma:estimate-of-u-and-its-derivatives} implies that
\begin{equation}
\begin{aligned}
\mathbb{J}_{2,2,1} 
&\leq C_{A} (s-t_{n}) \int_{t_{n}}^{t_{n+1}}
e^{-\tilde{\alpha}_{2}(T-s)} \dd s 
\sup_{r \in [0,T]}   
\left\|
\mathcal{P}_{\gamma-2}\left( X_{r}^{\mathbb{Z}^{n}(t_{n}) }\right)
\right\|_{L^{2\rho_{1}   }(\Omega, \mathbb{R})}   
\times \\
&\hspace{8em}\sum_{j=1}^{m}
\left\| 
Dg_{j}\big(\mathbb{Z}^{n}(t_{n})\big) 
F\big(\mathscr{P}(Y_{n})\big) 
\right\|_{L^{2\rho_{2} } (\Omega, \mathbb{R}^{d})}
\left\| g_{j}\big(\mathscr{P}(Y_{n})\big) \right\|_{L^{2\rho_{3}} (\Omega, \mathbb{R}^{d})}.  \\
\end{aligned}
\end{equation}
For the case that $\gamma \geq 2$, it follows from Assumption \ref{assumption:growth-condition-of-frechet-derivatives-of-drift-and-diffusion} and Lemma \ref{lemma:uniform-moments-bound-of-the-LTPE-method} that
\begin{equation}
\begin{aligned}
\mathbb{J}_{2,2,1} 
&\leq   C_{A} \int_{t_{n}}^{t_{n+1}}  
e^{-\tilde{\alpha}_{2}(T-s)}    \dd s   
\left(
1 + \sup_{0\leq r \leq N}
\left\|Y_{r} \right\|^{3\gamma-2}_{L^{6\gamma-4 }(\Omega, \mathbb{R}^{d})}  
\right) h
\end{aligned}
\end{equation}
where we let 
$\rho_{1}=(3\gamma-2)/(\gamma-2)$, 
$\rho_{2}=(6\gamma-4)/(3\gamma-1)$ and $\rho_{3}=(6\gamma-4)/(\gamma-1)$. 
For $1\leq \gamma \leq 2$, taking 
$\rho_{1}=\infty$, $\rho_{2}=4\gamma/(3\gamma-1)$ and 
$\rho_{3}=4\gamma/(\gamma+1)$ leads to
\begin{equation}
\begin{aligned}
\mathbb{J}_{2,2,1}
\leq  C_{A} \int_{t_{n}}^{t_{n+1}}  
e^{-\tilde{\alpha}_{2}(T-s)}    \dd s   
\left(
1 + \sup_{0\leq r \leq N}  
\left\|Y_{r} \right\|^{2\gamma}_{L^{4\gamma }(\Omega, \mathbb{R}^{d})}  
\right) h.
\end{aligned}
\end{equation}
Similarly, one gets
\begin{equation} \label{eq:estimate-of-J-2-2-2}
\begin{aligned}
\mathbb{J}_{2,2,2} 
&\leq C  \int_{t_{n}}^{t_{n+1}} 
e^{-\tilde{\alpha}_{2}(T-s)} \sup_{r \in [0,T]}  
\left\|
\mathcal{P}_{\gamma-2}
\left( X_{r}^{\mathbb{Z}^{n}(t_{n})}\right)
\right\|_{L^{2\rho_{1}   }(\Omega, \mathbb{R})}
\sum_{j=1}^{m}
\left\| 
\mathcal{R}_{g_{j}} 
\Big(\mathbb{Z}^{n}(s), \mathbb{Z}^{n}(t_{n})\Big)
\right\|_{L^{2\rho_{2} } (\Omega, \mathbb{R}^{d})} \times \\
&\hspace{28em}
\left\| 
g_{j}\big(\mathscr{P}(Y_{n})\big) 
\right\|_{L^{2\rho_{3}} (\Omega, \mathbb{R}^{d})} \dd s ,
\end{aligned}
\end{equation}   
where one obtains easily from Assumption \ref{assumption:growth-condition-of-frechet-derivatives-of-drift-and-diffusion} and Lemma \ref{lemma:holder-continuity-of-the-process-z} that, for some $\rho_{2}\geq 1$, $j\in \{1,\dots,m \}$,
\begin{equation}
\begin{aligned}
&\left\| 
\mathcal{R}_{g_{j}} 
\Big(
\mathbb{Z}^{n}(s), \mathbb{Z}^{n}(t_{n})
\Big)
\right\|_{L^{2\rho_{2} } (\Omega, \mathbb{R}^{d})}  \\
& \leq \int_{0}^{1} 
\left\| 
\Big[ Dg_{j} 
\Big(
\mathbb{Z}^{n}(t_{n}) + r\big(\mathbb{Z}^{n}(s) 
- \mathbb{Z}^{n}(t_{n})\big) 
\Big) 
- Dg_{j} \big(\mathbb{Z}^{n}(t_{n}) \big) 
\Big] 
\Big(
\mathbb{Z}^{n}(s) - \mathbb{Z}^{n}(t_{n})
\Big)\right\|_{L^{2\rho_{2} } (\Omega, \mathbb{R}^{d})}\dd r \\
& \leq C \int_0^1 
\left\|
\left(
1+\big\|
r \mathbb{Z}^{n}(s) +(1-r) \mathbb{Z}^{n}(t_{n})
\big\|
+\big\|\mathbb{Z}^{n}(t_{n})\big\|
\right)^{\frac{\gamma-3}{2}} 
\Big\| 
\mathbb{Z}^{n}(s)-\mathbb{Z}^{n}(t_{n})
\Big\|^2
\right\|_{L^{2\rho_{2}}(\Omega ; \mathbb{R})} \mathrm{d} r \\
& \leq C_{A}  
\left(
1 + \sup_{0\leq r \leq N}
\left\|Y_{r} \right\|^{\max\{ 2\gamma, \frac{5\gamma-3}{2} \}}_{L^{\max\{4\rho_{2}\gamma, \rho_{2}  (5\gamma-3)\} }(\Omega, \mathbb{R}^{d})}  
\right)h.
\end{aligned}
\end{equation}
Putting this estimate into \eqref{eq:estimate-of-J-2-2-2} with the H\"older inequality and taking the same discussion about $\gamma$ before
yield
\begin{equation}
\begin{aligned}
\mathbb{J}_{2,2,2} 
\leq C_{A}     
\int_{t_{n}}^{t_{n+1}}  
e^{-\tilde{\alpha}_{2}(T-s)}    \dd s  
\left(
1 + \sup_{0\leq r \leq N}
\left\|Y_{r} \right\|^{
\max\{\frac{5\gamma+1}{2}, \frac{7\gamma-3}{2} 4\gamma-3\}
}_{L^{\max\{5\gamma+1, 7\gamma-3, 8\gamma-6\} }(\Omega, \mathbb{R}^{d})}  
\right) h.
\end{aligned}
\end{equation}
Then, the Taylor expansion and Lemma \ref{lemma:estimate-of-u-and-its-derivatives} are used to give that, for some random variable $\widetilde{\upsilon}_{1} \in L^{\max\{6\gamma, 10\gamma-8 \}}(\Omega, \mathbb{R}^{d})$ lying between $\mathbb{Z}^{n}(s)$ and $\mathbb{Z}^{n}(t_{n})$,
\begin{equation}
\begin{aligned}
\mathbb{J}_{2,2,3} 
&=\tfrac{1}{2} \sum_{j=1}^{m} 
\mathbb{E}
\left[
\int_{t_{n}}^{t_{n+1}} 
D^{3}u(T-s, \widetilde{\upsilon}_{1}) 
\Big( 
\mathbb{Z}^{n}(s)- \mathbb{Z}^{n}(t_{n}), 
g_{j}(\mathbb{Z}^{n}(s)) - g_{j}\big(\mathbb{Z}^{n}(t_{n})\big), 
g_{j}\big(\mathscr{P}(Y_{n})\big) 
\Big) \dd s
\right] \\
& \leq C \sum_{j=1}^{m}    
\int_{t_{n}}^{t_{n+1}} e^{-\tilde{\alpha}_{3}(T-s)}  
\sup_{r \in [0,T]}  
\left\|
\mathcal{P}_{\gamma-2}
\left( X_{r}^{\widetilde{\upsilon}_{1}}\right)
\right\|_{L^{2\max\{\rho_{1}, \rho_{3}\rho_{4}\}   }(\Omega, \mathbb{R})} 
\left\| 
\mathbb{Z}^{n}(s) - \mathbb{Z}^{n}(t_{n})
\right\|_{L^{2\rho_{2} } (\Omega, \mathbb{R}^{d})} \times \\
&\hspace{10em}  
\left\| 
g_{j}(\mathbb{Z}^{n}(s)) - g_{j}\big(\mathbb{Z}^{n}(t_{n})\big)
\right\|_{L^{2\rho_{3} \rho_{5} } (\Omega, \mathbb{R}^{d})}
\left\| 
g_{j}\big(\mathscr{P}(Y_{n})\big) 
\right\|_{L^{2\rho_{3}\rho_{6} } (\Omega, \mathbb{R}^{d})} \dd s.
\end{aligned}
\end{equation}
Equipped with Lemma \ref{lemma:holder-continuity-of-the-process-z},  Assumption \ref{assumption:growth-condition-of-frechet-derivatives-of-drift-and-diffusion}  and the H\"older inequality, for $\gamma>2$, one can  choose $\rho_{1}=(5\gamma-4)/(\gamma-2)$, 
$\rho_{2}=(5\gamma-4)/\gamma$, 
$\rho_{3}=(5\gamma-4)/(3\gamma-2)$, 
$\rho_{4}=(3\gamma-2)/(\gamma-2)$, 
$\rho_{5}=(6\gamma-4)/(3\gamma-1)$ and 
$\rho_{6} = (6\gamma-4)/(\gamma+1)$ to get,
\begin{equation}
\begin{aligned}
\mathbb{J}_{2,2,3} 
&\leq C_{A} \int_{t_{n}}^{t_{n+1}}  
e^{-\tilde{\alpha}_{3}(T-s)}    \dd s    
\left(
1 
+ \sup_{0\leq r \leq N} 
\left\|
Y_{r} 
\right\|^{5\gamma-4}_{L^{10\gamma-8 }(\Omega, \mathbb{R}^{d})} 
\right) h.
\end{aligned}
\end{equation}
For $1\leq \gamma  \leq 2$, taking $\rho_{1}=\rho_{4} = \infty$, $\rho_{2}=3$, $\rho_{3}=3/2$, $\rho_{5}=4\gamma/(3\gamma-1)$ and 
$\rho_{6}=4\gamma/(\gamma+1)$ yields
\begin{equation}
\begin{aligned}
\mathbb{J}_{2,2,3} 
&\leq C_{A} \int_{t_{n}}^{t_{n+1}}  
e^{-\tilde{\alpha}_{3}(T-s)}    \dd s    
\left(
1 + \sup_{0\leq r \leq N} 
\left\|
Y_{r} 
\right\|^{3\gamma}_{L^{6\gamma }(\Omega, \mathbb{R}^{d})} 
\right) h.
\end{aligned}
\end{equation}
By Lemma \ref{lemma:estimate-of-u-and-its-derivatives}  with $q_{2}=1$, it is quite obvious that
\begin{equation}
\begin{aligned}
\mathbb{J}_{2,2,4} 
&\leq C \sum_{j=1}^{m}  \int_{t_{n}}^{t_{n+1}}
e^{-\tilde{\alpha}_{2}(T-s)} 
\sup_{r \in [0,T]}   
\left\|
\mathcal{P}_{\gamma-2}\left( X_{r}^{ \mathbb{Z}^{n}(s) }\right)
\right\|_{L^{2\rho_{1}   }(\Omega, \mathbb{R})}  
\left\| 
g_{j}\big(\mathbb{Z}^{n}(t_{n})\big) 
- g_{j}\big(\mathscr{P}(Y_{n})\big)
\right\|_{L^{2\rho_{2} } (\Omega, \mathbb{R}^{d})} \times \\
&\hspace{28em}
\left\| 
g_{j}\big(\mathscr{P}(Y_{n})\big) 
\right\|_{L^{2\rho_{3}} (\Omega, \mathbb{R}^{d})} \dd s.
\end{aligned}
\end{equation}
Following the same argument, we show that
\begin{equation}
\begin{aligned}
\mathbb{J}_{2,2,4} 
\leq C_{A} \int_{t_{n}}^{t_{n+1}}  
e^{-\tilde{\alpha}_{2}(T-s)} \dd s  
\left(
1 + \sup_{0\leq r \leq N} 
\left\|
Y_{r} 
\right\|^{\max\{\gamma+1, 2\gamma-1\}}_{L^{\max\{2\gamma+2, 4\gamma-2\} }(\Omega, \mathbb{R}^{d})}  
\right) h.
\end{aligned}
\end{equation}
Hence, by the estimates of $\mathbb{J}_{2,2,1} - \mathbb{J}_{2,2,4}$, we deduce that

\begin{equation}
\begin{aligned}
\max\left\{
\mathbb{J}_{2,2}, \mathbb{J}_{2,3} 
\right\}  
\leq C_{A}
\int_{t_{n}}^{t_{n+1}}  
e^{-\min(\tilde{\alpha}_{2}, \tilde{\alpha}_{3})(T-s) }    \dd s  \left(
1 
+ \sup_{0\leq r \leq N}
\left\|
Y_{r} 
\right\|^{\max\{3\gamma, 5\gamma-4\}}_{L^{\max\{6\gamma, 10\gamma-8\} }(\Omega, \mathbb{R}^{d})}  
\right) h,
\end{aligned}
\end{equation}
resulting in
\begin{equation}
\begin{aligned}
\mathbb{J}_{2}   
\leq C_{A}
\int_{t_{n}}^{t_{n+1}}  
e^{-\min(\tilde{\alpha}_{2}, \tilde{\alpha}_{3})(T-s) }    \dd s  \left(
1 + \sup_{0\leq r \leq N} 
\left\|
Y_{r} 
\right\|^{\max\{3\gamma, 5\gamma-4\}}_{L^{\max\{6\gamma, 10\gamma-8\} }(\Omega, \mathbb{R}^{d})}  
\right) h.
\end{aligned}
\end{equation}
Combining this with \eqref{eq:final-estimate-of-J1} leads to
\begin{equation}
\begin{aligned}
J_{3,2} 
&\leq C_{A}h   
\left(
1 + \sup_{0\leq r \leq N} 
\left\|Y_{r} \right\|^{\max\{3\gamma, 5\gamma-4\}}_{L^{\max\{6\gamma, 10\gamma-8\} }(\Omega, \mathbb{R}^{d})}  
\right)
\sum_{n=0}^{N-1}\int_{t_{n}}^{t_{n+1}}  
e^{-\min(\alpha_{1}, \tilde{\alpha}_{2}, \tilde{\alpha}_{3})(T-s) }    \dd s \\
& = C_{A}h    
\left(
1 
+ \sup_{0\leq r \leq N} 
\left\|
Y_{r} 
\right\|^{\max\{3\gamma, 5\gamma-4\}}_{L^{\max\{6\gamma, 10\gamma-8\} }(\Omega, \mathbb{R}^{d})}  
\right)
\int_{0}^{T}  
e^{-\min(\alpha_{1}, \tilde{\alpha}_{2}, \tilde{\alpha}_{3})(T-s) }    \dd s.
\end{aligned}
\end{equation}
It is known that 
\begin{equation}
\begin{aligned}
\int_{0}^{T}  
e^{-\min(\alpha_{1}, \tilde{\alpha}_{2}, \tilde{\alpha}_{3})(T-s) }   \dd s 
= \tfrac{
1-e^{-\min(\alpha_{1}, \tilde{\alpha}_{2}, \tilde{\alpha}_{3})T}
}
{
\min(\alpha_{1}, \tilde{\alpha}_{2}, \tilde{\alpha}_{3})
} 
\end{aligned}
\end{equation}
is uniformly bounded.
All in all, we are in a position to derive the estimate of $J_{3}$ as
\begin{equation} \label{eq:final-estimate-of-J2}
\begin{aligned}
J_{3} 
&\leq C_{A}  
\left(
1 + \sup_{0\leq r \leq N}  
\left\|
Y_{r} 
\right\|^{\max\{4\gamma+1, 5\gamma-4\}}_{L^{\max\{8\gamma+2, 10\gamma-8\} }(\Omega, \mathbb{R}^{d})}  
\right) h .\\
\end{aligned}
\end{equation}
Plugging \eqref{eq:estimate-of-J3}, \eqref{eq:final-estimate-of-J1} and \eqref{eq:final-estimate-of-J2} into \eqref{eq:decomposition-of-the-weak-error} gives
\begin{equation}
\big| 
\mathbb{E}\left[\varphi(Y^{x_{0}}_{N}) \right] 
- \mathbb{E}\left[\varphi(X^{x_{0}}_{T}) \right]
\big| 
\leq C_{A}
\left(
1 + \sup_{0\leq r \leq N}  
\left\|
Y_{r} 
\right\|^{\max\{4\gamma+1, 5\gamma-4\}}_{L^{\max\{8\gamma+2, 10\gamma-8\} }(\Omega, \mathbb{R}^{d})}  
\right) h,
\end{equation}
which completes the proof.
\end{proof}
To conclude, we deduce from Theorem \ref{theorem:time-independent-weak-error-analysis} that the weak convergence order of the $\pi$ and $\widetilde{\pi}$ is 1, i.e.
\begin{equation}
\left|
\int_{\mathbb{R}^d} \varphi(x) \pi(\dd x)
-\int_{\mathbb{R}^d} \varphi(x) \widetilde{\pi}( \dd x)
\right| 
\leq C_{A}h,
\end{equation}
since the constant $C_{A}$ is independent of $N$ in \eqref{equation:estimate-decomposition-of-two-measures}.

\section{Numerical experiments}
In this section, we illustrate the previous theoretical findings  through three numerical examples:
the scalar stochastic Ginzburg-Landau equation \cite{pekoldennumerical} in \textbf{Example 1}, the mean-reverting type model with super-linear coefficients \cite{liu2023backward,2013Convergence} in \textbf{Example 2} and the third is  the semi-linear stochastic partial equation (SPDE) \cite{liu2021strong,wang2023mean} in \textbf{Example 3}.

For all three numerical experiments, we consider a terminal time $T=5$, the timesteps $h=2^{-6}, 2^{-7}, 2^{-8}, 2^{-9}$ and four different choices for test function $\varphi(\cdot)$,
\begin{equation}
    \varphi(x)\in \{\arctan(\| x\|), \ e^{-\|x\|^{2}}, \ \cos(\| x\|), \ \sin(\| x\|^{2})\}.
\end{equation} The empirical mean of $\mathbb{E}\left[\varphi(X_{T}) \right]$
is estimated by a Monte Carlo approximation, involving 10,000 independent trajectories. 
It is worth noting that in \textbf{Example 2} we will test that the  terminal time $T=5$ what we have chosen is appropriate.

\textbf{Example 1.}
Consider the stochastic Ginzburg-Landau equation \cite{pekoldennumerical} from  the theory of superconductivity as follows,
\begin{equation} \label{equation:ginzburg-landau-model}
d X_t=
\left(
-X_t^3+\left(\alpha+\tfrac{1}{2} \sigma^2\right) X_t
\right) \dd t
+\sigma X_t \, \dd W_t, 
\quad \alpha, \sigma \in \mathbb{R}.
\end{equation}
Let $\alpha=-2$, $\sigma=0.5$ and $X_{0}=1$. 
Then, all conditions  in Assumptions \ref{assumption:one-side-Lipschitz-condition-for-linear-operator}-\ref{assumption:growth-condition-of-frechet-derivatives-of-drift-and-diffusion} are meet
with $\gamma = 3$ and for any $p_{0} \geq 13$.
We compute the equation \eqref{equation:ginzburg-landau-model} numerically using the explicit projected Euler method, i.e. $\theta=1$ in \eqref{introduction:LTPE-scheme},
and the \textit{exact} solutions are identified with the corresponding numerical approximations at a fine stepsize $h_{exact}=2^{-14}$.
Also, the reference lines of slope $0.5$ and $1$ are given here.
It turns out in
Figure 1 that the weak convergence rate of the approximation errors of the projected Euler method decrease at a slope close to $1$.
\begin{figure}[h] 
\centering
    \includegraphics[width=0.6\textwidth]
      {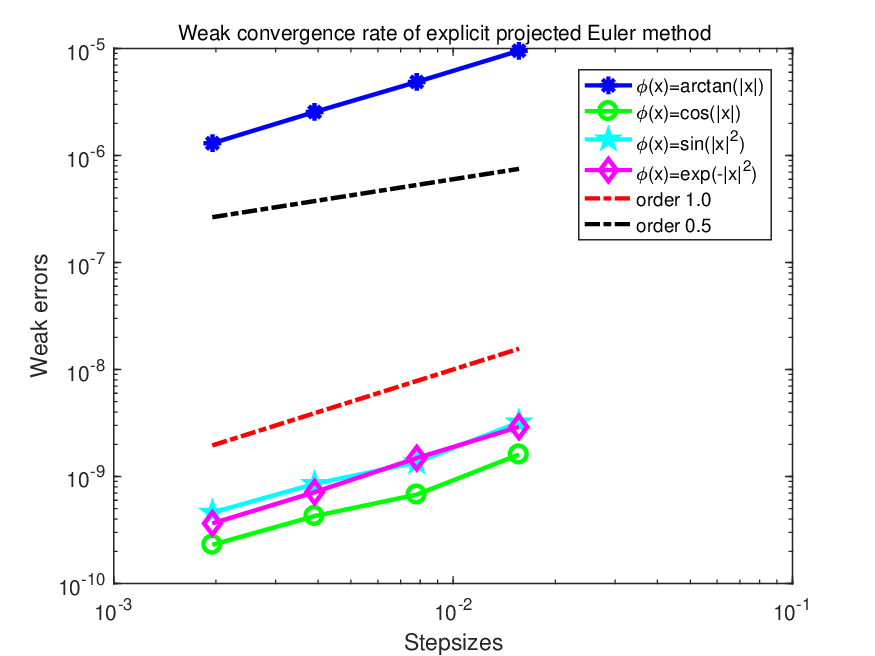}
\caption{Weak convergence rates of the explicit projected Euler method for stochastic Ginzburg-Landau model \eqref{equation:ginzburg-landau-model}}
\end{figure}

\textbf{Example 2.} Consider a scalar mean-reverting type model  with super-linear coefficients in financial and energy markets as follows,
\begin{equation} \label{eq:mean-reverting-model}
\dd X_{t} = 
\left(
b - \alpha X_{t} - \beta X^{3}_{t} 
\right) \dd t 
+ \sigma X^{2}_{t}\, \dd W_t, \quad b, \alpha, \beta, \sigma \in \mathbb{R}.
\end{equation}
Setting $b=0.3$, $\alpha=1$, $\beta=0.6$, $\sigma=0.2$ and $X_{0}=1$. The requirements from Assumptions \ref{assumption:one-side-Lipschitz-condition-for-linear-operator}-\ref{assumption:growth-condition-of-frechet-derivatives-of-drift-and-diffusion} can be verified with $\gamma=3$ and for any $p_{0} \in [13, 31/2]$.
We begin with  the probability density test of the LTPE sheme \eqref{introduction:LTPE-scheme} to 
discrete model \eqref{eq:mean-reverting-model} with three different $\theta$, $\theta=0, 0.5, 1$, at the terminal time $T=5$ using a stepsize $h=2^{-14}$, which can be found in Figure 2, respectively. 
Moreover, we put the probability density lines of such three numerical schemes with different choice of $\theta$ together and 
 directly observe that all the probability density lines are almost same so that the choice of time $T=5$ is suitable.

\begin{figure}[h]
\centering
\subfigure{
    \begin{minipage}[t]{0.45\textwidth}
    \centering
    \includegraphics[width=\textwidth]
      {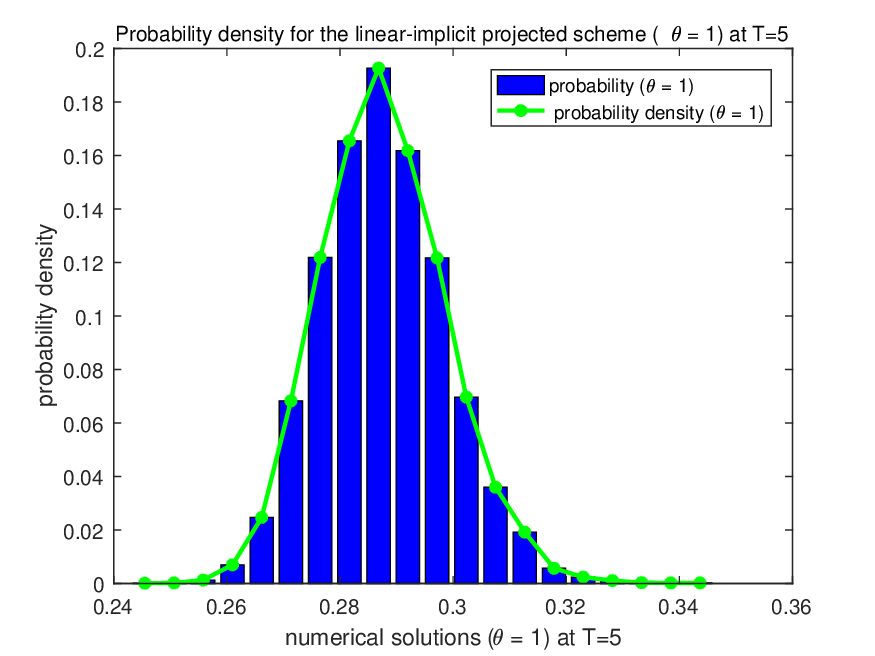}
    \end{minipage}
} 
\subfigure{
    \begin{minipage}[t]{0.45\textwidth}
    \centering
    \includegraphics[width=\textwidth]
      {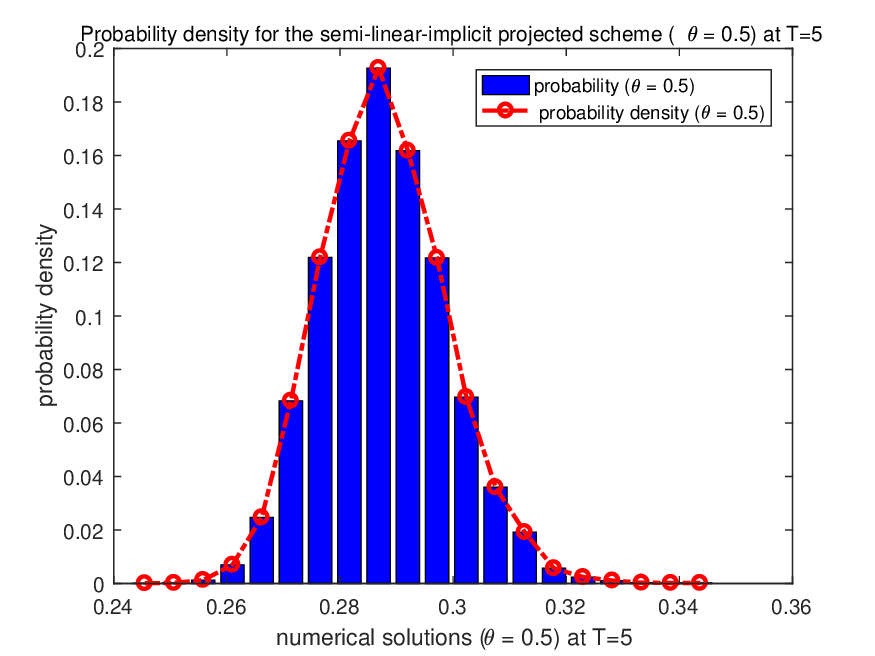}
    \end{minipage}
}
\subfigure{
    \begin{minipage}[t]{0.45\textwidth}
    \centering
    \includegraphics[width=\textwidth]
      {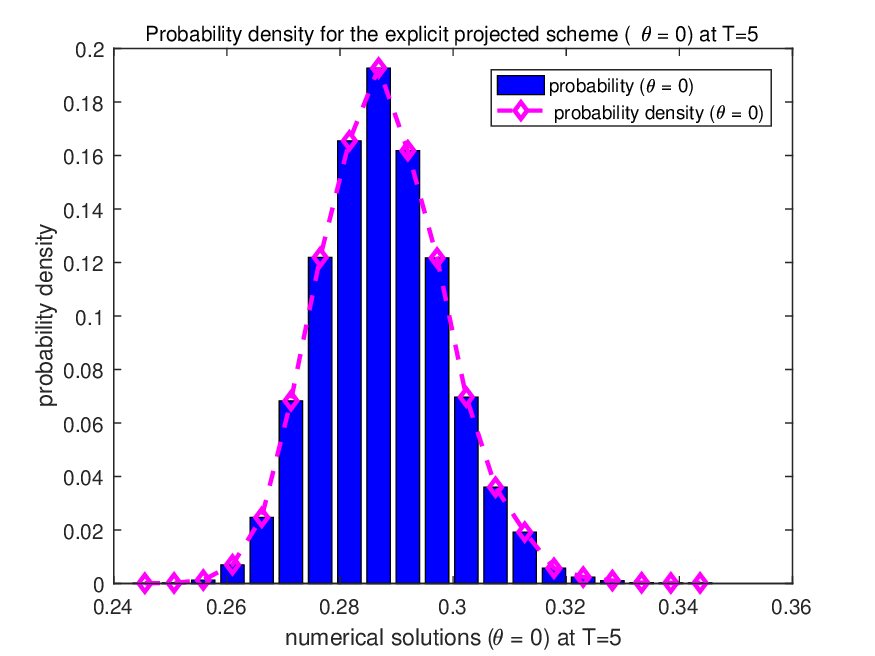}
    \end{minipage}
}
\subfigure{
    \begin{minipage}[t]{0.45\textwidth}
    \centering
    \includegraphics[width=\textwidth]
      {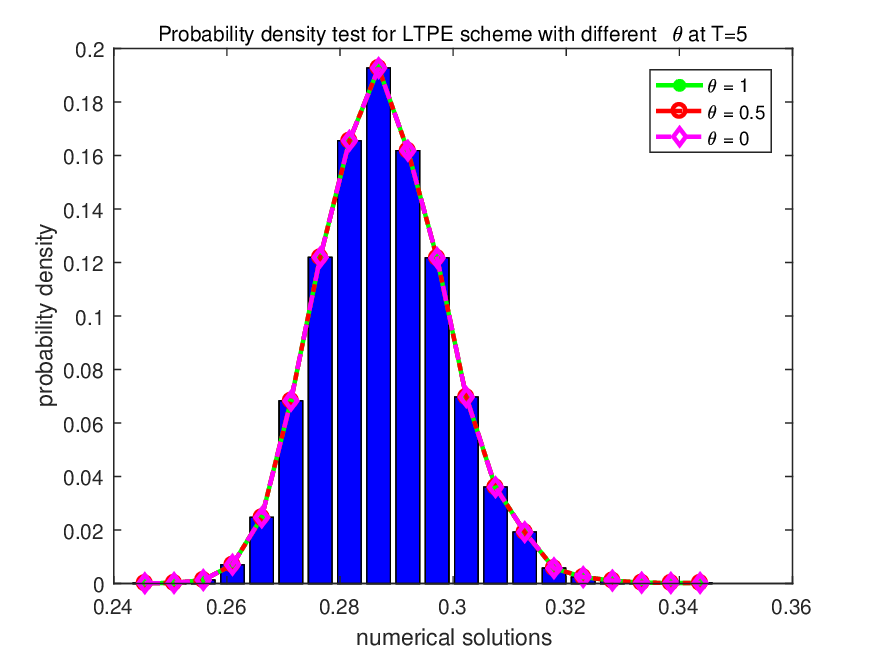}
    \end{minipage}
}
\caption{Probability density of LTPE scheme method for discretizing the mean reverting model \eqref{eq:mean-reverting-model} with different $\theta$.}
\end{figure}

We discrete this model \eqref{eq:mean-reverting-model} by the semi-linear-implicit projected Euler method (i.e. $\theta = 0.5$ in \eqref{introduction:LTPE-scheme}).
To find the \textit{exact} solutions, we discrete this model by the linear-implicit projected Euler method ($\theta = 1$ in \eqref{introduction:LTPE-scheme}) at a fine stepsize $h_{exact}=2^{-14}$. 
In Figure 3, the weak error lines have slopes close to $1$ for all cases.

\begin{figure}[h]
\centering
    \includegraphics[width=0.6\textwidth]
      {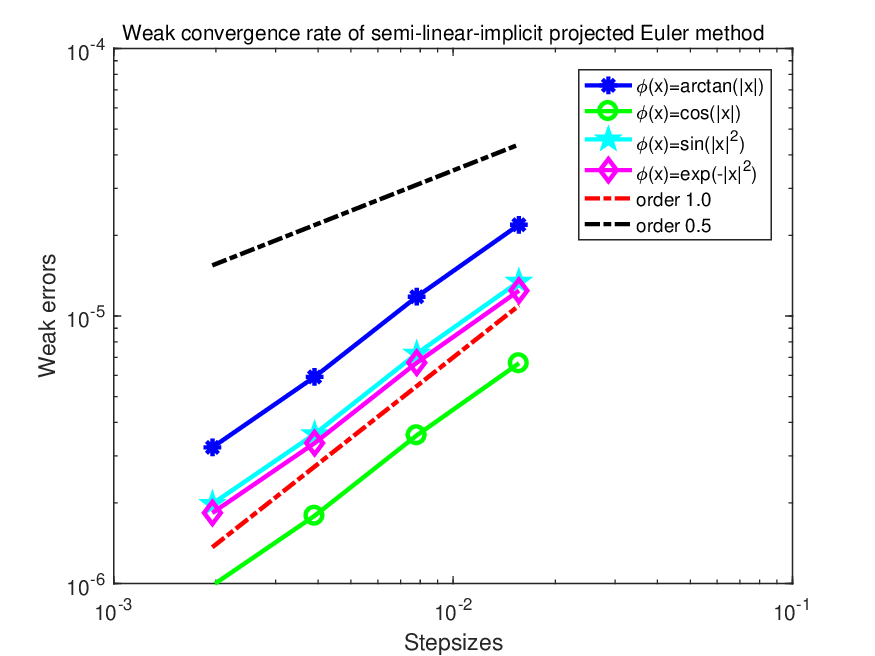}
\caption{Weak convergence rates of semi-linear-implicit projected Euler method for the mean reverting model \eqref{eq:mean-reverting-model}}
\end{figure}

\textbf{Example 3.} Consider the following semi-linear stochastic partial diﬀerential equation (SPDE),
\begin{equation} \label{equation:SPDE}
\left\{\begin{array}{l}
\mathrm{d} u(t, x)=\left[\frac{\partial^2}{\partial x^2} u(t, x)+u(t, x)-u^3(t, x)\right] \mathrm{d} t+g(u(t, x)) \mathrm{d} W_t, \quad t \in(0, T], \ x \in(0,1), \\
u(t, 0)=u(t, 1)=0 \\
u(0, x)=u_0(x)
\end{array}\right.
\end{equation}
where $g: \mathbb{R}\rightarrow \mathbb{R}$ and $W_{\cdot}: [0,T]\times \Omega \rightarrow \mathbb{R}$ is the real-valued standard Brownian motions. Such an SPDE is usually termed as the stochastic Allen-Cahn equation.
Discretizing such SPDE \eqref{equation:SPDE}
spatially by a finite difference method yields a system of SDE as below,
\begin{equation}
\label{eq:SPDE-SODE-system}
\dd  X_t = [\mathbb{A} X_t + \mathbb{F} ( X_ t ) ] \, \dd t + \mathbb{G} ( X_t ) \, \dd W_t,
\quad
t \in (0, T],
\quad
X_0 = x_0,
\end{equation}
where
$X_t 
=( X_{1,t},  X_{2,t}, \cdots, X_{K-1,t} ) ^T
:= (u(t, x_1),  u(t, x_2), \cdots, u(t, x_{K-1}) )^T
$,
$\mathbb{A} \in \mathbb{R}^{ (K-1) \times (K-1) }$,
$x_0 = (u_0 ( x_1 ), u_0 ( x_2 ), ..., u_0 ( x_{K-1} ) )^T $
and 
\begin{equation*}
\mathbb{A} 
= 
K^2 \left[\begin{array}{cccccc}
-2 & 1 & 0 & \cdots & 0 & 0   \\
1 & -2 & 1 & \cdots & 0 & 0 \\
0 & 1 & -2 & \cdots & 0 & 0 \\
 &  \cdots &  & \cdots &  &  \\
0 & 0 & 0 &  \cdots & -2 & 1 \\
0 & 0 & 0 &  \cdots & 1 & -2
\end{array}\right],\
\:
\mathbb{F} ( X ) = \left[\begin{array}{c} X_1 - (X_1)^3 \\ X_2 - (X_2)^3  \\ \vdots \\ X_{K-1} - (X_{K-1})^3 \end{array}\right],\
\:
\mathbb{G} ( X ) = \left[\begin{array}{c}  g(X_1) \\  g( X_2)   \\ \vdots \\  g ( X_{K-1} ) \end{array}\right].
\end{equation*}
Here we only focus on the temporal
discretization of the SDE system \eqref{eq:SPDE-SODE-system}.
In what follows we set $g (u) = \sin (u) + 1$ and $u_0 ( x ) \equiv 1$.
The eigenvalues $\{\lambda_{i} \}^{K-1}_{i=1}$ of the matrix $\mathbb{A}$ are $\lambda_{i}=-4K^{2} \sin^{2}(i\pi/2K) <0$ \cite{wang2023mean}, resulting in a very stiff system \eqref{eq:SPDE-SODE-system}.
Further, it is easy to check all conditions in Assumptions \ref{assumption:one-side-Lipschitz-condition-for-linear-operator}-\ref{assumption:growth-condition-of-frechet-derivatives-of-drift-and-diffusion} are fulfilled 
with $\gamma = 3$ and for any $p_{0} \geq 13$.

Here  we take the case $K=4$ as an example.
To deal with the stiffness, we take  the linear-implicit projected Euler method, i,e, $\theta=1$ in \eqref{introduction:LTPE-scheme}, to discretize \eqref{eq:SPDE-SODE-system} in time and the \textit{exact} solutions are given numerically by using a fine stepsize $h_{\text{exact}} = 2^{-14}$.
As can be observed from Figure 4, the weak convergence rate of the linear-implicit Euler method is $1$.

\begin{figure}[h]
\centering
    \includegraphics[width=0.6\textwidth]
      {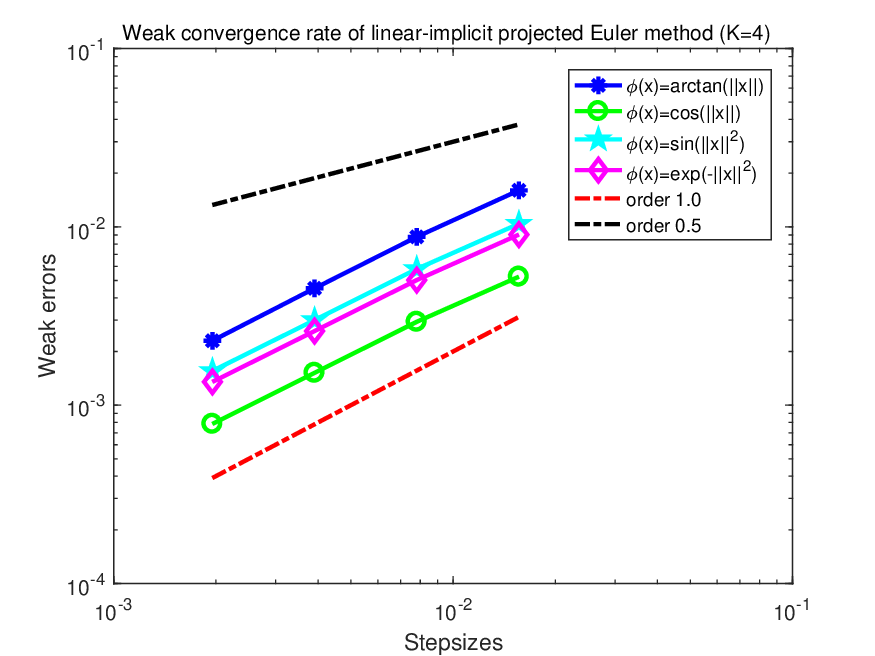}
\caption{Weak convergence rates of linear-implicit projected Euler method for model \eqref{eq:SPDE-SODE-system} (K=4).}
\end{figure}

\bibliographystyle{elsarticle-num}
\bibliography{refer}


\appendix
\section{Proof of Lemmas in Section \ref{subsection: Invariant measure of semi-linear SDE}}
\subsection{Proof of Lemma \ref{lemma:uniform-moments-bound-of-SDEs}}
\label{proof-of-lemma:uniform-moments-bound-of-SDEs}
\begin{proof}[Proof of Lemma \ref{lemma:uniform-moments-bound-of-SDEs}]

By the It\^o formula and the Cauchy-Schwarz inequality, for any $p \in [1 , \infty)$, we get
\begin{small}
\begin{equation}
\begin{aligned}
& \left(
1 + \big\| X^{-\iota, x_{0}}_{t} \big|^{2} 
\right)^{p} \\
& \leq \left( 1 + \left\| x_{0} \right\|^{2} \right)^{p} 
+ 2p \int_{-\iota}^{t} 
\left(
1 + \left\| X^{-\iota, x_{0}}_{s} \right\|^{2} 
\right)^{p-1} 
\left\langle 
X^{-\iota, x_{0}}_{s} , AX^{-\iota, x_{0}}_{s} 
\right\rangle \dd s  \\
&\quad + 2p \int_{-\iota}^{t} 
\left(
1 + \left\| X^{-\iota, x_{0}}_{s} \right\|^{2} 
\right)^{p-1} 
\left\langle X^{-\iota, x_{0}}_{s} , 
f \left( X^{-\iota, x_{0}}_{s} \right) 
\right\rangle \dd s \\
&\quad
+ 2p \int_{-\iota}^{t} 
\left(
1 + \left\| X^{-\iota, x_{0}}_{s}  \right\|^{2} 
\right)^{p-1} 
\left\langle 
X^{-\iota, x_{0}}_{s}  , 
g \left(X^{-\iota, x_{0}}_{s} \right)   \dd \widetilde{W}_{s} 
\right\rangle \\
&\quad + p(2p-1) \int_{-\iota}^{t} 
\left(
1 + \left\| X^{-\iota, x_{0}}_{s}  \right\|^{2} 
\right)^{p-1} 
\|  g (X^{-\iota, x_{0}}_{s} )  \|^{2} \dd s .
\end{aligned}
\end{equation}
\end{small}
Here we define a stopping time as 
\begin{equation}
\tau_{n} =
\inf \{ s \geq -\iota : \| X^{-\iota, x_{0}}_{s}  \| > n  \}.
\end{equation} 
Taking expectations on both sides with \eqref{equation:growth-conditoin-for-linear-operator} and Assumption \ref{assumption:coercivity-condition-for-drift-and-diffusion} shows that
\begin{footnotesize}
\begin{equation}
\begin{aligned}
&\mathbb{E} 
\left[
\left( 
1 + \big\| X^{-\iota, x_{0}}_{t \land \tau_{n}} \big\|^{2} 
\right)^{p} 
\right] \\
&\leq 
\mathbb{E} 
\left[
\left
( 1 + \left\| x_{0} \right\|^{2} 
\right)^{p} 
\right] 
- 2p \lambda_{1} 
\mathbb{E}  
\left[ 
\int_{-\iota}^{t \land \tau_{n}} 
\left(
1 + \left\| X^{-\iota, x_{0}}_{s} \right\|^{2} 
\right)^{p-1} 
\left\| X^{-\iota, x_{0}}_{s}  \right\|^{2} \dd s 
\right] 
+ p L_{1} 
\mathbb{E}  
\left[ 
\int_{-\iota}^{t \land \tau_{n}} 
\left(
1 + \left\| X^{-\iota, x_{0}}_{s}  \right\|^{2} 
\right)^{p}  \dd s 
\right]  \\
& = \mathbb{E} 
\left[
\left(
1 + \left\| x_{0} \right\|^{2} 
\right)^{p} 
\right] 
- p\left( 2\lambda_{1}-L_{1} \right) 
\mathbb{E}  
\left[
\int_{-\iota}^{t \land \tau_{n}} 
\left( 
1 + \left\| X^{-\iota, x_{0}}_{s}  \right\|^{2} 
\right)^{p}  \dd s 
\right] 
+ 
2p \lambda_{1} 
\mathbb{E}  
\left[ 
\int_{-\iota}^{t \land \tau_{n}} 
\left(
1 + \left\| X^{-\iota, x_{0}}_{s}  \right\|^{2} 
\right)^{p-1}  \dd s 
\right] .
\end{aligned}
\end{equation}
\end{footnotesize}
For $p \in [1,p_{0}]$,
using the Young inequality
\begin{equation}
a^{p-1}b \leq 
\epsilon \tfrac{p-1}{p}a^{p} 
+ \epsilon^{1-p} 
\tfrac{b^{p}}{p},
\quad \forall a,b \geq 1 \quad
\text{with} \quad
\epsilon \in 
\left(
0, \tfrac{p(2\lambda_{1}-L_{1})}{2(p-1)\lambda_{1}} 
\right)
\end{equation}
to indicate that, 
\begin{small}
\begin{equation}
\begin{aligned}
2p\lambda_{1}
\mathbb{E}  
\left[
\int_{-\iota}^{t \land \tau_{n}} 
\left(
1 + \left\| X^{-\iota, x_{0}}_{s} \right\|^{2} 
\right)^{p-1}  \dd s 
\right] 
\leq 2(p-1)\lambda_{1}\epsilon 
\mathbb{E}  
\left[ 
\int_{-\iota}^{t \land \tau_{n}} 
\left(
1 + \left\|X^{-\iota, x_{0}}_{s}\right\|^{2} 
\right)^{p}  \dd s 
\right] 
+ \int_{-\iota}^{t \land \tau_{n}} 
2\lambda_{1}\epsilon^{1-p}  \dd s.
\end{aligned}
\end{equation}
\end{small}
Then one can achieve that
\begin{equation}
\begin{aligned}
\mathbb{E} 
\left[
\left( 
1 + \left\| X^{-\iota, x_{0}}_{t \land \tau_{n}} \right\|^{2} 
\right)^{p} 
\right] 
&+ \big[
p\left( 2\lambda_{1}-L_{1} \right) 
- 2(p-1)\lambda_{1}\epsilon 
\big] 
\mathbb{E}  
\left[ 
\int_{-\iota}^{t \land \tau_{n}} 
\left( 
1 + \left\| X^{-\iota, x_{0}}_{s} \right\|^{2} 
\right)^{p}  \dd s 
\right] \\
& \leq 
\mathbb{E} 
\left[
\left(
1 + \left\| x_{0} \right\|^{2} 
\right)^{p} 
\right] 
+ \int_{-\iota}^{t \land \tau_{n}} 
2\lambda_{1}\epsilon^{1-p}  \dd s .
\end{aligned}
\end{equation}
Due to the  Fatou Lemma, let $n \rightarrow \infty$, we obtain that
\begin{equation}
\begin{aligned}
\mathbb{E} 
\left[
\left(
1 + \left\| X^{-\iota, x_{0}}_{t} \right\|^{2} 
\right)^{p} 
\right] 
&+ \big[
p\left( 2\lambda_{1}-L_{1} \right) 
- 2(p-1)\lambda_{1}\epsilon 
\big] 
\mathbb{E}  
\left[
\int_{-\iota}^{t } 
\left(
1 + \left\| X^{-\iota, x_{0}}_{s} \right\|^{2} 
\right)^{p}  \dd s 
\right] \\
& \leq 
\mathbb{E} 
\left[
\left(
1 + \left\| x_{0} \right\|^{2} 
\right)^{p} 
\right] 
+ \int_{-\iota}^{t } 2\lambda_{1}\epsilon^{1-p} \dd s.
\end{aligned}
\end{equation}
As $p\left( 2\lambda_{1}-L_{1} \right) - 2(p-1)\lambda_{1}\epsilon > 0$, the proof can be done by Lemma \ref{lemma:ito}. 
\end{proof}
\subsection{Proof of Lemma \ref{lemma:contractivity-of-sde}}
\label{proof-of-lemma:contractivity-of-sde}
\begin{proof}[Proof of Lemma \ref{lemma:contractivity-of-sde}]
For brevity, we define 
\begin{equation}
\Delta x _{t} = X^{-\iota, x^{(1)}_{0}}_{t}  
- X^{-\iota, x^{(2)}_{0}}_{t} , \quad 
\Delta f _{t} = 
f \Big( X^{-\iota, x^{(1)}_{0}}_{t}  \Big) 
- f\Big( X^{-\iota, x^{(2)}_{0}}_{t} \Big), \quad 
\Delta g _{t} = g \Big( X^{-\iota, x^{(1)}_{0}}_{t}  \Big) 
- g\Big( X^{-\iota, x^{(2)}_{0}}_{t} \Big).
\end{equation}
With the stopping time defined as follows,
\begin{equation}
\overline{\tau}^{(1)}_{n} =
\inf \{ s \geq -\iota : \| X^{-\iota, x^{(1)}_{0}}_{s}  \| > n \quad
\text{or} \quad
\| X^{-\iota, x^{(2)}_{0}}_{s}  \| > n
\},
\end{equation} 
one obtains by using the It\^o formula,
\begin{equation}
\begin{aligned}
&e^{c_{1}p (t \land \overline{\tau}^{(1)}_{n})
} 
\left\|
\Delta x _{t \land \overline{\tau}^{(1)}_{n}} 
\right\|^{2p} \\
&\leq \| \Delta x _{0} \|^{2p} 
+ c_{1}p\int_{-\iota}^{
t\land \overline{\tau}^{(1)}_{n}
} e^{c_{1}p  s} 
\| \Delta x _{s}  \| ^{2p} \dd s 
+ 2p \int_{-\iota}^{
t \land \overline{\tau}^{(1)}_{n}
} e^{c_{1}p  s} 
\| \Delta x _{s}  \| ^{2p-2} 
\left\langle 
\Delta x _{s} , A \Delta x _{s} 
\right\rangle \dd s  \\
& \quad + 2p \int_{-\iota}^{
t \land \overline{\tau}^{(1)}_{n}
} e^{c_{1}p s} 
\| \Delta x _{s}  \| ^{2p-2} 
\left\langle 
\Delta x _{s} , \Delta f _{s} 
\right\rangle \dd s 
+ p (2p-1) 
\int_{-\iota}^{
t \land \overline{\tau}^{(1)}_{n}
} e^{c_{1}p  s} 
\| \Delta x _{s}  \| ^{2p-2} 
\left\| \Delta g_{s} \right\|^{2} \dd s .
\end{aligned}
\end{equation}
Hence, by taking expectations on both sides with Assumptions \ref{assumption:one-side-Lipschitz-condition-for-linear-operator}, \ref{assumption:coupled-monotoncity-for-drift-and-diffusion} and the Fatou lemma, we reach that, for some positive constant $c_{1}\in (0,2\lambda_{1}-L_{2}]$,
\begin{equation}
\begin{aligned}
\mathbb{E} 
\left[ 
e^{c_{1}p (t \land \overline{\tau}^{(1)}_{n})} \left\| 
\Delta x _{t \land \overline{\tau}^{(1)}_{n}} \right\|^{2p}
\right] 
&\leq \mathbb{E} 
\left[ \| \Delta x _{0} \|^{2p} \right] 
+ \underbrace{
p \big[c_{1} - \left( 2\lambda_{1} - L_{2} \right) \big] 
\int_{-\iota}^{
t \land \overline{\tau}^{(1)}_{n}
} e^{c_{1}p s} 
\| \Delta x _{s}  \| ^{2p} \dd s}_{\leq 0},\\
\end{aligned}
\end{equation}
leading to
\begin{equation}
\begin{aligned}
\mathbb{E} 
\left[ 
\left\| 
\Delta x _{t } \right\|^{2p}
\right] 
\leq 
e^{-c_{1}p t } 
\mathbb{E} 
\left[ \| \Delta x _{0} \|^{2p} \right].
\end{aligned}
\end{equation}
The proof is completed.
\end{proof}

\subsection{Proof of Lemma \ref{lemma:cauchy-sequence-of-sode}}
\label{proof-of-lemma:cauchy-sequence-of-sode}
\begin{proof}[Proof of Lemma \ref{lemma:cauchy-sequence-of-sode}]
Let 
\begin{equation}
\begin{aligned}
&\Delta X_{t} =  X^{-s_{1}, x_{0}}_{t}  - X^{-s_{2}, x_{0}}_{t}, \ 
\Delta \bar{f}_{t} 
= f\left(X^{-s_{1}, x_{0}}_{t} \right) - f\left(X^{-s_{2}, x_{0}}_{t} \right), \
\Delta \bar{g}
= g\left(X^{-s_{1}, x_{0}}_{t} \right) - g\left(X^{-s_{2}, x_{0}}_{t} \right).
\end{aligned}
\end{equation}
With reference to the proof of Lemma \ref{lemma:contractivity-of-sde}, 
setting the stopping time as 
\begin{equation}
\overline{\tau}^{(2)}_{n} 
=
\inf \{ s \geq -s_{2} : \| X^{-s_{2}, x_{0}}_{s}  \| > n  \},
\end{equation}
by the It\^o formula, we deduce that
\begin{footnotesize}
\begin{equation}
\begin{aligned}
&\mathbb{E} 
\left[
e^{c_{2}p(
t \land \overline{\tau}^{(2)}_{n}+s_{2}
)
} 
\left\| 
\Delta X _{
t \land \overline{\tau}^{(2)}_{n}
} 
\right\|^{2p} 
\right] \\
&\leq \mathbb{E}
\left[
\| \Delta X _{-s_{2}} \|^{2p}
\right] 
+ c_{2}p 
\mathbb{E}
\left[
\int_{-s_{2}}^{
t \land \overline{\tau}^{(2)}_{n}
} e^{c_{2}p(s+s_{2})} 
\| \Delta X _{s}  \| ^{2p} \dd s 
\right]  
+ 2p \mathbb{E}
\left[
\int_{-s_{2}}^{
t \land \overline{\tau}^{(2)}_{n}
} e^{c_{2}p(s+s_{2})} 
\| \Delta X _{s}  \| ^{2p-2} 
\left\langle 
\Delta X _{s} , A \Delta X _{s} 
\right\rangle \dd s 
\right] \\
& \quad + 2p\mathbb{E}
\left[ 
\int_{-s_{2}}^{t \land \overline{\tau}^{(2)}_{n}} 
e^{c_{2}p(s+s_{2})} 
\| \Delta X _{s}  \| ^{2p-2} 
\left\langle 
\Delta X _{s} , \Delta \bar{f}_{s} 
\right\rangle \dd s 
\right]
+p (2p-1)\mathbb{E}
\left[  
\int_{-s_{2}}^{t \land \overline{\tau}^{(2)}_{n}} 
e^{c_{2}p(s+s_{2})} 
\| \Delta X _{s}  \| ^{2p-2} 
\left\| \Delta \bar{g}_{s} \right\|^{2} \dd s 
\right].
\end{aligned}
\end{equation}
\end{footnotesize}
According to Lemma \ref{lemma:uniform-moments-bound-of-SDEs}, one can directly obtain 

\begin{equation}
\mathbb{E}\left[\| \Delta X _{-s_{2}} \|^{2p}\right] = \mathbb{E}\left[\| X^{-s_{1}, x_{0}}_{-s_{2}}-x_{0} \|^{2p}\right] \leq  C\mathbb{E} \left[ \left( 1 + \left\| x_{0} \right\|^{2} \right)^{p} \right].
\end{equation}
Taking this with Assumption \ref{assumption:one-side-Lipschitz-condition-for-linear-operator},  Assumption \ref{assumption:coupled-monotoncity-for-drift-and-diffusion} and the Fatou lemma into account yields
\begin{small}
\begin{equation}
\begin{aligned}
e^{c_{2}p(
t \land \overline{\tau}^{(2)}_{n}
+s_{2}
)
} 
\mathbb{E} 
\left[
\Big\| 
\Delta X _{
t \land \overline{\tau}^{(2)}_{n}
} 
\Big\|^{2p} 
\right]
&\leq C\mathbb{E} 
\left[ 
\left( 1 + \left\| x_{0} \right\|^{2} \right)^{p} 
\right] 
+ 
\underbrace{
p \left[
c_{2} - \left( 2\lambda_{1} - L_{2} \right) 
\right] 
\int_{-s_{2}}^{
t \land \overline{\tau}^{(2)}_{n}
} e^{c_{2}p(s+s_{2})} \| \Delta X _{s}  \| ^{2p} \dd s 
}_{\leq 0},\\
\end{aligned}
\end{equation}
\end{small}
resulting in 
\begin{equation}
\begin{aligned}
\mathbb{E} 
\left[
\left\| 
\Delta X _{t} 
\right\|^{2p} 
\right]
\leq 
C
e^{-c_{2}p(t+s_{2})
} 
\mathbb{E} 
\left[
\left(
1 + \left\| x_{0} \right\|^{2} 
\right)^{p} 
\right]
\end{aligned}
\end{equation}
The proof is completed.
\end{proof}
\section{Proof of Lemmas in Section \ref{subsection:Invariant-measure-of-the-numerical-scheme} }
\subsection{Proof of Lemma \ref{lemma:necessary-estimates}}
\label{proof-of-lemma:necessary-estimates}
\begin{proof}[Proof of Lemma \ref{lemma:necessary-estimates}]
The first and the second  estimates are obvious from \eqref{equation:growth-of-the-drift-f} and \eqref{eq:projected-operator}. Equipped with these above, by Assumption \ref{assumption:coercivity-condition-for-drift-and-diffusion}
and the Cauchy-Schwarz inequality, one obtains
\begin{equation}
\begin{aligned}
(2p_{0}-1)\|g(\mathscr{P}(x)) \|^{2}  
&\leq L_{1} (1 + \| \mathscr{P}(x)\|^{2}) 
- 2 \big\langle 
\mathscr{P}(x), f\big(\mathscr{P}(x)\big)  
\big\rangle \\
& \leq  
L_{1} (1 + \| \mathscr{P}(x)\|^{2}) 
+ 2 \big\| \mathscr{P}(x) \big\| \cdot
\big\|
f\big(\mathscr{P}(x)\big)
\big\|  \\
& \leq L_{1} (1 + \| \mathscr{P}(x)\|^{2}) + 2C_{f}h^{-\frac{1}{2}} \|\mathscr{P}(x)\|.
\end{aligned}
\end{equation}
Owing to the fact that $p_{0} \in [1, \infty)$, the proof of the third estimate in \eqref{equation:first-three-estimates-in-lemma} is completed. Then taking $p$-th square on both sides yields
\begin{equation}
\begin{aligned}
(2p_{0}-1)^{p}\|g(\mathscr{P}(x)) \|^{2p} 
 &\leq L_{1} (1 + \| \mathscr{P}(x)\|^{2})^{p} + 2p C_{f} L_{1}^{p-1}
 h^{-\frac{1}{2}} (1 + \| \mathscr{P}(x)\|^{2})^{p-1} \|\mathscr{P}(x) \| \\
&\quad + \sum_{i=2}^{p} 
\mathscr{C}_{p}^{i} 
(2C_{f})^{i} L_{1}^{p-i} 
h^{-\frac{i}{2}} 
(1 + \| \mathscr{P}(x)\|^{2})^{p-i}  
\|\mathscr{P}(x) \|^{i},
\end{aligned}
\end{equation}
where $\mathscr{C}_{p}^{i}:= \frac{p!}{i! (p-i)!}$.
As we have claimed, $\|\mathscr{P}(x) \| \leq h^{-\frac{1}{2\gamma}}$ and 
$\|\mathscr{P}(x) \|^{i} \leq (1 + \|\mathscr{P}(x) \|^{2}) ^{\frac{i}{2}}$ for any $i \geq 2$, so that
\begin{small}
\begin{equation}
\begin{aligned}
& (2p_{0}-1)^{p}\|g(\mathscr{P}(x)) \|^{2p} \\
&\leq L_{1}^{p} 
(1 + \| \mathscr{P}(x)\|^{2})^{p} 
+ 2 C_{f} p L_{1}^{p-1} 
h^{- \left(\frac{1}{2} + \frac{1}{2\gamma} \right)} 
(1 + \| \mathscr{P}(x)\|^{2})^{p-1}  
+ h^{-\frac{p}{2}} (1 + \| \mathscr{P}(x)\|^{2})^{p-1}  \sum_{i=2}^{p}  \mathscr{C}_{p}^{i} 
(2C_{f})^{i} L_{1}^{p-i}   \\
& \leq  L_{1}^{p} 
(1 + \| \mathscr{P}(x)\|^{2})^{p} 
+ C h^{-\frac{p}{2}} (1 + \| \mathscr{P}(x)\|^{2})^{p-1},
\end{aligned}
\end{equation}
\end{small}
where $C=C(L_{1}, C_{f}, p)=\sum_{i=1}^{p}  \mathscr{C}_{p}^{i}  (2C_{f})^{i} L_{1}^{p-i}$. 

Turning now on to the estimate \eqref{eq:estimates-in-lemma-projected-Lipschitz},
the proof of the first estimate in  \eqref{eq:estimates-in-lemma-projected-Lipschitz} can be found from Lemma 6.2 in \cite{beyn2016stochastic}.
For the second estimate, we know from \eqref{introduction:LTPE-scheme},  Assumption \ref{assumption:growth-condition-of-frechet-derivatives-of-drift-and-diffusion}  and Lemma \ref{lemma:necessary-estimates} that,
\begin{equation}
\begin{aligned}
\left\| f\big(\mathscr{P}(x)\big) - f\big(\mathscr{P}(y)\big)
\right\| 
& \leq C_{1} \big(
1 + \|\mathscr{P}(x) \|^{\gamma-1} 
+ \|\mathscr{P}(y) \|^{\gamma-1}  
\big) 
\| \mathscr{P}(x) - \mathscr{P}(x)  \| \\
& \leq C_{1} (1 + 2h^{-\frac{\gamma-1}{2\gamma}}) 
\| \mathscr{P}(x) - \mathscr{P}(x)  \| \\
&\leq \lambda_{f}  h^{-\frac{\gamma- 1}{2\gamma}} \|
\mathscr{P}(x) - \mathscr{P}(x)  \|,
\end{aligned}
\end{equation}
where one can follow the first estimate to complete the proof.
The proof is completed.
\end{proof}
%
%
\subsection{Proof of Lemma \ref{lemma:contractivity-of-the-theta-linear-projected-Euler-method}}
\label{proof-of-lemma:contractivity-of-the-theta-linear-projected-Euler-method}
\begin{proof}[Proof of Lemma \ref{lemma:contractivity-of-the-theta-linear-projected-Euler-method}]
Shortly, we denote
\begin{equation}
\begin{aligned}
&\Delta Y_{n} = Y_n^{(1)} - Y_n^{(2)}, 
\quad \Delta \mathscr{P}(Y_{n}) 
= \mathscr{P}(Y^{(1)}_{n}) - \mathscr{P}(Y^{(2)}_{n}),  
\quad \Delta \widetilde{f}_{n} 
= f\big(\mathscr{P}(Y^{(1)}_{n})\big)
-f\big(\mathscr{P}(Y^{(2)}_{n})\big),\\
& \Delta \widetilde{g}_{n} 
= g\big(\mathscr{P}(Y^{(1)}_{n})\big)
- g\big(\mathscr{P}(Y^{(2)}_{n})\big) . \\
\end{aligned}
\end{equation}
It is apparent to show that
\begin{equation}
\begin{aligned}
\Delta Y_{n+1} - \theta A \Delta Y_{n+1}h = \Delta \mathscr{P}(Y_{n}) + 
(1-\theta) A \Delta \mathscr{P}(Y_{n}) h+
\Delta \widetilde{f}_{n} h + \Delta \widetilde{g}_{n} \Delta W_{n}.
\end{aligned}
\end{equation}
Taking square on both sides, we then take expectations and follow Assumption \ref{assumption:one-side-Lipschitz-condition-for-linear-operator} and Assumption \ref{assumption:coupled-monotoncity-for-drift-and-diffusion} to imply
\begin{equation}
\begin{aligned}
&(1 + 2\theta \lambda_{1}h + \theta^{2} h ^{2}) \mathbb{E} \left[ \left\|\Delta Y_{n+1}  \right\|^{2} \right] \\
&\leq \left[ 1- 2(1-\theta)\lambda_{1}h \right] \mathbb{E} \left[ \| \Delta \mathscr{P}(Y_{n}) \|^{2} \right] +
(1-\theta)^{2} h^{2} \mathbb{E} \left[  \| A \Delta \mathscr{P}(Y_{n}) \|^{2} \right] 
+  h^{2} \mathbb{E} \left[  \|  \Delta \widetilde{f}_{n} \|^{2} \right]  \\
& \quad +
h \mathbb{E} \left[ \left\|\Delta \widetilde{g}_{n} \right\|^{2} \right]+ 2h   \mathbb{E} \big[  \langle \Delta \mathscr{P}(Y_{n}) , \Delta \widetilde{f}_{n}  \rangle \big] + 2 (1-\theta) h^{2}   \mathbb{E} \big[  \langle A \Delta \mathscr{P}(Y_{n}) , \Delta \widetilde{f}_{n}  \rangle \big].
\end{aligned}
\end{equation}
Using the Cauchy-Schwarz inequality leads to
\begin{equation}
\begin{aligned}
2 (1-\theta) h^{2}   
\mathbb{E} 
\big[
\langle A \Delta \mathscr{P}(Y_{n}) , 
\Delta \widetilde{f}_{n}  \rangle 
\big] 
\leq 
2(1-\theta) h^{2} 
\mathbb{E} 
\left[  
\big\| A \Delta \mathscr{P}(Y_{n})\big\| \cdot
\big\|\Delta \widetilde{f}_{n} \big\| 
\right] .
\end{aligned}
\end{equation}
Recalling Assumption \ref{assumption:one-side-Lipschitz-condition-for-linear-operator}, Assumption \ref{assumption:coupled-monotoncity-for-drift-and-diffusion}, Lemma \ref{lemma:necessary-estimates}, we can obtain that
\begin{equation}
\begin{aligned}
&(1 + 2\theta \lambda_{1}h ) \mathbb{E} \left[ \left\|\Delta Y_{n+1}  \right\|^{2} \right] \\
&\leq \left[ 1- 2(1-\theta)\lambda_{1}h \right] \mathbb{E} \left[ \| \Delta \mathscr{P}(Y_{n}) \|^{2} \right] + 
2h   \mathbb{E} \big[  \langle \Delta \mathscr{P}(Y_{n}) , \Delta \widetilde{f}_{n}  \rangle \big] + (2p_{1}-1)h \mathbb{E} \left[ \left\|\Delta \widetilde{g}_{n} \right\|^{2} \right] \\
& \quad + (1-\theta)^{2} \lambda^{2}_{d}h^{2} \mathbb{E} \left[  \|  \Delta \mathscr{P}(Y_{n}) \|^{2} \right] + 2(1-\theta) \lambda_{d} \lambda_{f}  h^{1+\frac{\gamma+1}{2\gamma}} \mathbb{E} \left[  \|  \Delta \mathscr{P}(Y_{n}) \|^{2} \right]  + \lambda_{f}^{2}h^{1+\frac{1}{\gamma}} \mathbb{E} \left[  \|  \Delta \mathscr{P}(Y_{n}) \|^{2} \right] \\
& \leq \left\{ 1- 2(1-\theta)\lambda_{1}h +  L_{2}h + \left[ (1-\theta)\lambda_{d}h^{\frac{1}{2}} + \lambda_{f} h^{\frac{1}{2\gamma}} \right]^{2} h\right\} \mathbb{E} \left[ \| \Delta Y_{n} \|^{2} \right].
\end{aligned}
\end{equation}
Here we choose a conditional constant $\kappa \in (0,1)$
such that
\begin{equation}
(1-\theta)\lambda_{d}h^{\frac{1}{2}} < \kappa \sqrt{2\lambda_{1} - L_{2}}, 
\quad 
\lambda_{f} h^{\frac{1}{2\gamma}} < (1-\kappa) \sqrt{2\lambda_{1} - L_{2}},
\end{equation}
which leads to 
\begin{equation}
h \in 
\left(
0, \min\left\{
\tfrac{\kappa^{2}(2\lambda_{1}-L_{2})}{(1- \theta)^{2}\lambda^{2}_{d}}, \tfrac{(1-\kappa)^{2\gamma}(2\lambda_{1}-L_{2})^{\gamma}}{(\lambda_{f})^{2\gamma}},{1} 
\right\} 
\right),
\end{equation}
to ensure
\begin{equation}
   2\lambda_{1} - L_{2}- \left[ (1-\theta)\lambda_{d}h^{\frac{1}{2}} + \lambda_{f} h^{\frac{1}{2\gamma}} \right]^{2}  > 0.
\end{equation}
As a result, there exists some positive constant $\widetilde{C}_{1}$ satisfying
\begin{equation}
    \widetilde{C}_{1} \in \left(0, \tfrac{2\lambda_{1} - L_{2}- \left[ (1-\theta)\lambda_{d}h^{\frac{1}{2}} + \lambda_{f} h^{\frac{1}{2\gamma}} \right]^{2}}{1+2\theta L_{1}h} \right)
\end{equation}
such that
\begin{equation}
\begin{aligned}
\mathbb{E} 
\left[ 
\left\|\Delta Y_{n+1}  \right\|^{2} 
\right] 
& \leq 
(1-\widetilde{C}_{1}h)
\mathbb{E} 
\left[
\left\|\Delta Y_{n}  \right\|^{2} 
\right] 
\leq e^{-\widetilde{C}_{1}t_{n+1}}
\mathbb{E} 
\left[ \big\|x_{0}^{(1)} - x_{0}^{(2)} \big\|^{2} \right].
\end{aligned}
\end{equation}
The proof is completed.


\end{proof}
\section{Proof of Lemmas in Section \ref{subsection: Time-independent weak error analysis}}
\subsection{Proof of Lemma \ref{lemma:differentiability-of-solutions}}
\label{proof-of-lemma:differentiability-of-solutions}
\begin{proof} [Proof of Lemma \ref{lemma:differentiability-of-solutions}]
The existence of the mean-square derivatives up to the third order can be proved in a similar way as shown in \cite{cerrai2001second}.
Based on our assumptions, we would like to obtain the time-independent estimate of the derivatives of solutions $\{X^{x}_{t} \}_{t\in [0,T]}$ given by \eqref{eq:semi-linear-SODE} with respect to the initial condition $x$.

For simplicity, we denote that
\begin{equation}
\eta^{v_{1}} (t,x):=\mathcal{D}X^{x}_{t} v_{1}, 
\quad 
\xi^{v_{1}, v_{2}}(t,x) 
:=  \mathcal{D}^{2}X^{x}_{t}  (v_{1}, v_{2}), 
\quad
\zeta^{v_{1}, v_{2}, v_{3}}(t,x) 
:= \mathcal{D}^{3}X^{x}_{t}( v_{1}, v_{2}, v_{3}).
\end{equation}
\textbf{Part I: estimate of the first variation process}

For the first variation process of SDE \eqref{eq:semi-linear-SODE}, we have
\begin{equation} \label{eq:1st-variation-of-SODE}
\dd \eta^{v_{1}} (t,x)  
= DF( X^{x}_t)\eta^{v_{1}} (t,x) \, \dd t 
+ \sum^{m}_{j=1}
Dg_{j}( X^{x}_t) \eta^{v_{1}} (t,x) \, \dd W_{j, t}, \quad 
\eta^{v_{1}} (0,x) = v_{1}.
\end{equation}
Define a stopping time as 
\begin{equation}
\widetilde{\tau}^{(1)}_{n} =
\inf \left\{
s \in [0,t] : \|\eta^{v_{1}} (s,x) \| > n
\quad 
\text{or}
\quad
\left\| 
X^{x}_{s}   
\right\|  > n  
\right\}.
\end{equation} 
Using the It\^o formula, the Cauchy-Schwarz inequality and \eqref{eq:enhanced-coupled-condition} to attain that, for some $q_{1} \in [1,p_{1}]$ and $\delta >0$,
\begin{equation} \label{eq:ito-formula-for-1st-variation-process}
\begin{aligned}
&\mathbb{E} 
\left[
e^{(\alpha q_{1} t) \land \widetilde{\tau}^{(1)}_{n}}
\|\eta^{v_{1}} (t \land \widetilde{\tau}^{(1)}_{n},x) \| ^{2q_{1}} 
\right] \\
&\leq   
\mathbb{E} 
\left[  \|v_{1}\| ^{2q_{1}}
\right] 
+
\alpha q_{1}
\mathbb{E} 
\left[
\int_{0}^{t \land \widetilde{\tau}^{(1)}_{n}} 
e^{\alpha q_{1} s}
\|\eta^{v_{1}} (s,x) \| ^{2q_{1}} 
\dd s
\right]  \\
&\quad
+ 2q_{1} \mathbb{E} 
\left[
\int_{0}^{t \land \widetilde{\tau}^{(1)}_{n}} 
e^{\alpha q_{1} s}
\|\eta^{v_{1}} (s,x) \| ^{2q_{1}-2} 
\big\langle 
\eta^{v_{1}} (s,x) ,  DF( X^{x}_s) \eta^{v_{1}} (s,x) 
\big\rangle \dd s 
\right]  \\
& \quad + q_{1}(2q_{1}-1) 
\sum^{m}_{j=1}
\mathbb{E} 
\left[
\int_{0}^{t \land \widetilde{\tau}^{(1)}_{n}}    
e^{\alpha q_{1} s}
\|\eta^{v_{1}} (s,x) \| ^{2q_{1}-2} 
\left\| 
Dg_{j}( X^{x}_s)  \eta^{v_{1}} (s,x)   
\right\|^{2} 
\dd s
\right] \\
& \leq 
\mathbb{E} 
\left[
\|v_{1}\| ^{2q_{1}}
\right].
\end{aligned}
\end{equation}
Hence, by Fatou lemma and taking $\alpha_{1}=\alpha /2$, the estimate above leads to 
\begin{equation} \label{eq:uniform-estimate-of-1st-variation-process}
\begin{aligned}
\mathbb{E} 
\left[
\|\eta^{v_{1}} (t,x) \| ^{2q_{1}}
\right]  
& \leq e^{-2\alpha_{1}q_{1} t} 
\mathbb{E} 
\left[ 
\left\| v_{1}\right\|^{2q_{1}} 
\right].
\end{aligned}
\end{equation}

\noindent
\textbf{Part II: estimate of the second variation process}

For the second variation process of  SDE \eqref{eq:semi-linear-SODE}, we then acquire that,
\begin{equation} \label{eq:2nd-variation-of-SODE}
\begin{aligned}
\dd \xi^{v_{1}, v_{2}} (t,x)  
&= \Big(
DF( X^{x}_t)\xi^{v_{1}, v_{2}} (t,x) 
+ D^{2}F( X^{x}_t)\big(
\eta^{v_{1}} (t,x), \eta^{v_{2}} (t,x) 
\big)  
\Big) \dd t\\
&\quad + 
\sum^{m}_{j=1}
\Big(
Dg_{j}( X^{x}_t)\xi^{v_{1}, v_{2}} (t,x) 
+ D^{2}g_{j}( X^{x}_t)\big( 
\eta^{v_{1}} (t,x), \eta^{v_{2}} (t,x) 
\big)   
\Big) \dd W_{j,t}, \quad \xi^{v_{1}, v_{2}} (0,x_{0}) = 0.
\end{aligned}
\end{equation}
Following the same idea as \eqref{eq:ito-formula-for-1st-variation-process}, 
we begin with the definition of the stopping time as follows,
\begin{equation}
\widetilde{\tau}^{(2)}_{n} =
\inf \left\{
s \in [0,t] : \|\xi^{v_{1}, v_{2}} (s,x) \| > n
\quad 
\text{or}
\quad
\left\| 
X^{x}_{s}   
\right\|  > n  
\right\}.
\end{equation} 
Then,
for some $q_{2}\in [1,q_{1})$ and $\delta >0$, by taking the It\^o formula, one will arrive at
\begin{footnotesize}
\begin{equation} \label{eq:ito-formula-for-2nd-variation-process}
\begin{aligned}
&\mathbb{E} 
\left[
\left(
\delta + \| \xi^{v_{1}, v_{2}} (t \land \widetilde{\tau}^{(2)}_{n},x)\|^{2} 
\right)^{q_{2}} 
\right]\\
&\leq \delta^{q_{2}} 
+ 2q_{2}
\mathbb{E} 
\left[
\int_{0}^{t \land \widetilde{\tau}^{(2)}_{n}} 
\left(
\delta + \| \xi^{v_{1}, v_{2}} (s,x)\|^{2} 
\right)^{q_{2}-1} 
\big\langle 
\xi^{v_{1}, v_{2}} (s,x) ,  DF( X^{x}_s) \xi^{v_{1}, v_{2}} (s,x)
\big\rangle \dd s 
\right] \\
& \quad + 
2q_{2}\mathbb{E} 
\left[
\int_{0}^{t \land \widetilde{\tau}^{(2)}_{n}} 
\underbrace{
\left(
\delta + \| \xi^{v_{1}, v_{2}} (s,x)\|^{2} 
\right)^{q_{2}-1} 
\Big\langle 
\xi^{v_{1}, v_{2}} (s,x) ,  
D^{2}F( X^{x}_s)\big( 
\eta^{v_{1}} (s,x), \eta^{v_{2}} (s,x) 
\big) 
\Big\rangle 
}_{=: \mathbb{T}_{1}} \dd s 
\right] \\
& \quad
+ q_{2}(2q_{2}-1) \sum^{m}_{j=1}
\mathbb{E} 
\left[
\int_{0}^{t \land \widetilde{\tau}^{(2)}_{n}}
\underbrace{
\left(
\delta + \| \xi^{v_{1}, v_{2}} (s,x)\|^{2} 
\right)^{q_{2}-1} 
\Big\|
Dg_{j}( X^{x}_s)\xi^{v_{1}, v_{2}} (s,x) 
+ D^{2}g_{j}( X^{x}_s)\big(
\eta^{v_{1}} (s,x), \eta^{v_{2}} (s,x) 
\big)  
\Big\|^{2}
}_{=: \mathbb{T}_{2}} \dd s 
\right]. \\
\end{aligned}
\end{equation}
\end{footnotesize}
The Cauchy-Schwarz inequality and the Young inequality are used several times to indicate that, for two positive constants $\tilde{\epsilon}_{1}, \tilde{\epsilon}_{2}$ with  $\tilde{\epsilon}_{1} \in (0, q_{2}\alpha)$ and $\tilde{\epsilon}_{2} \in (0, (p_{1}-q_{2})/q_{2}]$, 
\begin{equation}
\begin{aligned}
\mathbb{T}_{1} 
\leq 
\tfrac{\tilde{\epsilon}_{1}}{2}
\left(
\delta + \| \xi^{v_{1}, v_{2}} (s,x)\|^{2} 
\right)^{q_{2}}  
+ C_{\tilde{\epsilon}_{1}}  
\left\| 
D^{2}F( X^{x}_s)
\big( 
\eta^{v_{1}} (s,x), \eta^{v_{2}} (s,x) 
\big)  
\right\|^{2q_{2}},  \\
\end{aligned}
\end{equation}
and
\begin{equation}
\begin{aligned}
\mathbb{T}_{2} 
&\leq (1+\tilde{\epsilon}_{2})  
\left(
\delta + \| \xi^{v_{1}, v_{2}} (s,x)\|^{2} 
\right)^{q_{2}-1} 
\big\|
Dg_{j}( X^{x}_s)\xi^{v_{1}, v_{2}} (s,x)  
\big\|^{2} \\
& \quad + C_{\tilde{\epsilon}_{2}} 
\left(
\delta + \| \xi^{v_{1}, v_{2}} (s,x)\|^{2} 
\right)^{q_{2}-1}  
\big\| 
D^{2}g_{j}( X^{x}_s)
\big(
\eta^{v_{1}} (s,x), \eta^{v_{2}} (s,x) 
\big)  
\big\|^{2}  \\
& \leq (1+\tilde{\epsilon}_{2})  
\left(
\delta + \| \xi^{v_{1}, v_{2}} (s,x)\|^{2} 
\right)^{q_{2}-1} 
\big\|
Dg_{j}( X^{x}_s)\xi^{v_{1}, v_{2}} (s,x)  
\big\|^{2} 
+ \tfrac{\tilde{\epsilon}_{1}}{2}
\left(
\delta + \| \xi^{v_{1}, v_{2}} (s,x)\|^{2} 
\right)^{q_{2}} \\
& \quad + C_{\tilde{\epsilon}_{1}, \tilde{\epsilon}_{2}}  
\big\| 
D^{2}g_{j}( X^{x}_s)
\big(
\eta^{v_{1}} (s,x), \eta^{v_{2}} (s,x) 
\big)  
\big\|^{2q_{2}} .
\end{aligned}
\end{equation}
With these estimates above, we obtain that
\begin{equation} 
\begin{aligned}
&\mathbb{E} 
\left[
\left(\delta + \| \xi^{v_{1}, v_{2}} (t \land \widetilde{\tau}^{(2)}_{n},x)\|^{2} \right)^{q_{2}}
\right]\\
&\leq \delta^{q_{2}} 
+ 2q_{2}
\mathbb{E} 
\left[
\int_{0}^{t \land \widetilde{\tau}^{(2)}_{n}} 
\left(
\delta + \| \xi^{v_{1}, v_{2}} (s,x)\|^{2} 
\right)^{q_{2}-1} 
\Big(
\big\langle
\xi^{v_{1}, v_{2}} (s,x) , 
DF( X^{x}_s) \xi^{v_{1}, v_{2}} (s,x) 
\big\rangle 
\Big) \dd s 
 \right] \\
 & \quad 
+  q_{2}(2q_{2}-1)(1+\tilde{\epsilon}_{2})
\sum^{m}_{j=1}
\mathbb{E} 
\left[
\int_{0}^{t \land \widetilde{\tau}^{(2)}_{n}} 
\left(
\delta + \| \xi^{v_{1}, v_{2}} (s,x)\|^{2} 
\right)^{q_{2}-1} 
\big\|
    Dg_{j}( X^{x}_s)\xi^{v_{1}, v_{2}} (s,x)  
\big\|^{2} \dd s 
 \right] \\
& \quad 
+  2\tilde{\epsilon}_{1}
\mathbb{E} 
\left[
\int_{0}^{t \land \widetilde{\tau}^{(2)}_{n}} 
\left(
\delta + \| \xi^{v_{1}, v_{2}} (s,x)\|^{2} 
\right)^{q_{2}}  \dd s 
\right] 
+ C_{\tilde{\epsilon}_{1}} 
\mathbb{E} 
\left[
\int_{0}^{t \land \widetilde{\tau}^{(2)}_{n}} 
\big\| 
D^{2}F( X^{x}_s)
\big( 
\eta^{v_{1}} (s,x), \eta^{v_{2}} (s,x) 
\big)  
\big\|^{2q_{2}}  \dd s 
\right]  \\
&\quad + C_{\tilde{\epsilon}_{1}, \tilde{\epsilon}_{2}} 
\sum^{m}_{j=1}
\mathbb{E} 
\left[
\int_{0}^{t \land \widetilde{\tau}^{(2)}_{n}} 
\big\| 
D^{2}g_{j}( X^{x}_s)\big(
\eta^{v_{1}} (s,x), \eta^{v_{2}} (s,x) 
\big)  
\big\|^{2q_{2}}  \dd s 
\right] . \\
\end{aligned}
\end{equation}
With Assumption \ref{assumption:growth-condition-of-frechet-derivatives-of-drift-and-diffusion}, Lemma \ref{lemma:uniform-moments-bound-of-SDEs}, the H\"older inequality and \eqref{eq:uniform-estimate-of-1st-variation-process} in mind, and recall the definition $\mathcal{P}_{\cdot}(\cdot)$ in \eqref{eqn:Pgammax} its property in \eqref{eqn:D2FPestimate},  we are able to show that, for some positive constants $\rho_{1}, \rho_{2}, \rho_{3}$ satisfying $1/\rho_{1} + 1/\rho_{2} + 1/\rho_{3}= 1$ and \eqref{eq:range-of-p0-in-lemma:differentiability-of-solutions},
\begin{equation} 
\begin{aligned}
&\left\|
D^{2}F( X^{x}_s)\big(
\eta^{v_{1}} (s,x), \eta^{v_{2}} (s,x) 
\big) 
\right\|_{L^{2q_{2}}(\Omega, \mathbb{R}^{d})} \\
&\leq C 
\Big\|
\mathcal{P}_{\gamma-2}(X^{x}_s) \cdot 
\|\eta^{v_{1}} (s,x) \| \cdot 
\|\eta^{v_{2}} (t,x) \|
\Big\|_{L^{2q_{2}}(\Omega, \mathbb{R})} \\
& \leq C 
\left\|
\mathcal{P}_{\gamma-2}(X^{x}_s)
\right\|_{L^{2\rho_{1} q_{2}  }(\Omega, \mathbb{R})} 
\left\| 
\eta^{v_{1}} (s,x) 
\right\|_{L^{2\rho_{2} q_{2}} (\Omega, \mathbb{R}^{d})}
\left\|
\eta^{v_{2}} (s,x) 
\right\|_{L^{2\rho_{3} q_{2}} (\Omega, \mathbb{R}^{d})} \\
& \leq  C e^{-2\alpha_{1} s}
\left\| 
\mathcal{P}_{\gamma-2}(X^{x}_s)
\right\|_{L^{2\rho_{1} q_{2}  }(\Omega, \mathbb{R})} 
\left\|
v_{1}
\right\|_{L^{2\rho_{2} q_{2}} (\Omega, \mathbb{R}^{d})}
\left\| 
v_{2} 
\right\|_{L^{2\rho_{3} q_{2}} (\Omega, \mathbb{R}^{d})}.
\end{aligned}
\end{equation}
Following the same idea and taking into account \eqref{eqn:D2GPestimate}, one can get, for $j\in \{1, \dots,m \}$ and $\rho_{1}q_{2}(\gamma-3) \leq 2p_{0}$,
\begin{equation} 
\begin{aligned}
&\left\|
D^{2}g_{j}( X^{x}_s)\big(
\eta^{v_{1}} (s,x), \eta^{v_{2}} (s,x) 
\big) 
\right\|_{L^{2q_{2}}(\Omega, \mathbb{R}^{d})} \\
& \leq C \left\|
\mathcal{P}_{(\gamma-3)/2}(X^{x}_s) 
\right\|_{L^{2\rho_{1} q_{2} }(\Omega, \mathbb{R})} 
\left\| 
\eta^{v_{1}} (s,x) 
\right\|_{L^{2\rho_{2} q_{2}} (\Omega, \mathbb{R}^{d})}
\left\| 
\eta^{v_{2}} (s,x) 
\right\|_{L^{2\rho_{3} q_{2}} (\Omega, \mathbb{R}^{d})}\\
& \leq  C e^{-2\alpha_{1} s} 
\left\|
\mathcal{P}_{(\gamma-3)/2}(X^{x}_s) 
\right\|_{L^{2\rho_{1} q_{2} }(\Omega, \mathbb{R})} 
\left\|
v_{1} 
\right\|_{L^{2\rho_{2} q_{2}} (\Omega, \mathbb{R}^{d})}
\left\| 
v_{2} 
\right\|_{L^{2\rho_{3} q_{2}} (\Omega, \mathbb{R}^{d})}.
\end{aligned}
\end{equation}
Combining these estimates with \eqref{eq:enhanced-coupled-condition}, \eqref{eq:range-of-p0-in-lemma:differentiability-of-solutions}, Lemma \ref{lemma:uniform-moments-bound-of-SDEs} and  the Young inequality and the monotonicity of $\mathcal{P}_{\cdot}(X_{s}^x)$ yields, 
\begin{small}
\begin{equation} 
\begin{aligned}
&\mathbb{E} 
\left[
\left(
\delta + \| \xi^{v_{1}, v_{2}} (t \land \widetilde{\tau}^{(2)}_{n},x)\|^{2} 
\right)^{q_{2}} 
\right]\\
&\leq \delta^{q_{2}} 
- (2q_{2}\alpha - \tilde{\epsilon}_{1})
\mathbb{E} 
\left[
\int_{0}^{t \land \widetilde{\tau}^{(2)}_{n}} 
\left(
\delta + \| \xi^{v_{1}, v_{2}} (s,x)\|^{2} 
\right)^{q_{2}} 
 \dd s \right]  
+  2q_{2}\alpha \delta
\mathbb{E} 
\left[
\int_{0}^{t \land \widetilde{\tau}^{(2)}_{n}} 
\left(
\delta + \| \xi^{v_{1}, v_{2}} (s,x)\|^{2} 
\right)^{q_{2}-1}  \dd s 
\right]   \\
& \quad
+ C_{\tilde{\epsilon}_{1}, \tilde{\epsilon}_{2}} 
\sup_{r \in [0,T]}  
\left\| 
\mathcal{P}_{\gamma-2}(X^{x}_r) 
\right\|^{2q_{2}}_{L^{2\rho_{1} q_{2} }(\Omega, \mathbb{R})} 
\left\|
v_{1} 
\right\|^{2q_{2}}_{L^{2\rho_{2} q_{2}} (\Omega, \mathbb{R}^{d})}
\left\| 
v_{2} 
\right\|^{2q_{2}}_{L^{2\rho_{3} q_{2}} (\Omega, \mathbb{R}^{d})}  
\int_{0}^{t \land \widetilde{\tau}^{(2)}_{n}} e^{-2q_{2}\alpha_{1} s} \dd s \\
& \leq 
\delta^{q_{2}} 
- \left(
2q_{2}\alpha - 2\tilde{\epsilon_{1}}
\right)
\mathbb{E}
\left[
\int_{0}^{t \land \widetilde{\tau}^{(2)}_{n}} 
\left(
\delta + \| \xi^{v_{1}, v_{2}} (s,x)\|^{2} 
\right)^{q_{2}} 
\dd s 
\right]  
+ C_{\tilde{\epsilon_{2}}}  
\left(
2q_{2}\delta
\right)^{q_{2}}
  \\
& \quad 
+ C_{\tilde{\epsilon_{1}}, \tilde{\epsilon_{2}}}  
\sup_{r \in [0,T]}  
\left\| 
\mathcal{P}_{\gamma-2}(X^{x}_r) 
\right\|^{2q_{2}}_{L^{2\rho_{1} q_{2} }(\Omega, \mathbb{R})} 
\left\| 
v_{1}
\right\|^{2q_{2}}_{L^{2\rho_{2} q_{2}} (\Omega, \mathbb{R}^{d})}
\left\| 
v_{2} 
\right\|^{2q_{2}}_{L^{2\rho_{3} q_{2}} (\Omega, \mathbb{R}^{d})}  
\int_{0}^{t \land \widetilde{\tau}^{(2)}_{n}} 
e^{-2q_{2}\alpha_{1} s} \dd s .
\end{aligned}
\end{equation}
\end{small}
Setting $\delta \rightarrow 0^{+}$,
one observes
\begin{equation}
\begin{aligned}
&\mathbb{E} 
\left[
\left(
\delta + \| \xi^{v_{1}, v_{2}} (t \land \widetilde{\tau}^{(2)}_{n},x_{0})\|^{2} 
\right)^{q_{2}} 
\right]
+ 
\left(
2q_{2}\alpha - 2\tilde{\epsilon_{1}}
\right)
\mathbb{E}
\left[
\int_{0}^{t \land \widetilde{\tau}^{(2)}_{n}} 
\left(
\delta + \| \xi^{v_{1}, v_{2}} (s,x)\|^{2} 
\right)^{q_{2}} 
\dd s 
\right]\\
& \leq
C_{\tilde{\epsilon_{1}}, \tilde{\epsilon_{2}}}  
\sup_{r \in [0,T]}  
\left\| 
\mathcal{P}_{\gamma-2}(X^{x}_r) 
\right\|^{2q_{2}}_{L^{2\rho_{1} q_{2} }(\Omega, \mathbb{R})} 
\left\| 
v_{1}
\right\|^{2q_{2}}_{L^{2\rho_{2} q_{2}} (\Omega, \mathbb{R}^{d})}
\left\| 
v_{2} 
\right\|^{2q_{2}}_{L^{2\rho_{3} q_{2}} (\Omega, \mathbb{R}^{d})}  
\int_{0}^{t \land \widetilde{\tau}^{(2)}_{n}} 
e^{-2q_{2}\alpha_{1} s} \dd s .
\end{aligned}
\end{equation}
By virtue of \eqref{eq:range-of-p0-in-lemma:differentiability-of-solutions}, Lemma \ref{lemma:ito}, Lemma \ref{lemma:uniform-moments-bound-of-SDEs} and the Fatou lemma, one will arrive at, 
\begin{equation} \label{eq:uniform-estimate-of-2nd-variation-process-of-sde}
\begin{aligned}
&\mathbb{E} 
\left[
\| \xi^{v_{1}, v_{2}} (t,x)\|^{2q_{2}} 
\right]\\
&\leq C_{\tilde{\epsilon_{1}}, \tilde{\epsilon_{2}}}    
\sup_{r \in [0,T]} 
\left\|
\mathcal{P}_{\gamma-2}(X^{x}_r) 
\right\|^{2q_{2}}_{L^{2\rho_{1} q_{2} }(\Omega, \mathbb{R})}
\left\| 
v_{1} 
\right\|^{2q_{2}}_{L^{2\rho_{2} q_{2}} (\Omega, \mathbb{R}^{d})}
\left\|
v_{2} 
\right\|^{2q_{2}}_{L^{2\rho_{3} q_{2}} (\Omega, \mathbb{R}^{d})} 
\int_{0}^{t} 
e^{-2(q_{2}\alpha_{1}-\tilde{\epsilon}_{1} )(t-s)} 
e^{-2q_{2}\alpha_{1} s} \dd s \\
& \leq
C_{
\tilde{\epsilon}_{1}, \tilde{\epsilon}_{2}
} 
e^{-\alpha_{2}q_{2}t} 
\sup_{r \in [0,T]}
\left\|
\mathcal{P}_{\gamma-2}(X^{x}_r) 
\right\|^{2q_{2}}_{L^{2\rho_{1} q_{2} }(\Omega, \mathbb{R})}
\left\|
v_{1} 
\right\|^{2q_{2}}_{L^{2\rho_{2} q_{2}} (\Omega, \mathbb{R}^{d})}
\left\|
v_{2} 
\right\|^{2q_{2}}_{L^{2\rho_{3} q_{2}} (\Omega, \mathbb{R}^{d})},
\end{aligned}
\end{equation}
where $\alpha_{2}:= (q_{2}\alpha_{1}-\tilde{\epsilon}_{1})/q_{2}$.

\noindent
\textbf{Part III: estimate of the third variation process} 

For the third variation process of the SDE \eqref{eq:semi-linear-SODE}, we get 
\begin{equation} 
\begin{aligned}
&\dd \zeta^{v_{1}, v_{2}, v_{3}} (t,x)  \\
&= \Big(
DF(X^{x}_{t}) \zeta^{v_{1}, v_{2}, v_{3}} (t,x)  
+ D^{2}F( X^{x}_t)\big(
\eta^{v_{1}} (t,x), \xi^{v_{2}, v_{3}} (t,x) 
\big) 
+  D^{2}F( X^{x}_t)\big(
\xi^{v_{1}, v_{3}} (t,x), \eta^{v_{2}} (t,x) 
\big)\\
& \quad 
+  D^{2}F( X^{x}_t)\big(
\xi^{v_{1}, v_{2}} (t,x), \eta^{v_{3}} (t,x)
\big) 
+ D^{3}F( X^{x}_t)\big(
\eta^{v_{1}} (t,x) , \eta^{v_{2}} (t,x) , \eta^{v_{3}} (t,x)
\big) 
\Big) \dd t \\
& \quad + \sum^{m}_{j=1}
\Big( 
Dg_{j}(X^{x}_{t}) \zeta^{v_{1}, v_{2}, v_{3}} (t,x) 
+ D^{2}g_{j}( X^{x}_t)\big(
\eta^{v_{1}} (t,x), \xi^{v_{2}, v_{3}} (t,x) 
\big)
+  D^{2}g_{j}( X^{x}_t)\big(  
\xi^{v_{1}, v_{3}} (t,x), \eta^{v_{2}} (t,x) 
\big) \\
& \quad 
+  D^{2}g_{j}( X^{x}_t)\big(
\xi^{v_{2}, v_{3}} (t,x), \eta^{v_{1}} (t,x) 
\big) 
+ D^{3}g_{j}( X^{x}_t)\big(   
\eta^{v_{1}} (t,x) , \eta^{v_{2}} (t,x) , \eta^{v_{3}} (t,x)
\big) 
\Big) \dd W_{j,t}\\
&=: \Big( DF(X^{x}_{t}) \zeta^{v_{1}, v_{2}, v_{3}} (t,x) + H(X^{x}_{t}) \Big) \dd t 
+ \sum^{m}_{j=1}
\Big(
Dg_{j}(X^{x}_{t}) \zeta^{v_{1}, v_{2}, v_{3}} (t,x) 
+ G_{j}(X^{x}_{t}) 
\Big) \dd W_{j,t}.
\end{aligned}
\end{equation}
Similarly, 
given the stopping time as below,
\begin{equation}
\widetilde{\tau}^{(3)}_{n} =
\inf \left\{
s \in [0,t] : \|\zeta^{v_{1}, v_{2}, v_{3}} (s,x) \| > n
\quad 
\text{or}
\quad
\left\| 
X^{x}_{s}   
\right\|  > n  
\right\},
\end{equation} 
due to the It\^o formula and the Young inequality, we obtain that, for  positive constants $\tilde{\epsilon}_{3} \in (0, 2p_{1}-2)$ and $\tilde{\epsilon}_{4} \in (0,\alpha)$,
\begin{footnotesize}
\begin{equation} \label{eq:ito-expansion-of-the-3rd-variation-process}
\begin{aligned}
&\mathbb{E} 
\left[
\|\zeta^{v_{1}, v_{2}, v_{3}} 
(t \land \widetilde{\tau}^{(3)}_{n},x) \| ^{2}
\right] \\
&\leq 2 \mathbb{E} 
\left[
\int_{0}^{t \land \widetilde{\tau}^{(3)}_{n}} 
\big\langle \zeta^{v_{1}, v_{2}, v_{3}} (s,x) 
,  DF( X^{x}_s) \zeta^{v_{1}, v_{2}, v_{3}} 
(s,x) \big\rangle   
\dd s
\right] 
+ 2 \mathbb{E} 
\left[
\int_{0}^{t \land \widetilde{\tau}^{(3)}_{n}}
\big\langle 
\zeta^{v_{1}, v_{2}, v_{3}} (s,x) ,  H( X^{x}_s) 
\big\rangle  
\dd s
\right] \\
& \quad 
+ \sum^{m}_{j=1} 
\mathbb{E} 
\left[
\int_{0}^{t \land \widetilde{\tau}^{(3)}_{n}} 
\left\|  
Dg_{j}( X^{x}_s)  \zeta^{v_{1}, v_{2}, v_{3}} (s,x) 
+ G_{j}( X^{x}_s) 
\right\|^{2} \dd s
\right]  \\
& \leq  
\mathbb{E} 
\left[
\int_{0}^{t \land \widetilde{\tau}^{(3)}_{n}}
2\big\langle
\zeta^{v_{1}, v_{2}, v_{3}} (s,x) ,  DF( X^{x}_s) \zeta^{v_{1}, v_{2}, v_{3}} (s,x) 
\big\rangle 
+  (1+\tilde{\epsilon}_{3})  \sum^{m}_{j=1}
\left\|
Dg_{j}( X^{x}_s)  \zeta^{v_{1}, v_{2}, v_{3}} (s,x)  
\right\|^{2}
\dd s
\right] \\
& \quad + \tilde{\epsilon}_{4} 
\mathbb{E} 
\left[
\int_{0}^{t \land \widetilde{\tau}^{(3)}_{n}}
\|\zeta^{v_{1}, v_{2}, v_{3}} (s,x) \| ^{2}
\dd s
\right] 
+ C_{\tilde{\epsilon}_{4}} 
\mathbb{E} 
\left[ 
\int_{0}^{t \land \widetilde{\tau}^{(3)}_{n}}
\|H( X^{x}_s) \| ^{2}
\dd s
\right]
+ C_{\tilde{\epsilon}_{3}} \sum^{m}_{j=1}
\mathbb{E} 
\left[
\int_{0}^{t \land \widetilde{\tau}^{(3)}_{n}}
\|G_{j}( X^{x}_s) \| ^{2}
\dd s
\right] \\
&\leq -(\alpha-\tilde{\epsilon}_{4}) 
\mathbb{E} 
\left[
\int_{0}^{t \land \widetilde{\tau}^{(3)}_{n}}
\|\zeta^{v_{1}, v_{2}, v_{3}} (s,x) \| ^{2}
\dd s
\right] 
+ C_{\tilde{\epsilon}_{4}} 
\mathbb{E} 
\left[
\int_{0}^{t \land \widetilde{\tau}^{(3)}_{n}}
\|H( X^{x}_s) \| ^{2} \dd s
\right]
+ C_{\tilde{\epsilon}_{3}} \sum^{m}_{j=1}
\mathbb{E} 
\left[
\int_{0}^{t \land \widetilde{\tau}^{(3)}_{n}}
\|G_{j}( X^{x}_s) \| ^{2} \dd s
\right]. \\
\end{aligned}
\end{equation}
\end{footnotesize}
The elementary inequality is used to imply that
\begin{equation} \label{eq:remaining-term-H}
\begin{aligned}
&\|H( X^{x}_s) \|_{L^{2}(\Omega, \mathbb{R}^{d})} \\
&\leq 
\left\|
D^{2}F( X^{x}_s)\big(
\eta^{v_{1}} (s,x), \xi^{v_{2}, v_{3}} (s,x) 
\big) 
\right\|_{L^{2}(\Omega, \mathbb{R}^{d})} 
+  \left\|
D^{2}F( X^{x}_s)\big(
\xi^{v_{1}, v_{3}} (s,x), \eta^{v_{2}} (s,x)  
\big)
\right\|_{L^{2}(\Omega, \mathbb{R}^{d})}\\
& \quad + 
\left\|
D^{2}F( X^{x}_s)\big(
\xi^{v_{1}, v_{2}} (s,x), \eta^{v_{3}} (s,x) 
\big)   
\right\|_{L^{2}(\Omega, \mathbb{R}^{d})} 
+ \left\| D^{3}F( X^{x}_s)\big(
\eta^{v_{1}} (s,x), \eta^{v_{2}} (s,x), \eta^{v_{3}} (s,x) 
\big)\right\|_{L^{2}(\Omega, \mathbb{R}^{d})}.
\end{aligned}
\end{equation}
In the following, we first show that the analysis of the first term to the third term on the right hand of \eqref{eq:remaining-term-H} is equivalent. Taking the first term and the second term as examples,  by Assumption \ref{assumption:growth-condition-of-frechet-derivatives-of-drift-and-diffusion}, \eqref{eq:uniform-estimate-of-1st-variation-process}, \eqref{eq:uniform-estimate-of-2nd-variation-process-of-sde}, Lemma \ref{lemma:uniform-moments-bound-of-SDEs} and the H\"older inequality, for some positive constants $\rho_{1}$, $\rho_{2}$, $\rho_{3}$, $\rho_{4}$, $\rho_{5}$ and $\rho_{6}$ with $1/\rho_{1} + 1/\rho_{2} + 1/\rho_{3}= 1$, $1/\rho_{4} + 1/\rho_{5} + 1/\rho_{6}= 1$ and \eqref{eq:range-of-p0-in-lemma:differentiability-of-solutions}, one derives, 
\begin{equation}  \label{eq:first-term-in-remaining-term-H}
\begin{aligned}
&\left\| 
D^{2}F( X^{x}_s)\big(
\eta^{v_{1}} (s,x), \xi^{v_{2}, v_{3}} (s,x) 
\big) 
\right\|_{L^{2}(\Omega, \mathbb{R}^{d})} \\
&\leq C  \Big\|
\mathcal{P}_{\gamma-2}(X^{x}_s)  \cdot 
\|\eta^{v_{1}} (s,x) \| \cdot 
\|\xi^{v_{2}, v_{3}} (s,x)  \| 
\Big\|_{L^{2}(\Omega, \mathbb{R})} \\
& \leq C 
\left\| 
\mathcal{P}_{\gamma-2}(X^{x}_s) 
\right\|_{L^{2\rho_{1}} (\Omega, \mathbb{R})}
\|\eta^{v_{1}} (s,x) \|_{L^{2\rho_{2}} (\Omega, \mathbb{R}^{d})}
\|\xi^{v_{2}, v_{3}} (s,x)  \|_{L^{2\rho_{3}} (\Omega, \mathbb{R}^{d})} \\
& \leq Ce^{-(\alpha_{1}+\alpha_{2}) s} 
\sup_{r \in [0,T]}
\left\|
\mathcal{P}_{\gamma-2}(X^{x}_r) 
\right\|_{L^{2 \max\{ \rho_{1}, \rho_{3}\rho_{4}\}}(\Omega, \mathbb{R})}
\|v_{1} \|_{L^{2\rho_{2}} (\Omega, \mathbb{R}^{d})} 
\|v_{2} \|_{L^{2\rho_{3}\rho_{5}} (\Omega, \mathbb{R}^{d})} 
\|v_{3} \|_{L^{2\rho_{3}\rho_{6}} (\Omega, \mathbb{R}^{d})} .\\
\end{aligned}
\end{equation}
For the second term, we choose another series of constants $\kappa_{1}$, $\kappa_{2}$, $\kappa_{3}$, $\kappa_{4}$, $\kappa_{5}$ and $\kappa_{6}$ with $1/\kappa_{1}+ 1/\kappa_{2} + 1/\kappa_{3} = 1$, $1/\kappa_{4} + 1/\kappa_{5} + 1/\kappa_{6}= 1$ and \eqref{eq:range-of-p0-in-lemma:differentiability-of-solutions} to show, 
\begin{equation}  \label{eq:second-term-in-remaining-term-H}
\begin{aligned}
&\left\|
D^{2}F( X^{x}_s)\big(
\xi^{v_{1}, v_{3}} (s,x), \eta^{v_{2}} (s,x) 
\big) 
\right\|_{L^{2}(\Omega, \mathbb{R}^{d})} \\
&\leq C  
\Big\|
\mathcal{P}_{\gamma-2}(X^{x}_s)  \cdot 
\| \xi^{v_{1}, v_{3}} (s,x) \| \cdot 
\|\eta^{v_{2}} (s,x)  \| 
\Big\|_{L^{2}(\Omega, \mathbb{R})} \\
& \leq C 
\left\| 
\mathcal{P}_{\gamma-2}(X^{x}_s) 
\right\|_{L^{2\kappa_{1}} (\Omega, \mathbb{R})}
\|
\xi^{v_{1}, v_{3}} (s,x) 
\|_{L^{2\kappa_{2}} (\Omega, \mathbb{R}^{d})}
\|
\eta^{v_{2}} (s,x)  
\|_{L^{2\kappa_{3}} (\Omega, \mathbb{R}^{d})} \\
& \leq Ce^{-(\alpha_{1}+\alpha_{2}) s} 
\sup_{r \in [0,T]}
\left\| 
\mathcal{P}_{\gamma-2}(X^{x}_r) 
\right\|_{L^{2 \max\{ \kappa_{1}, \kappa_{2}\kappa_{4}\}}(\Omega, \mathbb{R})}
\|v_{1} \|_{L^{2\kappa_{2}\kappa_{5}} (\Omega, \mathbb{R}^{d})} 
\|v_{2} \|_{L^{2\kappa_{3}} (\Omega, \mathbb{R}^{d})} 
\|v_{3} \|_{L^{2\kappa_{2}\kappa_{6}} (\Omega, \mathbb{R}^{d})} .\\
\end{aligned}
\end{equation}
Then, we take $\kappa_{1}=\rho_{1}$, $\kappa_{2}\kappa_{5}=\rho_{2}$, $\kappa_{3}=\rho_{3}\rho_{5}$, $\kappa_{2}\kappa_{6}=\rho_{3}\rho_{6}$, $\kappa_{2}\kappa_{4}=\rho_{3}\rho_{4}$. It is obvious that
\begin{equation}
\tfrac{1}{\rho_{3}} = 
\tfrac{1}{\kappa_{3}} + \tfrac{1}{\kappa_{2}\kappa_{4}} 
+ \tfrac{1}{\kappa_{2}\kappa_{6}} 
= 1-\tfrac{1}{\rho_{1}} - \tfrac{1}{\rho_{2}}  
= 1- \tfrac{1}{\kappa_{1}}-\tfrac{1}{\kappa_{2}\kappa_{5}},
\end{equation}
which also leads to
\begin{equation}
\tfrac{1}{\kappa_{1}}
+\tfrac{1}{\kappa_{3}} 
+ \tfrac{1}{\kappa_{2}\kappa_{4}} 
+\tfrac{1}{\kappa_{2}\kappa_{5}} 
+\tfrac{1}{\kappa_{2}\kappa_{6}} 
= \tfrac{1}{\kappa_{1}}+\tfrac{1}{\kappa_{3}} 
+ \tfrac{1}{\kappa_{2}}
\left(
\tfrac{1}{\kappa_{4}} + \tfrac{1}{\kappa_{5}} + \tfrac{1}{\kappa_{6}} 
\right)  = 1.
\end{equation}
This implies that there is a one-to-one correspondence between $\rho_{i}$ and $\kappa_{i}$, $i\in\{1,2, \cdots,6 \}$.
About the fourth item in \eqref{eq:remaining-term-H}, we deduce by Assumption \ref{assumption:growth-condition-of-frechet-derivatives-of-drift-and-diffusion} and the H\"older inequality, for some positive constants $\rho'_{1}$, $\rho'_{2}$, $\rho'_{3}$ and $\rho'_{4}$ with  $1/\rho'_{1} + 1/\rho'_{2}  + 1/\rho'_{3} + 1/\rho'_{4} = 1$,
\begin{equation} 
\begin{aligned}
&\left\| 
D^{3}F( X^{x}_s)\big(
\eta^{v_{1}} (s,x), \eta^{v_{2}} (s,x), \eta^{v_{3}} (s,x) 
\big)
\right\|_{L^{2}(\Omega, \mathbb{R}^{d})} \\
& \leq C  
\Big\|  
\mathcal{P}_{\gamma-3}(X^{x}_s)  \cdot 
\|\eta^{v_{1}} (s,x) \| \cdot 
\|\eta^{v_{2}} (s,x) \| \cdot 
\|\eta^{v_{3}} (s,x) \|  
\Big\|_{L^{2}(\Omega, \mathbb{R})} \\
& \leq C 
\left\| 
\mathcal{P}_{\gamma-3}(X^{x}_s) 
\right\|_{L^{2\rho'_{1}} (\Omega, \mathbb{R})}
\|
\eta^{v_{1}} (s,x) 
\|_{L^{2\rho'_{2}} (\Omega, \mathbb{R}^{d})}
\|
\eta^{v_{2}} (s,x) 
\|_{L^{2\rho'_{3}} (\Omega, \mathbb{R}^{d})}
\|
\eta^{v_{3}} (s,x) 
\|_{L^{2\rho'_{4}} (\Omega, \mathbb{R}^{d})}
\\
&\leq Ce^{-3\alpha_{1} s} 
\sup_{r \in [0,T]}
\left\|
\mathcal{P}_{\gamma-3}(X^{x}_r) 
\right\|_{L^{ 2\rho'_{1} }(\Omega, \mathbb{R})}  
\|v_{1} \|_{L^{2\rho'_{2}} (\Omega, \mathbb{R}^{d})} 
\|v_{2} \|_{L^{2\rho'_{3}} (\Omega, \mathbb{R}^{d})} 
\|v_{3} \|_{L^{2\rho'_{4}} (\Omega, \mathbb{R}^{d})} .\\
\end{aligned}
\end{equation}
Assuming, for example, $\rho'_{1} = \rho_{1}$ and $\rho'_{2} = \rho_{2}$, it is obvious for us to choose $\rho'_{3}$ and $\rho'_{4}$ satisfying $\rho_{3}\rho_{5} \geq \rho'_{3} $ and $\rho_{3}\rho_{6} \geq \rho'_{4} $. That is to say, the 
fourth term in the right hand side of \eqref{eq:remaining-term-H} can be controlled by the first three terms. 
Hence, by \eqref{eq:range-of-p0-in-lemma:differentiability-of-solutions}, we get, 
\begin{equation}
\begin{aligned}
&\max\left\{
\|
H( X^{x}_s) 
\|_{L^{2}(\Omega, \mathbb{R}^{d})}, \
\|
G_{j}( X^{x}_s) 
\|_{L^{2}(\Omega, \mathbb{R}^{d})} 
\right\} \\
&\leq Ce^{- \min\{\alpha_{1}+\alpha_{2} , 3\alpha_{1}\} s} 
\sup_{r \in [0,T]}
\left\| 
\mathcal{P}_{\gamma-2}(X^{x}_r) 
\right\|_{L^{2\max\{ \rho_{1}, \rho_{3}\rho_{4} \} }(\Omega, \mathbb{R})} 
\|v_{1} \|_{L^{2\rho_{2}} (\Omega, \mathbb{R}^{d})} 
\|v_{2} \|_{L^{2\rho_{3}\rho_{5}} (\Omega, \mathbb{R}^{d})} 
\|v_{3} \|_{L^{2\rho_{3}\rho_{6}} (\Omega, \mathbb{R}^{d})},
\end{aligned}
\end{equation}
where the analysis of $\|
G_{j}( X^{x}_s) 
\|_{L^{2}(\Omega, \mathbb{R}^{d})}$, $j\in \{1,\dots,m \}$,  is virtually identical
to the estimate of $\|
H( X^{x}_s) 
\|_{L^{2}(\Omega, \mathbb{R}^{d})}$ so that we omit it here.
Plugging these estimates with \eqref{eq:enhanced-coupled-condition} into \eqref{eq:ito-expansion-of-the-3rd-variation-process} yields, 
\begin{equation}
\begin{aligned}
&\mathbb{E} 
\left[
\|\zeta^{v_{1}, v_{2}, v_{3}} (
t \land \widetilde{\tau}^{(3)}_{n},x
) \| ^{2}
\right]
+(\alpha-\tilde{\epsilon}_{4})
\mathbb{E} 
\left[
\int_{0}^{t \land \widetilde{\tau}^{(3)}_{n}}
\|\zeta^{v_{1}, v_{2}, v_{3}} (s,x) \| ^{2}
\dd s
\right]
\\
&\leq  C_{\tilde{\epsilon}_{3}, \tilde{\epsilon}_{4}} 
\int_{0}^{t \land \widetilde{\tau}^{(3)}_{n}}
e^{-2 \min\{ \alpha_{1}+\alpha_{2} , 3\alpha_{1}\} s} 
\dd s \ \times \\
& \hspace{6em}
\sup_{r \in [0,T]}
\left\|
\mathcal{P}_{\gamma-2}(X^{x}_r) 
\right\|_{L^{2 \max\{ \rho_{1}, \rho_{3}\rho_{4} \}}(\Omega, \mathbb{R})}
\|v_{1} \|_{L^{2\rho_{2}} (\Omega, \mathbb{R}^{d})} 
\|v_{2} \|_{L^{2\rho_{3}\rho_{5}} (\Omega, \mathbb{R}^{d})} 
\|v_{3} \|_{L^{2\rho_{3}\rho_{6}} (\Omega, \mathbb{R}^{d})}.
\end{aligned}
\end{equation}
As a direct consequence of the Fatou lemma, Lemma \ref{lemma:ito}, Lemma \ref{lemma:uniform-moments-bound-of-SDEs} and \eqref{eq:range-of-p0-in-lemma:differentiability-of-solutions}, we have with $\alpha_{3} := \alpha - \tilde{\epsilon}_{4}$,
\begin{equation}
\begin{aligned}
&\|\zeta^{v_{1}, v_{2}, v_{3}} (t,x) \|_{L^{2} (\Omega, \mathbb{R}^{d})}\\
&\leq 
C_{\tilde{\epsilon}_{3}, \tilde{\epsilon}_{4}} 
e^{-\alpha_{3} t} 
\sup_{r \in [0,T]}
\left\| 
\mathcal{P}_{\gamma-2}(X^{x}_r) 
\right\|_{L^{2\max\{ \rho_{1}, \rho_{3}\rho_{4} \} }(\Omega, \mathbb{R})}
\|v_{1} \|_{L^{2\rho_{2}} (\Omega, \mathbb{R}^{d})} 
\|v_{2} \|_{L^{2\rho_{3}\rho_{5}} (\Omega, \mathbb{R}^{d})} 
\|v_{3} \|_{L^{2\rho_{3}\rho_{6}} (\Omega, \mathbb{R}^{d})}.
\end{aligned}
\end{equation}
The proof is completed.
\end{proof}
\subsection{Proof of Lemma \ref{lemma:estimate-of-u-and-its-derivatives}}
\label{proof-of-lemma:estimate-of-u-and-its-derivatives}
\begin{proof}[Proof of Lemma \ref{lemma:estimate-of-u-and-its-derivatives}]
As we know,
the first-order derivatives of $u(t,x)$ is 
\begin{equation}
\begin{aligned}
Du(t,x)   v_{1} = 
\mathbb{E} 
\left[ 
D\varphi \left( X^{x}_{t}\right)
\eta^{v_{1}}(t,x) 
\right].
\end{aligned}
\end{equation}
Hence, for $\varphi \in C^{3}_{b}(\mathbb{R}^{d})$, we obtain from Lemma \ref{lemma:uniform-moments-bound-of-SDEs}, Lemma \ref{lemma:differentiability-of-solutions} and the H\"older inequality that
\begin{equation}
\begin{aligned}
\left\| Du(t,x)   v_{1} 
\right\|_{L^{1}(\Omega, \mathbb{R})}  
\leq 
\| \varphi\|_{1} \cdot 
\left\| 
\eta^{v_{1}}(t,x) 
\right\|_{L^{2}(\Omega, \mathbb{R}^{d})}  
\leq Ce^{-\alpha_{1}t}  
\left\|  
v_{1}
\right\|_{L^{2}(\Omega, \mathbb{R}^{d})}.
\end{aligned}
\end{equation}
And the second-order derivatives of $u(t,x)$ goes to
\begin{equation}
\begin{aligned}
D^{2}u(t,x)   (v_{1}, v_{2}) 
= \mathbb{E} 
\left[
D\varphi \left( X^{x}_{t}\right)
\xi^{v_{1}, v_{2}}(t,x) 
\right] 
+ \mathbb{E} 
\left[
D^{2}\varphi \left( X^{x}_{t}\right) 
\big(\eta^{v_{1}}(t,x), \eta^{v_{2}}(t,x) \big)
\right].
\end{aligned}
\end{equation}
In a similar way, by \eqref{eq:range-of-p0-lemma:estimate-of-u-and-its-derivatives}, we get, 
\begin{equation}
\begin{aligned}
\left\| 
D^{2}u(t,x)   (v_{1}, v_{2}) 
\right\|_{L^{1}(\Omega, \mathbb{R})}  
&\leq
\| \varphi\|_{1} \cdot   
\left\|
\xi^{v_{1}, v_{2}}(t,x) 
\right\|_{L^{2}(\Omega, \mathbb{R}^{d})} 
+ \|\varphi\|_{2} \cdot  
\left\|
\eta^{v_{1}}(t,x) 
\right\|_{L^{2}(\Omega, \mathbb{R}^{d})}
\left\| 
\eta^{v_{2}}(t,x) 
\right\|_{L^{2}(\Omega, \mathbb{R}^{d})} \\
& \leq Ce^{-\tilde{\alpha}_{2}t} 
\left\| 
\mathcal{P}_{\gamma-2}(X^{x}_{r}) 
\right\|_{L^{2\rho_{1} }(\Omega, \mathbb{R})}
\left\| 
v_{1} 
\right\|_{L^{2\rho_{2} } (\Omega, \mathbb{R}^{d})}
\left\|
v_{2} 
\right\|_{L^{2\rho_{3} } (\Omega, \mathbb{R}^{d})},
\end{aligned}
\end{equation}
where $\tilde{\alpha}_{2}:= \min\{2\alpha_{1},\alpha_{2} \}$.
In addition, it follows
\begin{equation}\label{equation:expansion-of-u-3-times}
\begin{aligned}
& D^{3}u(t,x)   (v_{1}, v_{2}, v_{3}) \\
&= \mathbb{E} 
\left[
D\varphi \left( X^{x}_{t}\right)   
\zeta^{v_{1}, v_{2}, v_{3}}(t,x) 
\right] 
+ \mathbb{E} 
\left[
D^{2}\varphi \left( X^{x}_{t}\right)  
\big(
\xi^{v_{1}, v_{2}}(t,x), \eta^{v_{3}}(t,x) 
\big) 
\right] \\
& \quad + 
\mathbb{E} 
\left[
D^{2}\varphi \left( X^{x}_{t}\right)   
\big(
\xi^{v_{1}, v_{3}}(t), 
\eta^{v_{2}}(t,x) 
\big) 
\right] 
+ \mathbb{E} 
\left[
D^{2}\varphi \left( X^{x}_{t}\right)  
\big(
\eta^{v_{1}}(t,x),  
\xi^{v_{2}, v_{3}}(t,x) 
\big) 
\right] \\
&\quad +
\mathbb{E} 
\left[
D^{3}\varphi \left( X^{x}_{t}\right)  
\big(
\eta^{v_{1}}(t,x),  \eta^{v_{2}}(t,x), \eta^{v_{3}}(t,x)
\big) 
\right].\\
\end{aligned}
\end{equation}
As shown in the proof of Lemma \ref{lemma:differentiability-of-solutions}, the analysis from the second term to the fourth term in \eqref{equation:expansion-of-u-3-times} is equivalent and the estimate of the last term in the right hand side of \eqref{equation:expansion-of-u-3-times} can be bounded by the other terms. With Lemma \ref{lemma:uniform-moments-bound-of-SDEs}, Lemma \ref{lemma:differentiability-of-solutions} and the H\"older inequality, we get, for $\bar{\rho}_{1}, \bar{\rho}_{2}, \bar{\rho}_{3}, \rho_{1}, \rho_{2}, \rho_{3}, \rho_{4}, \rho_{5}, \rho_{6} >1$ satisfying
$1/\bar{\rho}_{1} + 1/\bar{\rho}_{2} + 1/\bar{\rho}_{3}=1$,
$1/\rho_{1} + 1/\rho_{2} + 1/\rho_{3}=1$,
$1/\rho_{4} + 1/\rho_{5} + 1/\rho_{6}=1$ and 
\eqref{eq:range-of-p0-lemma:estimate-of-u-and-its-derivatives},
\begin{equation}\label{equation:prior-estimate-of-D3u}
\begin{aligned}
&\left\|
D^{3}u(t,x)   (v_{1}, v_{2}, v_{3}) 
\right\|_{L^{1}(\Omega, \mathbb{R})} \\
& \leq 
\| \varphi\|_{1} \cdot 
\left\|  
\zeta^{v_{1}, v_{2}, v_{3}}(t,x)  
\right\|_{L^{2}(\Omega, \mathbb{R}^{d})} 
+ \|\varphi\|_{2} \cdot 
\left\| 
\xi^{v_{1}, v_{2}}(t,x) 
\right\|_{L^{2}(\Omega, \mathbb{R}^{d})} 
\left\| 
\eta^{v_{3}}(t,x) 
\right\|_{L^{2}(\Omega, \mathbb{R}^{d})} \\
& \quad + \|\varphi\|_{2} \cdot 
\left\| 
\xi^{v_{1}, v_{3}}(t,x) 
\right\|_{L^{2}(\Omega, \mathbb{R}^{d})} 
\left\| 
\eta^{v_{2}}(t,x) 
\right\|_{L^{2}(\Omega, \mathbb{R}^{d})} 
+ \|\varphi\|_{2} \cdot  
\left\| 
\xi^{v_{2}, v_{3}}(t,x) 
\right\|_{L^{2}(\Omega, \mathbb{R}^{d})} 
\left\| 
\eta^{v_{1}}(t,x) 
\right\|_{L^{2}(\Omega, \mathbb{R}^{d})}\\
&\quad + \|\varphi\|_{3} \cdot  
\left\| 
\eta^{v_{1}}(t,x) 
\right\|_{L^{2\bar{\rho}_{1}}(\Omega, \mathbb{R}^{d})}
\left\| 
\eta^{v_{2}}(t,x) 
\right\|_{L^{2\bar{\rho}_{2}}(\Omega, \mathbb{R}^{d})} 
\left\| 
\eta^{v_{3}}(t,x) 
\right\|_{L^{2\bar{\rho}_{3}}(\Omega, \mathbb{R}^{d})}.\\
\end{aligned}
\end{equation}
If we take
$\bar{\rho}_{2} = \rho_{3}\rho_{5}$, $\bar{\rho}_{3} = \rho_{3}\rho_{6}$, then
\begin{equation}
\begin{aligned}
\tfrac{1}{\bar{\rho}_{1}}
= 1- \tfrac{1}{\rho_{3}\rho_{5}} - \tfrac{1}{\rho_{3}\rho_{6}}
=1-\tfrac{1}{\rho_{3}}(1-\tfrac{1}{\rho_{4}})
=\tfrac{1}{\rho_{1}} + \tfrac{1}{\rho_{2}} + \tfrac{1}{\rho_{3} \rho_{4}} \geq \tfrac{1}{\rho_{2}},
\end{aligned}
\end{equation}
leading to $\bar{\rho}_{1} \leq \rho_{2}$.
Combining this with \eqref{equation:prior-estimate-of-D3u} and Lemma \ref{lemma:differentiability-of-solutions} yields
\begin{equation}
\begin{aligned}
&\left\|
D^{3}u(t,x)   (v_{1}, v_{2}, v_{3}) 
\right\|_{L^{1}(\Omega, \mathbb{R})} \\
&\leq C e^{-\tilde{\alpha}_{3}t} 
\sup_{r \in [0,T]}
\left\|
\mathcal{P}_{\gamma-2}(X^{x}_{r}) 
\right\|_{L^{2\max\{ \rho_{1}, \rho_{3}\rho_{4} \} }(\Omega, \mathbb{R})} 
\|v_{1} \|_{L^{2\rho_{2}} (\Omega, \mathbb{R}^{d})} 
\|v_{2} \|_{L^{2\rho_{3}\rho_{5}} (\Omega, \mathbb{R}^{d})} 
\|v_{3} \|_{L^{2\rho_{3}\rho_{6}} (\Omega, \mathbb{R}^{d})},
\end{aligned}
\end{equation}
where  $\tilde{\alpha}_{3}:= \min\{3\alpha_{1}, \alpha_{1}+\alpha_{2}, \alpha_{3} \}$.
 The proof is completed.
\end{proof}
\subsection{Proof of Lemma \ref{lemma:holder-continuity-of-the-process-z}}
\label{proof-of-lemma:holder-continuity-of-the-process-z}
\begin{proof}[Proof of Lemma \ref{lemma:holder-continuity-of-the-process-z}]
In view of \eqref{equation:power-property-of-Brownian-motion}, \eqref{intro:continuous-version-of-the-numerical-scheme},  and the triangle inequality,
we obtain that, for $t\in  [t_{n}, t_{n+1}]$, $n\in \{0,1,\dots,N-1 \}$, $N\in \mathbb{N}$,
\begin{equation}
\begin{aligned}
\|\mathbb{Z}^{n}(t) \|_{L^{2p}(\Omega, \mathbb{R}^{d})} 
&\leq 
\|\mathbb{Z}^{n}(t_{n}) \|_{L^{2p}(\Omega, \mathbb{R}^{d})}
+ (t-t_{n}) \|F(\mathscr{P}(Y_{n})) \|_{L^{2p}(\Omega, \mathbb{R}^{d})} \\
&\quad
+ (t-t_{n})^{\frac{1}{2}} \|g(\mathscr{P}(Y_{n})) \|_{L^{2p}(\Omega, \mathbb{R}^{d\times m})}. \\
\end{aligned}
\end{equation}
According to Lemma \ref{lemma:necessary-estimates}, it suffices to show that
\begin{equation}
(t-t_{n}) 
\|
F(\mathscr{P}(Y_{n})) 
\|_{L^{2p}(\Omega, \mathbb{R}^{d})} 
\leq C_{A}h^{\frac{1}{2}},
\end{equation}
and
\begin{equation}
(t-t_{n})^{\frac{1}{2}} \|g(\mathscr{P}(Y_{n})) \|_{L^{2p}(\Omega, \mathbb{R}^{\times m})} \leq  Ch^{\frac{1}{4}}   (1 + \|\mathscr{P}(Y_{n}) \|_{L^{2p}(\Omega, \mathbb{R}^{d})}).
\end{equation}
Hence, one obtains from Lemma \ref{lemma:uniform-moments-bound-of-the-LTPE-method} that
\begin{equation}
\|
\mathbb{Z}^{n}(t) 
\|_{L^{2p}(\Omega, \mathbb{R}^{d})} 
\leq  C_{A} (1 + \|x_{0} \|_{L^{2p}(\Omega, \mathbb{R}^{d})}).
\end{equation}
For the estimate of \eqref{equation:holder-continuty-of-the-auxiliary-process-Z}, the proof is obvious due to \eqref{equation:moments-bound-of-the-auxiliary-process-Z} and  Assumption \ref{assumption:growth-condition-of-frechet-derivatives-of-drift-and-diffusion}  with $p \in [1,p_{0}/\gamma]$. The proof is completed.
\end{proof}

\subsection{Proof of Lemma \ref{lemma:error-estimate-between-x-and-projected-x}}
\label{subsection:proof-of-lemma:error-estimate-between-x-and-projected-x}
\begin{proof}[Proof of Lemma \ref{lemma:error-estimate-between-x-and-projected-x}]
Consider these two measurable sets
\begin{equation}
\mathcal{A}_{h} := 
\left\{
\omega \in \Omega : \|\zeta(\omega) \| 
\leq h^{-\frac{1}{2\gamma}} 
\right\}, \quad 
\mathcal{A}_{h}^{c} := \Omega \backslash \mathcal{A}_{h} .
\end{equation}
Therefore, owing to the H\"older inequality, for $1/q+1/q'=1$, we obtain
\begin{equation}
\begin{aligned}
\mathbb{E}
\left[
\|\zeta-\mathscr{P}(\zeta) \|^{2}
\right] 
= \mathbb{E}
\left[
\|\zeta-\mathscr{P}(\zeta) \|^{2} \textbf{1}_{\mathcal{A}_{h}^{c}}
\right] 
\leq
\|
\zeta-\mathscr{P}(\zeta) 
\|^{2}_{L^{2q} (\Omega, \mathbb{R}^{d})}
\|
\textbf{1}_{\mathcal{A}_{h}^{c}} 
\|_{L^{q'} (\Omega, \mathbb{R})}.
\end{aligned}
\end{equation}
Here, using Lemma \ref{lemma:necessary-estimates} with the triangular inequality yields
\begin{equation}
\begin{aligned}
\|
\zeta-\mathscr{P}(\zeta) 
\|^{2}_{L^{2q} (\Omega, \mathbb{R}^{d})} 
\leq
\|\zeta \|^{2}_{L^{2q} (\Omega, \mathbb{R}^{d})} 
+ \|
\mathscr{P}(\zeta) 
\|^{2}_{L^{2q} (\Omega, \mathbb{R}^{d})} 
\leq 2\|\zeta \|^{2}_{L^{2q} (\Omega, \mathbb{R}^{d})}.
\end{aligned}
\end{equation}
In addition, it follows from the Markov inequality that,
\begin{equation}
\begin{aligned}
\|
\textbf{1}_{\mathcal{A}_{h}^{c}} 
\|_{L^{q'} (\Omega, \mathbb{R})} 
=\big(
\mathbb{P}(\mathcal{A}_{h}^{c})
\big)^{\frac{1}{q'}} 
\leq h^{\frac{\beta}{2\gamma q'}}
\|\zeta \|^{\frac{\beta}{q'}}_{L^{\beta} (\Omega, \mathbb{R}^{d})}.
\end{aligned}
\end{equation}
We choose $q=4\gamma+1$, $q'=1+1/4\gamma$ and $\beta=8\gamma+2$, then the proof is completed.

\end{proof}

\end{document}